\documentclass[12pt,a4paper,leqno]{article}
\usepackage{amsmath,amsfonts,amssymb,amscd}
\usepackage[margin=2cm]{geometry}

\usepackage{theorem}
\usepackage{url}
\usepackage{hyperref}
\usepackage{tikz-cd}
\usetikzlibrary{cd}

\theoremstyle{change}
\theorembodyfont{\sl}
\newtheorem{Theorem}[subsection]{Theorem}
\newtheorem{Proposition}[subsection]{Proposition}

\newtheorem{subProposition}[subsubsection]{Proposition}
\newtheorem{subLemma}[subsubsection]{Lemma}

\theorembodyfont{\rmfamily}

\newtheorem{Remark}[subsection]{Remark}

\newtheorem{subRemark}[subsubsection]{Remark}

\makeatletter
\let\c@equation=\c@subsubsection

\makeatother

\makeatletter
\newenvironment{eqn}{\refstepcounter{subsection}
$$}{\leqno{\rm(\thesubsection)}$$\global\@ignoretrue}
\makeatother  

\makeatletter 
\newenvironment{subeqn}{\refstepcounter{subsubsection}
$$}{\leqno{\rm(\thesubsubsection)}$$\global\@ignoretrue}
\makeatother

\newenvironment{prf}[1]{\trivlist
\item[\hskip \labelsep{\it
#1\hspace*{.3em}}]}{~\hspace{\fill}~$\square$\endtrivlist}

\newenvironment{proof}{\begin{prf}{\bf Proof}}{\end{prf}}

\newcommand{\ol}{\overline}
\newcommand{\pr}{\mathrm{pr}}
\newcommand{\C}{\mathbb C}

\newcommand{\Q}{\mathbb Q}
\newcommand{\Z}{\mathbb Z}
\newcommand{\F}{\mathbb F}
\newcommand{\Fbar}{{\ol{\mathbb F}}}
\newcommand{\Qbar}{{\ol{\mathbb Q}}}
\newcommand{\N}{\mathbb N}
\newcommand{\G}{\mathbb G}
\newcommand{\A}{\mathbb A}
\newcommand{\Zpp}{{\mathbb Z/p^2}}
\newcommand{\Gm}{\G_{\mathrm{m}}}
\newcommand{\Ga}{\G_{\mathrm{a}}}
\renewcommand{\P}{\mathbb P}
\newcommand{\Pic}{\mathrm{Pic}}
\newcommand{\calO}{\mathcal{O}}
\newcommand{\calE}{\mathcal{E}}
\newcommand{\calL}{\mathcal{L}}
\newcommand{\calM}{\mathcal{M}}
\newcommand{\calN}{\mathcal{N}}
\newcommand{\calQ}{\mathcal{Q}}
\newcommand{\divisor}{\mathrm{div}}
\newcommand{\quot}{\mathrm{quot}}
\newcommand{\Isom}{\mathrm{Isom}}
\newcommand{\bfIsom}{\mathbf{Isom}}
\newcommand{\Norm}{\mathrm{Norm}}
\newcommand{\Hom}{\mathrm{Hom}}
\newcommand{\GL}{\mathrm{GL}}
 \newcommand{\id}{\mathrm{id}}
\newcommand{\lto}{\longrightarrow}
\newcommand{\diag}{\mathrm{diag}}
\newcommand{\rmH}{\mathrm{H}}
\newcommand{\rmM}{\mathrm{M}}
\newcommand{\tr}{\mathrm{tr}}
\newcommand{\rank}{\mathrm{rank}}
\renewcommand{\phi}{\varphi}
\newcommand{\Spec}{\mathrm{Spec}}

\newcommand{\sm}{\mathrm{sm}}
\newcommand{\NS}{\mathrm{NS}}

\newcommand{\de}{\mathrm{d}}
\newcommand{\red}{\mathrm{red}}
\newcommand{\rig}{\mathrm{rig}}
\newcommand{\univ}{\mathrm{univ}}
\newcommand{\Gal}{\mathrm{Gal}}
\newcommand{\unr}{\mathrm{unr}}
\newcommand{\multideg}{\mathrm{mdeg}}
\newcommand{\eps}{\varepsilon}
\newcommand{\wt}{\widetilde}
\newcommand{\into}{\hookrightarrow}
\newcommand{\onto}{\twoheadrightarrow}
\newcommand{\Sym}{\mathrm{Sym}}

\begin{document}

\title{Geometric quadratic Chabauty}
\author{Bas  Edixhoven  \&  Guido Lido \\
\small{
  \href{mailto:edix@math.leidenuniv.nl}{edix@math.leidenuniv.nl}
\quad
\href{mailto:guidomaria.lido@gmail.com}{guidomaria.lido@gmail.com}
\quad \small{Math. Inst., Universiteit Leiden}
}}

\maketitle

\begin{abstract}
  Since Faltings proved Mordell's conjecture (1983, \cite{Fal}) we know that the
  sets of rational points on curves of genus at least 2 are finite.
  Determining these sets, in individual cases, is still an unsolved
  problem. Chabau\-ty's method (1941, \cite{Chab}) is to intersect, for a prime
  number~$p$, in the $p$-adic Lie group of $p$-adic points of the
  jacobian, the closure of the Mordell-Weil group with the $p$-adic
  points of the curve. Under the condition that the Mordell-Weil rank
  is less than the genus, Chabauty's method, in combination with other
  methods such as the Mordell-Weil sieve, has been applied
  successfully to determine all rational points in many cases.

  Minhyong Kim's non-abelian Chabauty programme aims to remove the
  condition on the rank. The simplest case, called quadratic Chabauty,
  was developed by Balakrishnan, Besser, Dogra, M\"uller, Tuitman and
  Vonk, and applied in a tour de force to the so-called cursed curve
  (rank and genus both~3).

  This article aims to make the quadratic Chabauty method \emph{small}
  and \emph{geometric} again, by describing it in terms of only
  `simple algebraic geometry' (line bundles over the jacobian and
  models over the integers). 

  \noindent Keywords: rational points, curves, Chabauty, non-abelian,
  quadratic, geometric.\\
  \noindent 2010 Mathematical Subject Classification: 14G05, 11G30.
\end{abstract}

\setcounter{tocdepth}{2}
\tableofcontents

\section{Introduction}
Faltings proved in 1983, see~\cite{Fal}, that for every number field
$K$ and every curve $C$ over $K$ of genus at least~$2$, the set of
$K$-rational points $C(K)$ is finite.  However, determining $C(K)$, in
individual cases, is still an unsolved problem. For simplicity, we
restrict ourselves in this article to the case $K=\Q$.

Chabauty's method (1941) for determining $C(\Q)$ is to intersect, for
a prime number~$p$, in the $p$-adic Lie group of $p$-adic points of
the jacobian, the closure of the Mordell-Weil group with the $p$-adic
points of the curve. There is a fair amount of evidence (mainly
hyperelliptic curves of small genus, see~\cite{BBCCE}) that Chabauty's
method, in combination with other methods such as the Mordell-Weil
sieve, does determine all rational points when $r<g$, with $r$ the
Mordell-Weil rank and $g$ the genus of~$C$.

For a general introduction to Chabauty's method and Coleman's
effective version of it, we highly recommend~\cite{M-P}, and, for an
implementation of it that is `geometric' in the sense of this article,
to~\cite{Flynn}, in which equations for the curve embedded in the
Jacobian are pulled back via local parametrisations of the closure of
the Mordell-Weil group.

Minhyong Kim's non-abelian Chabauty programme aims to remove the
condition that $r<g$. The `non-abelian' refers to fundamental groups;
the fundamental group of the jacobian of a curve is the abelianised
fundamental group of the curve. The most striking result in this
direction is the so-called quadratic Chabauty method, applied
in~\cite{BDMTV}, a technical tour de force, to the so-called cursed
curve ($r=g=3$). For more details we recommend the introduction
of~\cite{BDMTV}. 

This article aims to make the quadratic Chabauty method \emph{small}
and \emph{geometric} again, by describing it in terms of only `simple
algebraic geometry' (line bundles over the jacobian, models over the
integers, and biextension structures). The main result is
Theorem~\ref{thm:finiteness}. It gives a criterion for a given list of
rational points to be complete, in terms of points with values
in~$\Z/p^2\Z$ only. Section~2 describes the geometric method in less
than 3 pages, Sections 3--5 give the necessary theory, Sections 6--7
give descriptions that are suitable for computer calculations, and
Section~\ref{sec:example} treats an example with $r=g=2$ and $14$
rational points. As explained in the remarks following
Theorem~\ref{thm:finiteness}, we expect that this approach will make
it possible to treat many more curves. Section~\ref{sec:fund_grps}
gives some remarks on the fundamental groups of the objects we
use. They are subgroups of higher dimensional Heisenberg groups, where
the commutator pairing is the intersection pairing of the first
homology group of the curve.  Section~\ref{sec:finiteness} reproves the
finiteness of $C(\Q)$, for $C$ with $r<g+\rho-1$, with $\rho$ the rank
of the $\Z$-module of symmetric endomorphisms of the jacobian
of~$C$. It also shows that a version of Theorem~\ref{thm:finiteness}
that uses higher $p$-adic precision will always give a finite upper
bound for~$C(\Q)$. Section~\ref{sec:p-adic_heights} gives, through an
appropriate choice of coordinates that split the Poincar\'e
biextension, the relation between our geometric approach and the $p$-adic
heights used in the cohomological approach. 

Already for the case of classical Chabauty (working with $J$ instead
of~$T$, and under the assumption that $r<g$), where everything is
linear, the criterion of~Theorem~\ref{thm:finiteness} can be useful;
this has been worked out and implemented in~\cite{Spelier}. We
recommend this work as a gentle introduction into the geometric
approach taken in this article. A generalisation from $\Q$ to number
fields is given in~\cite{CLXY}. For a generalisation of the
cohomological approach, see~\cite{BBBM} (quadratic Chabauty)
and~\cite{Dogra} (non-abelian Chabauty).

Although this article is about geometry, it contains no
pictures. Fortunately, many pictures can be found in~\cite{Ha}, and
some in~\cite{Ed_AWS}.

\section{Algebraic geometry}\label{sec:alg_geometry}

Let $C$ be a scheme over~$\Z$, proper, flat, regular, with $C_\Q$ of
dimension one and geometrically connected. Let $n$ be in $\Z_{\geq1}$
such that the restriction of $C$ to $\Z[1/n]$ is smooth. Let $g$ be
the genus of~$C_\Q$. We assume that $g\geq2$ and that we have a
rational point $b\in C(\Q)$; it extends uniquely to a $b\in C(\Z)$. We
let $J$ be the N\'eron model over $\Z$ of the jacobian
$\Pic^0_{C_\Q/\Q}$. We denote by $J^\vee$ the N\'eron model over $\Z$
of the dual $J_\Q^\vee$ of~$J_\Q$, and $\lambda\colon J\to J^\vee$ the
isomorphism extending the canonical principal polarisation
of~$J_\Q$. We let $P_\Q$ be the Poincar\'e \emph{line bundle} on
$J_\Q\times J_\Q^\vee$, trivialised on the union of $\{0\}\times
J_\Q^\vee$ and $J_\Q\times\{0\}$. Then the Poincar\'e \emph{torsor} is
the $\Gm$-torsor on $J_\Q\times J_\Q^\vee$ defined as
\begin{eqn}\label{eq:def_P}
P^\times_\Q = \bfIsom_{J_\Q\times J_\Q^\vee}(\calO_{J_\Q\times J_\Q^\vee},P_\Q)\,.  
\end{eqn}
For every scheme $S$ over $J_\Q\times J_\Q^\vee$, $P^\times_\Q(S)$ is the set
of isomorphisms from $\calO_S$ to $(P_\Q)_S$, with a free and
transitive action of~$\calO_S(S)^\times$. Locally on $S$ for the
Zariski topology, $(P^\times_\Q)_S$ is trivial, and $P^\times_\Q$ is represented by a
scheme over $J_\Q\times J_\Q^\vee$.

The theorem of the cube gives $P^\times_\Q$ the structure of a
\emph{biextension} of $J_\Q$ and $J_\Q^\vee$ by~$\Gm$, a notion for
the details of which we recommend Section~I.2.5 of~\cite{M-B_PVA},
Grothendieck's Expos\'es~VII and~VIII~\cite{SGA7.I}, and references
therein. This means the following. For $S$ a $\Q$-scheme, $x_1$ and
$x_2$ in~$J_\Q(S)$, and $y$ in~$J^\vee_\Q(S)$, the theorem of the cube
gives a canonical isomorphism of $\calO_S$-modules
\begin{eqn}\label{eq:P_add_1}
(x_1,y)^*P_\Q\otimes_{\calO_S} (x_2,y)^*P_\Q = (x_1+x_2,y)^*P_\Q \,.  
\end{eqn}
This induces a morphism of schemes 
\begin{eqn}\label{eq:plus_1_P}
(x_1,y)^*P^\times_\Q\times_S (x_2,y)^*P^\times_\Q \lto (x_1+x_2,y)^*P^\times_\Q \,.  
\end{eqn}
as follows. For any $S$-scheme~$T$, and $z_1$ in
$((x_1,y)^*P^\times_\Q)(T)$ and $z_2$ in $((x_2,y)^*P^\times_\Q)(T)$,
we view $z_1$ and $z_2$ as nowhere vanishing sections of the invertible
$\calO_T$-modules $(x_1,y)^*P_\Q$ and $(x_2,y)^*P_\Q$. The tensor
product of these two then gives an element of
$((x_1+x_2,y)^*P^\times_\Q)(T)$. This gives $P^\times_\Q\to
J^\vee_\Q$ the structure of a commutative group scheme, which is an
extension of $J_\Q$ by $\Gm$, over the base~$J^\vee_\Q$. We denote
this group law, and the one on $J_\Q\times J^\vee_\Q$, as
\begin{eqn}\label{eq:partial_group_law}
\begin{tikzcd}
(z_1,z_2)\arrow[r, mapsto] \arrow[d, mapsto] &  z_1 +_1 z_2  \arrow[d,
mapsto] & \\
((x_1,y),(x_2,y)) \arrow[r, mapsto] & (x_1,y)+_1(x_2,y) \arrow[r,
equals] & (x_1+x_2,y)\,. 
\end{tikzcd}
\end{eqn}
In the same way, $P^\times_\Q\to J_\Q$ has a group law $+_2$ that
makes it an extension of $J^\vee_\Q$ by $\Gm$ over the
base~$J_\Q$. In this way, $P^\times_\Q$ is both the universal
extension of $J_\Q$ by $\Gm$ and the universal extension of
$J^\vee_\Q$ by~$\Gm$. The final ingredient of the notion of
biextension is that the two partial group laws are compatible in the
following sense. For any $\Q$-scheme~$S$, for $x_1$ and $x_2$ in
$J_\Q(S)$, $y_1$ and $y_2$ in $J^\vee_\Q(S)$, and, for all $i$ and $j$
in $\{1,2\}$, $z_{i,j}$ in $((x_i,y_j)^*P^\times_\Q)(S)$, we have
\begin{eqn}\label{eq:bi_ext_compat}
\begin{tikzcd}
(z_{1,1}+_1 z_{2,1}) +_2 (z_{1,2} +_1 z_{2,2}) \arrow[r,equal]
\arrow[d,mapsto] & 
(z_{1,1}+_2 z_{1,2}) +_1 (z_{2,1} +_2 z_{2,2}) \arrow[d,mapsto] \\
(x_1+x_2,y_1) +_2 (x_1+x_2,y_2) \arrow[r,equal] & 
(x_1,y_1+y_2) +_1 (x_2, y_1+y_2) 
\end{tikzcd}
\end{eqn}
with the equality in the upper line taking place in
$((x_1+x_2,y_1+y_2)^*P^\times_\Q)(S)$.

Now we extend the geometry above over~$\Z$. We denote by $J^0$ the
fibrewise connected component of~$0$ in~$J$, which is an open subgroup
scheme of~$J$, and by $\Phi$ the quotient~$J/J^0$, which is an \'etale
(not necessarily separated) group scheme over~$\Z$, with finite
fibres, supported on~$\Z/n\Z$.  Similarly, we let $J^{\vee 0}$ be the
fibrewise connected component of~$J^\vee$.  Theorem~7.1, in
Expos\'e~VIII of \cite{SGA7.I} gives that $P^\times_\Q$ extends uniquely to a
$\Gm$-biextension
\begin{eqn}\label{eq:P_extended}
P^\times\lto J\times J^{\vee 0}
\end{eqn}
(Grothendieck's pairing on component groups is the obstruction to the
existence of such an extension). Note that in this case the existence
and the uniqueness follow directly from the requirement of extending
the rigidification on $J_\Q\times\{0\}$. For details see
Section~\ref{sec:extension_Poincare_Neron}. 

Our base point $b\in C(\Z)$ gives an embedding $j_b\colon C_\Q\to
J_\Q$, which sends, functorially in $\Q$-schemes~$S$, an element $c\in
C_\Q(S)$ to the class of the invertible $\calO_{C_S}$-module
$\calO_{C_S}(c-b)$. Then $j_b$ extends uniquely to a morphism 
\begin{eqn}\label{eq:jb_extended}
j_b\colon C^\sm \lto J
\end{eqn}
where $C^\sm$ is the open subscheme of $C$ consisting of points at
which $C$ is smooth over~$\Z$. Note that $C_\Q(\Q)=C(\Z)=C^\sm(\Z)$. 

Our next step is to lift~$j_b$, at least on certain opens of~$C^\sm$,
to a morphism to a $\Gm^{\rho-1}$-torsor over~$J$, where $\rho$ is
the rank of the free $\Z$-module $\Hom(J_\Q,J_\Q^\vee)^+$, the
$\Z$-module of self-dual morphisms from $J_\Q$ to~$J^\vee_\Q$. This
torsor will be the product of pullbacks of $P^\times$ via morphisms
\begin{eqn}\label{eq:f_and_c}
(\id,m{\cdot}\circ\tr_c\circ f)\colon J\to J\times J^{\vee 0}\,,
\end{eqn}
with $f\colon J\to J^\vee$ a morphism of group schemes,
$c\in J^\vee(\Z)$, $\tr_c$ the translation by~$c$, $m$ the least
common multiple of the exponents of all $\Phi(\Fbar_p)$ with $p$
ranging over all primes, and $m{\cdot}$ the multiplication by $m$ map
on~$J^\vee$. For such a map~$m{\cdot}\circ\tr_c\circ f$,
$j_b\colon C_\Q\to J_\Q$ can be lifted to
$(\id,m{\cdot}\circ\tr_c\circ f)^*P^\times_\Q$ if and only if
$j_b^*(\id,m{\cdot}\circ\tr_c\circ f)^*P^\times_\Q$ is trivial. The
degree of this $\Gm$-torsor on $C_\Q$ is minus the trace of
$\lambda^{-1}\circ m{\cdot}\circ (f+f^\vee)$ acting on
$\rmH_1(J(\C),\Z)$. For example, for $f=\lambda$ the degree
is~$-4mg$. Note that $j_b\colon C_\Q\to J_\Q$ induces 
\begin{eqn}\label{eq:lambda_and_jb}
j_b^*=-\lambda^{-1}\colon J^\vee_\Q\to J_\Q  \,,
\end{eqn}
(see~\cite{M-B_MP}, Propositions~2.7.9 and~2.7.10). This implies that
for $f$ such that this degree is zero, there is a unique $c$ such that
$j_b^*(\id,\tr_c\circ f)^*P^\times_\Q$ is trivial on~$C_\Q$, and hence
also its $m$th power
$j_b^*(\id,m{\cdot}\circ\tr_c\circ f)^*P^\times_\Q$.

The map 
\begin{eqn}\label{eq:id_f_pullback_P_NS}
\Hom(J_\Q,J_\Q^\vee) \lto \Pic(J_\Q) \lto \NS_{J_\Q/\Q}(\Q) = \Hom(J_\Q,J_\Q^\vee)^+
\end{eqn}
sending $f$ to the class of $(\id,f)^*P_\Q$ sends $f$ to
$f+f^\vee$, hence its kernel is $\Hom(J_\Q,J_\Q^\vee)^-$, the group of
antisymmetric morphisms. But actually, for $f$ antisymmetric, its
image in $\Pic(J_\Q)$ is already zero (see for example~\cite{Be-Ed}
and the references therein). 
Hence the image of $\Hom(J_\Q,J_\Q^\vee)$ in $\Pic(J_\Q)$ is free of
rank~$\rho$, and its subgroup of classes with degree zero on $C_\Q$ is
free of rank~$\rho{-}1$. Let $f_1,\ldots,f_{\rho-1}$ be elements of
$\Hom(J_\Q,J_\Q^\vee) $ whose images in $\Pic(J_\Q)$ form a basis of this
subgroup, and let $c_1,\ldots,c_{\rho-1}$ be the corresponding
elements of~$J^\vee(\Z)$. 

By construction, for each~$i$, the morphism $j_b\colon C_\Q\to J_\Q$
lifts to $(\id,m{\cdot}\circ\tr_{c_i}\circ f_i)^*P^\times_\Q$, unique
up to~$\Q^\times$. Now we spread this out over~$\Z$, to open
subschemes $U$ of $C^\sm$ obtained by removing, for each $q$
dividing~$n$, all but one irreducible
components of~$C^\sm_{\F_q}$, with the remaining irreducible component
geometrically irreducible. For such a $U$, the morphism $\Pic(U)\to\Pic(C_\Q)$
is an isomorphism, and $\calO_C(U)=\Z$, thus, for each~$i$, there is a lift 
\begin{eqn}\label{eq:fi_ci_P}
\begin{tikzcd}
 & (\id,m{\cdot}\circ\tr_{c_i}\circ f_i)^*P^\times \ar[d] \\
U \arrow[r,"j_b"] \arrow[ru,"\wt{j_b}"] & J
\end{tikzcd}
\end{eqn}
unique up to $\Z^\times=\{1,-1\}$.

At this point we can explain the strategy of our approach to the
quadratic Chabauty method. Let $T$ be the $\Gm^{\rho-1}$-torsor on $J$
obtained by taking the product of all
$T_i:=(\id,m{\cdot}\circ\tr_{c_i}\circ f_i)^*P^\times$:
\begin{eqn}\label{eq:def_T}
\begin{tikzcd}
  & T\arrow[d] \arrow[rrr] & & & P^{\times,\rho-1} \arrow[d] \\
  U \arrow[ru,"\wt{j_b}"] \arrow[r,"j_b"]
  & J \arrow{rrr}{(\id,m{\cdot}\circ\tr_{c_i}\circ f_i)_i} & & & J\times(J^{\vee 0})^{\rho-1}\,.
\end{tikzcd}
\end{eqn}
Then each $c\in C_\Q(\Q)=C^\sm(\Z)$ lies in one of the finitely many
$U(\Z)$'s. For each~$U$, we have a lift
$\wt{j_b}\colon U\to T$, and, for each prime number~$p$,
$\wt{j_b}(U(\Z))$ is contained in the intersection, in
$T(\Z_p)$, of $\wt{j_b}(U(\Z_p))$ and the closure $\ol{T(\Z)}$
of $T(\Z)$ in~$T(\Z_p)$ with the $p$-adic topology. Of course, one
expects this closure to be of dimension at most~$r:=\rank(J(\Q))$, and
therefore one expects this method to be successful if $r<g+\rho-1$,
the dimension of~$T(\Z_p)$. The next two sections make this strategy
precise, giving first the necessary $p$-adic formal and
analytic geometry, and then the description of $\ol{T(\Z)}$ as a
finite disjoint union of images of $\Z_p^r$ under maps constructed
from the biextension structure.

\section{From algebraic geometry to formal geometry}
\label{sec:formal_geometry}
Let $p$ be a prime number. Given $X$ a smooth scheme of relative
dimension $d$ over $\Z_p$ and $x \in X(\F_p)$ let us describe the set
$X(\Z_p)_x$ of elements of $X(\Z_p)$ whose image in $X(\F_p)$ is~$x$.
The smoothness implies that the maximal ideal of $\calO_{X,x}$ is
generated by $p$ together with $d$ other elements $t_1,\ldots, t_d$.
In this case we call $p,t_1,\ldots ,t_d$ \emph{parameters at~$x$}; if
moreover $x_l \in X(\Z_p)_x$ is a lift of~$x$ such that
$t_1(x_l)=\ldots t_d(x_l)=0$ then we say that the $t_i$'s are
\emph{parameters at~$x_l$}. The $t_i$ can be evaluated on all the
points in~$X(\Z_p)_x$, inducing a bijection
$t:=(t_1,\ldots, t_d)\colon X(\Z_p)_x \to (p\Z_p)^d$. We get a
bijection
\begin{eqn}\label{eq:param_bijection}
\tilde{t}:=(\tilde t_1,\ldots, \tilde t_d) = \left( \frac{t_1}{p},\ldots,
  \frac{t_d}{p}\right)\colon X(\Z_p)_x \overset{\sim}{\lto} \Z_p^d\,.  
\end{eqn}
This bijection can be interpreted geometrically as follows. Let
$\pi\colon\wt{X}_x\to X$ denote the blow up of $X$ in~$x$.  By
shrinking~$X$, $X$ is affine and the $t_i$ are regular on~$X$,
$t\colon X\to\A^d_{\Z_p}$ is etale, and
$t^{-1}\{0_{\F_p}\}=\{x\}$. Then $\pi\colon \wt{X}_x\to X$ is
the pull back of the blow up of $\A^d_{\Z_p}$ at the origin
over~$\F_p$. The affine open part $\wt{X}_x^p$ of
$\wt{X}_x$ where $p$ generates the image of the ideal $m_x$
of~$x$ is the pullback of the corresponding open part of the blow up
of $\A^d_{\Z_p}$, which is the multiplication by~$p$ morphism
$\A^d_{\Z_p}\to \A^d_{\Z_p}$ that corresponds to
$\Z_p[t_1,\ldots,t_d]\to \Z_p[\tilde{t}_1,\ldots,\tilde{t}_d]$ with
$t_i\mapsto p\tilde{t}_i$. It follows that the $p$-adic completion
$\calO(\wt{X}_x^p)^{\wedge_p}$ of $\calO(\wt{X}_x^p)$ is
the $p$-adic completion
$\Z_p\langle\tilde{t}_1,\ldots,\tilde{t}_d\rangle$ of
$\Z_p[\tilde{t}_1,\ldots,\tilde{t}_d]$. Explicitly, we have
\begin{eqn}\label{eq:int_conv_power_series}
\Z_p\langle\tilde{t}_1,\ldots,\tilde{t}_d\rangle= \left\{ \sum_{ I \in
    \N^d} a_I \tilde{t}^I\in\Z_p[[\tilde{t}_1,\ldots,\tilde{t}_d]] :
  \forall n\geq0, \, \text{ $\forall^{\text{almost}}I$},\, v_p(a_I)
  \geq n \right\} \,.
\end{eqn}
With these definitions, we have 
\begin{eqn}\label{eq:res_disk_power_series}
\begin{aligned}
 X(\Z_p)_x & = \wt{X}_x^p(\Z_p) =
\Hom(\Z_p\langle\tilde{t}_1,\ldots,\tilde{t}_d\rangle,\Z_p) =
\A^d(\Z_p)\,,\\
(\wt{X}_x^p)_{\F_p} & = \Spec(\F_p[\tilde{t}_1,\ldots,\tilde{t}_d]).
\end{aligned}
\end{eqn}
The affine space $(\wt{X}_x^p)_{\F_p}$ is canonically a torsor
under the tangent space of $X_{\F_p}$ at~$x$.

This construction is functorial. Let $Y$ be a smooth $\Z_p$-scheme,
$f\colon X\to Y$ a morphism over~$\Z_p$, and $y:=f(x)\in
Y(\F_p)$. Then the ideal in $\calO_{\wt{X}_x^p}$ generated by
the image of $m_{f(x)}$ is generated by~$p$. That gives us a morphism
$\wt{X}_x^p\to \wt{Y}_{f(x)}^p$, and then a morphism
from $\calO(\wt{Y}_{f(x)}^p)^{\wedge_p}$ to
$\calO(\wt{X}_x^p)^{\wedge_p}$. Reduction mod~$p$ then gives a
morphism
$(\wt{X}_x^p)_{\F_p}\to (\wt{Y}_{f(x)}^p)_{\F_p}$, the
tangent map of $f$ at~$x$, up to a translation.

If this tangent map is injective, and $d_x$ and $d_y$ denote the
dimensions of $X_{\F_p}$ at~$x$ and of $Y_{\F_p}$ at~$y$, then there
are $t_1,\ldots,t_{d_y}$ in $\calO_{Y,y}$ such that
$p,t_1,\ldots,t_{d_y}$ are parameters at~$y$, and such that
$t_{d_x+1},\ldots,t_{d_y}$ generate the kernel of
$\calO_{Y,y}\to\calO_{X,x}$. Then the images in~$\calO_{X,x}$ of
$p,t_1,\ldots,t_{d_x}$ are parameters at~$x$, and
$\calO(\wt{Y}_{f(x)}^p)^{\wedge_p}\to\calO(\wt{X}_x^p)^{\wedge_p}$
is
$\Z_p\langle\tilde{t}_1,\ldots,\tilde{t}_{d_y}\rangle\to
\Z_p\langle\tilde{t}_1,\ldots,\tilde{t}_{d_x}\rangle$, with kernel
generated by $\tilde{t}_{d_x+1},\ldots,\tilde{t}_{d_y}$.

\section{Integral points, closure and finiteness}
\label{sec:closure-finiteness}
Let us now return to our original problem. The notation $U$, $J$, $T$,
$j_b$, $\wt{j_b}$, $r$, $\rho$ etc., is as at the end of
Section~\ref{sec:alg_geometry}. We assume moreover that $p$ does not
divide~$n$ ($n$ as in the start of Section~\ref{sec:alg_geometry}) and
that $p>2$ (for $p=2$ everything that follows can probably 
be adapted by working with residue polydiscs modulo~$4$). 

Let $u$ be in $U(\F_p)$, and $t:=\wt{j_b}(u)$.  We want a
description of the closure $\ol{T(\Z)_t}$ of~$T(\Z)_t$
in~$T(\Z_p)_t$. Using the biextension structure of~$P^\times$, we will
produce, for each element of~$J(\Z)_{j_b(u)}$, an element of $T(\Z)$
over it. Not all of these points are in~$T(\Z)_t$, but we will then
produce a subset of~$T(\Z)_t$ whose closure is~$\ol{T(\Z)_t}$.

If $T(\Z)_t$ is empty then $\ol{T(\Z)_t}$ is empty, too. So we assume
that we have an element $\wt{t} \in T(\Z)_t$ and we define
$x_{\wt t} \in J(\Z)$ to be the projection of~$\wt t$.
Let $f=(f_1,\ldots, f_{\rho-1})\colon J\to J^{\vee,\rho-1}$, let
$c = (c_1,\ldots,c_{\rho-1}) \in J^{\vee,\rho-1}(\Z)$.  We denote by
$P^{\times,\rho-1}$ the product over $J\times (J^{\vee 0})^{\rho -1}$
of the $\rho{-}1$ $\Gm$-torsors obtained by pullback of $P^\times$ via
the projections to $J\times J^{\vee 0}$; it is a biextension of $J$
and $(J^{\vee 0})^{\rho -1}$ by $\Gm^{\rho-1}$, and
$T=(\id,m{\cdot}\circ\tr_c\circ f)^*P^{\times,\rho-1}$. We choose a
basis $x_1,\ldots, x_r$ of the free $\Z$-module~$J(\Z)_0$, the kernel
of $J(\Z)\to J(\F_p)$.  For each $i,j \in \{1,\ldots, r \}$ we choose
$P_{i,j}$, $R_{i,\wt{t}}$, and $S_{\wt{t},j}$ in
$P^{\times,\rho-1}(\Z)$ whose images in
$(J\times (J^{\vee 0 })^{\rho -1})(\Z)$ are $(x_i,f(mx_j))$,
$(x_i, (m{\cdot} \circ \tr_c \circ f)(x_{\wt t}))$
and~$( x_{\wt t},f(mx_j))$:
\begin{eqn}\label{eq:Pij}
\begin{tikzcd}
  P_{i,j} \arrow[d, mapsto] &  R_{i,\wt t} \arrow[d, mapsto]
  & S_{\wt t,j} \arrow[d, mapsto] & P^{\times,\rho-1} \arrow[d] \\
  (x_i,f(mx_j))
  & (x_i, (m{\cdot} \circ \tr_c \circ f)(x_{\wt{t}})) &
  ( x_{\wt t},f(mx_j)) & J\times (J^{\vee 0 })^{\rho -1}\,.
\end{tikzcd}
\end{eqn}
For each such choice there are $2^{\rho-1}$ possibilities.

For each $n \in \Z^r$ we use the biextension structure on
$P^{\times,\rho-1}\to J \times (J^{\vee 0})^{\rho-1}$ to define the
following points in~$P^{\times,\rho-1}(\Z)$, with specified images
in $(J \times (J^{\vee 0})^{\rho-1})(\Z)$:
\begin{eqn}\label{eq:A_and_B}
\begin{tikzcd} 
\displaystyle  A_{\wt t}(n)=\sideset{}{_2}\sum_{j=1}^r n_j
\cdot_2 S_{\wt t, j}\, \arrow[d, mapsto] & \displaystyle
B_{\wt t}(n) = \sideset{}{_1}\sum_{i=1}^r n_i \cdot_1
R_{i,\wt t} \, \arrow[d, mapsto]  \\
\displaystyle \left( x_{\wt t}\,,  \sum_{i=1}^r n_i f(mx_i)\right)
&  \displaystyle  \left( \sum_{i=1}^r n_i x_i\,,  (m{\cdot}\circ\tr_c
\circ f)(x_{\wt t}) \right)
\end{tikzcd}
\end{eqn}
\begin{eqn}\label{eq:C}
\begin{tikzcd} 
\displaystyle C (n) = \sideset{}{_1}\sum_{i=1}^r n_i \cdot_1 \left(
\sideset{}{_2}\sum_{j=1}^r n_j \cdot_2 P_{i,j} \right) \, \arrow[d,
mapsto]
\\
\displaystyle \left( \sum_{i=1}^rn_i x_i\,,
\sum_{i=1}^r n_i  f(mx_i) \right)
\end{tikzcd}
\end{eqn}
where $\sum_1$ and $\cdot_1$ denote iterations of the first partial
group law~$+_1$ as in~(\ref{eq:partial_group_law}), and analogously
for the second group law.  We define, for all $n\in\Z^r$, 
\begin{eqn}\label{eq:D_ttilde_n}
D_{\wt{t}}(n) := \left(C(n)+_2 B_{\wt
  t}(n)\right) +_1 \left(A_{\wt t}(n) +_2 \wt{t} \right) \,
\in \, P^{\times,\rho-1}(\Z) \,,
\end{eqn}
which is mapped to 
\begin{eqn}\label{eq:D_ttilde_n_J}
\left( x_{\wt t} + \sum_{i=1}^r n_i x_i, (m{\cdot}\circ \tr_c \circ
  f) \left(x_{\wt t} + \sum_{i=1}^r n_i x_i \right)\right)
\in \left(J\times (J^{\vee 0})^{\rho-1}\right)(\Z)\,.
\end{eqn}
Hence $D_{\wt{t}}(n) $ is in~$T(\Z)$, and its image in
$J(\F_p)$ is~$j_b(u)$. We do not know its image in~$T(\F_p)$.

We claim that for $n$ in $(p{-}1)\Z^r$, $D_{\wt{t}}(n)$ is
in~$T(\Z)_t$. Let $n'$ be in $\Z^r$ and let $n=(p{-}1)n'$. Then, in
the trivial
$\F_p^{\times,\rho-1}$-torsor~$P^{\times,\rho-1}(j_b(u),0)$,
on which~$+_2$ is the group law, we have:
\begin{eqn}\label{eq:Atilde_n_is_1}
  A_{\wt t}(n) = (p{-}1){\cdot}_2 A_{\wt
    t}(n')=1\quad\text{in $\F_p^{\times,\rho-1}.$}
\end{eqn}
Similarly, in
$P^{\times,\rho-1}(0,(m{\cdot} \circ \tr_c \circ f)(j_b(u)))
= \F_p^{\times,\rho-1}$, we have 
$B_{\wt t}(n)=1$, and, similarly, in
$P^{\times,\rho-1}(0,0) = \F_p^{\times,\rho-1}$, we have
$C(n) =1$.  So, with apologies for the mix of additive
and multiplicative notations, in $P^{\times,\rho-1}(\F_p)$ we have
\begin{eqn}\label{eq:Dtilde_n_Fp}
D_{\wt t}(n) = \left(1 +_2 1\right) +_1 \left(1 +_2 t \right) =  t\,,
\end{eqn}
mapping to the following element in~$(J\times J^{\vee0,\rho-1})(\F_p)$:
\begin{eqn}\label{eq:Dtilde_n_Fp_J}
\begin{aligned}
& \left((0,0) +_2 ((0,(m{\cdot} \circ \tr_c \circ f)(j_b(u))))\right)
+_1\left((j_b(u),0) +_2 (j_b(u), (m{\cdot} \circ \tr_c \circ
  f)(j_b(u)))\right)\\
& = (j_b(u), (m{\cdot} \circ \tr_c \circ f)(j_b(u)))\,.
\end{aligned}
\end{eqn}
We have proved our claim that $D_{\wt t}(n) \in T(\Z)_t$.

So we now have the map
\begin{eqn}\label{eq:kappa_Z}
\kappa_\Z\colon \Z^r\to T(\Z)_t,\quad n\mapsto D_{\wt t}((p-1)n)\,.
\end{eqn}
The following theorem will be proved in Section~\ref{SectionBiextensionZp}.
\begin{Theorem}\label{thm:p-adic_closure} 
Let $x_1,\ldots,x_g$ be in~$\calO_{J,j_b(u)}$ such that together
with~$p$ they form a system of parameters of~$\calO_{J,j_b(u)}$, and
let $v_1,\ldots,v_{\rho-1}$ be in~$\calO_{T,t}$ such that
$p,x_1,\ldots,x_g,v_1,\ldots,v_{\rho-1}$ are parameters
of~$\calO_{T,t}$. As in Section~\ref{sec:formal_geometry} these
parameters, divided by~$p$, give a bijection
\begin{subeqn}\label{eq:T_to_Zp}
T(\Z_p)_t \lto \Z_p^{g+\rho-1}\,.
\end{subeqn}
The composition of $\kappa_\Z$ with the map~(\ref{eq:T_to_Zp}) is
given by uniquely determined $\kappa_1,\ldots,\kappa_{g+\rho-1}$
in~$\calO(\A_{\Z_p}^r)^{\wedge_p}=\Z_p\langle z_1,\ldots,z_r\rangle$.
The images in $\F_p[z_1,\ldots,z_r]$ of $\kappa_1,\ldots,\kappa_g$ are
of degree at most~$1$, and the images of
$\kappa_{g+1},\ldots,\kappa_{g+\rho-1}$ are of degree at most~$2$.
The map $\kappa_\Z$ extends uniquely to the continuous map
\begin{subeqn}\label{eq:kappa}
  \kappa=(\kappa_1,\ldots,\kappa_{g+\rho-1})\colon
  \A^r(\Z_p) = \Z_p^r  \lto T(\Z_p)_t\,.   
\end{subeqn}
and the image of~$\kappa$ is~$\ol{T(\Z)_t}$.
\end{Theorem}

Now the moment has come to confront $U(\Z_p)_u$
with~$\ol{T(\Z)_t}$. We have $\wt{j_b}\colon U\to T$, whose tangent
map (mod~$p$) at $u$ is injective (here we use that $C_{\F_p}$ is
smooth over~$\F_p$).  Then, as at the end
of~Section~\ref{sec:formal_geometry},
$\wt{j_b}\colon \wt{U}_u^p\to \wt{T}_t^p$ is, after reduction mod~$p$,
an affine linear embedding of codimension~$g{+}\rho{-}2$,
$\wt{j_b}^*\colon \calO(\wt{T}_t^p)^{\wedge_p}\to
\calO(\wt{U}_u^p)^{\wedge_p}$ is surjective and its kernel is
generated by elements $f_1,\ldots,f_{g+\rho-2}$ (we apologise for
using the same letter as for the components of
$f\colon J\to J^{\vee,\rho-1}$), whose images in
$\F_p\otimes\calO(\wt{T}_t^p)$ are of degree at most~$1$, and such
that $f_1,\ldots,f_{g-1}$ are
in~$\calO(\wt{J}_{j_b(u)}^p)^{\wedge_p}$.  The pullbacks $\kappa^*f_i$
are in $\Z_p\langle z_1,\ldots,z_r\rangle$; let $I$ be the ideal in
$\Z_p\langle z_1,\ldots,z_r\rangle$ generated by them, and let
\begin{eqn}\label{eq:def_A}
A:=\Z_p\langle z_1,\ldots,z_r\rangle/I \,.
\end{eqn}
Then the elements of
$\Z_p^r$ whose image is in $U(\Z_p)_u$ are zeros of~$I$, hence
morphisms of rings from $A$ to~$\Z_p$, and hence from the reduced
quotient $A_\red$ to~$\Z_p$.

\begin{Theorem}\label{thm:finiteness}
For $i\in\{1,\ldots,g{+}\rho{-}2\}$, let $\kappa^*\ol{f_i}$ be the
image of $\kappa^*f_i$ in $\F_p[z_1,\ldots,z_r]$, and let $\ol{I}$
be the ideal of $\F_p[z_1,\ldots,z_r]$ generated by them. Then
$\kappa^*\ol{f_1},\ldots, \kappa^*\ol{f_{g-1}}$ are of degree at
most~$1$, and $\kappa^*\ol{f_g},\ldots, \kappa^*\ol{f_{g+\rho-2}}$
are of degree at most~$2$. Assume that
$\ol{A}:=A/pA=\F_p[z_1,\ldots,z_r]/\ol{I}$ is finite. Then $\ol{A}$
is the product of its localisations $\ol{A}_m$ at its finitely many
maximal ideals~$m$.  The sum of the $\dim_{\F_p}\ol{A}_m$ over the
$m$ such that $\ol{A}/m=\F_p$ is an upper bound for the number of
elements of $\Z_p^r$ whose image under $\kappa$ is in~$U(\Z_p)_u$,
and also an upper bound for the number of elements of $U(\Z)$ with
image $u$ in~$U(\F_p)$.
\end{Theorem}
\begin{proof}
As every $\ol{f_i}$ is of degree at most $1$ in
$x_1,\ldots,x_g,v_1,\ldots,v_{\rho-1}$, every $\kappa^*\ol{f_i}$ is an
$\F_p$-linear combination of $\kappa_1,\ldots,\kappa_{g+\rho-1}$,
hence of degree at most~$2$. For $i<g$, $\ol{f_i}$ is a linear
combination of $x_1,\ldots,x_g$, and therefore $\kappa^*\ol{f_i}$ is
of degree at most~$1$.

We claim that $A$ is $p$-adically complete. More generally, let $R$ be
a noetherian ring that is $J$-adically complete for an ideal~$J$, and
let $I$ be an ideal in~$R$. The map from $R/I$ to its $J$-adic
completion $(R/I)^\wedge$ is injective (\cite[Thm.10.17]{At-McD}).  As
$J$-adic completion is exact on finitely generated $R$-modules
(\cite[Prop.10.12]{At-McD}), it sends the surjection $R\to R/I$ to a
surjection $R=R^\wedge\to (R/I)^\wedge$ (see~\cite[Prop.10.5]{At-McD}
for the equality $R=R^\wedge$).  It follows that
$R/I \to (R/I)^\wedge$ is surjective.

Now we assume that $\ol{A}$ is finite. As $A$ is $p$-adically
complete, $A$ is the limit of the system of its quotients by powers
of~$p$. These quotients are finite: for every $m\in\Z_{\geq1}$,
$A/p^{m+1}A$ is, as abelian group, an extension of $A/pA$ by a
quotient of~$A/p^mA$. As $\Z_p$-module, $A$ is generated by any lift
of an $\F_p$-basis of~$\ol{A}$. Hence $A$ is finitely generated as
$\Z_p$-module.

The set of elements of $\Z_p^r$ whose image under $\kappa$ is
in~$U(\Z_p)$ is in bijection with the set of $\Z_p$-algebra morphisms
$\Hom(A,\Z_p)$. As $A$ is the product of its localisations~$A_m$ at
its maximal ideals, $\Hom(A,\Z_p)$ is the disjoint union of the
$\Hom(A_m,\Z_p)$. For each~$m$, $\Hom(A_m,\Z_p)$ has at most
$\rank_{\Z_p}(A_m)$ elements, and is empty if $\F_p\to A/m$ is not an
isomorphism. This establishes the upper bound for the number of
elements of $\Z_p^r$ whose image under $\kappa$ is in~$U(\Z_p)$.  By
Theorem~\ref{thm:p-adic_closure}, the elements of $U(\Z)$ with image
$u$ in $U(\F_p)$ are in $\ol{T(\Z)_t}$, and therefore of the form
$\kappa(x)$ with $x\in\Z_p^r$ such that $\kappa(x)$ is
in~$U(\Z_p)_u$. This establishes the upper bound for the number of
elements of $U(\Z)$ with image $u$ in~$U(\F_p)$.
\end{proof}

We include some remarks to explain how Theorem~\ref{thm:finiteness}
can be used, and what we hope that it can do. 

\begin{Remark}
The $\kappa^*\ol{f_i}$ in Theorem~\ref{thm:finiteness} can be computed
from the reduction $\F_p^r\to T(\Z/p^2\Z)$ of~$\kappa_\Z$ and (to get
the~$\ol{f_i}$) from $\wt{j_b}\colon U(\Z/p^2\Z)_u\to
T(\Z/p^2\Z)_t$. For this, one does not need to treat $T$ and $J$ as schemes,
one just computes with $\Z/p^2\Z$-valued points. Now assume that
$r\leq g+\rho-2$. If, for some prime~$p$, the criterion in
Theorem~\ref{thm:finiteness} fails (that is, $\ol{A}$ is not finite),
then one can try the next prime. We hope (but also
expect) that one quickly finds a prime $p$ such that $\ol{A}$ is
finite for every $U$ and for every $u$ in $U(\F_p)$ such that
$\wt{j_b}(u)$ is in the image of $T(\Z)\to T(\F_p)$. By the
way, note that our notation in Theorem~\ref{thm:finiteness} does not
show the dependence on~$U$ and $u$ of $\wt{j_b}$, $\kappa_\Z$,
$\kappa$ and the~$\ol{f_i}$. Instead of varying $p$, one could also
increase the $p$-adic precision, and then the result of
Section~\ref{sec:finiteness} proves that one gets an upper bound for
the number of elements of~$U(\Z)$.
\end{Remark}

\begin{Remark}
  If $r < g + \rho -2$ then we think that it is likely (when
  varying~$p$), for dimension reasons, unless something special
  happens as in~\cite{BBCCE} or Remark~8.9 of~\cite{Ba-Do-1}, that, for all
  $u\in U(\F_p)$, the upper bound in Theorem~\ref{thm:finiteness} for
  the number of elements of $U(\Z)$ with image $u$ in~$U(\F_p)$ is
  sharp. For a precise conjecture in the context of Chabauty's method,
  see the ``Strong Chabauty'' Conjecture in~\cite{Stoll}. 
\end{Remark}

\begin{Remark}
Suppose that $r = g + \rho -2$. Then we expect, for dimension
reasons, that it is likely (when varying~$p$) that, for some
$u\in U(\F_p)$, the upper bound in Theorem~\ref{thm:finiteness} for the number of
elements of $U(\Z)$ with image $u$ in~$U(\F_p)$ is not sharp. Then, as
in the classical Chabauty method, one must combine the information
gotten from several primes, analogous to `Mordell-Weil
sieving'; see~\cite{Muller}. In our situation, this amounts to the
following. Suppose that we are given a subset $B$ of~$U(\Z)$ that we
want to prove to be equal to~$U(\Z)$. Let $B'$ be the complement in
$U(\Z)$ of~$B$. For
every prime $p>2$ not dividing~$n$, Theorem~\ref{thm:finiteness}
gives, interpreting~$\ol{A}$ as in the end of the proof of
Theorem~\ref{thm:finiteness}, 
a subset $O_p$ of~$J(\Z)$, with $O_p$ a union of cosets for the subgroup
$p{\cdot}\ker(J(\Z)\to J(\F_p))$, that contains~$j_b(B')$. Then one
hopes that, taking a large enough finite set $S$ of primes, the
intersection of the $O_p$ for $p$ in~$S$ is empty.
\end{Remark}

\section{Parametrisation of integral points, and power series}
\label{SectionBiextensionZp}

In this section we give a proof of Theorem~\ref{thm:p-adic_closure}.
The main tools here are the formal logarithm and formal exponential of
a commutative smooth group scheme over a $\Q$-algebra (\cite{Hon},
Theorem~1): they give us identities like
$n{\cdot}g = \exp(n{\cdot}\log g)$ that allow us to extend the
multiplication to elements $n$ of~$\Z_p$.

The evaluation map from $\Z_p\langle z_1,\ldots,z_n\rangle$ to the set
of maps $\Z_p^n\to\Z_p$ is injective (induction on~$n$, non-zero
elements of $\Z_p\langle z\rangle$ have only finitely many zeros
in~$\Z_p$).

We say that a map $f\colon\Z_p^n\to\Z_p^m$ is \emph{given by integral
  convergent power series} if its coordinate functions are in
$\Z_p\langle z_1,\ldots,z_n\rangle =
\calO(\A^n_{\Z_p})^{\wedge_p}$. This property is stable under
composition: composition of polynomials over $\Z/p^k\Z$ gives
polynomials.

\subsection{Logarithm and exponential}
Let $p$ be a prime number, and let $G$ be a commutative group scheme,
smooth of relative dimension~$d$ over a scheme~$S$ smooth over~$\Z_p$,
with unit section $e$ in~$G(S)$.  For any $s$ in~$S(\F_p)$,
$G(\Z_p)_{e(s)}$ is a group fibred over~$S(\Z_p)_s$. The fibres have a
natural $\Z_p$-module structure: $G(\Z_p)_{e(s)}$ is the limit of the
$G(\Z/p^n\Z)_{e(s)}$ ($n\geq 1$), $S(\Z_p)_s$ is the limit of
the~$S(\Z/p^n\Z)_s$, and for each $n\geq 1$, the fibres of
$G(\Z/p^n\Z)_{e(s)}\to S(\Z/p^n\Z)_s$ are commutative groups
annihilated by~$p^{n-1}$. Let $T_{G/S}$ be the relative (geometric)
tangent bundle of $G$ over~$S$. Then its pullback $T_{G/S}(e)$ by $e$
is a vector bundle on $S$ of rank~$d$.
\begin{subLemma}\label{lem:log_and_exp}
  In this situation, and with $n$ the relative dimension of $S$
  over~$\Z_p$, the formal logarithm and exponential of $G$ base
  changed to $\Q\otimes\calO_{S,s}$ converge to maps
\[
\begin{aligned}
    \log \colon & \wt{G}^p_{e(s)}(\Z_p) = G(\Z_p)_{e(s)}\to (T_{G/S}(e))(\Z_p)_{0(s)} \\
    \exp\colon & \wt{T}_{G/S}(e)_{0(s)}^p(\Z_p) = (T_{G/S}(e))(\Z_p)_{0(s)}\to G(\Z_p)_{e(s)}\,,
\end{aligned}
\]
that are each other's inverse and, after a choice of parameters for
$G\to S$ at~$e(s)$ as in~(\ref{eq:param_bijection}), are given by $n+d$
elements of $\calO(\wt{G}^p_{e(s)})^{\wedge_p}$ and $n+d$
elements of $\calO(\wt{T}_{G/S}(e)_{0(s)}^p)^{\wedge_p}$.

For $a$ in $\Z_p$ and $g$ in $G(\Z_p)_{e(s)}$ we have $a{\cdot}g =
\exp(a{\cdot}\log g)$, and,  after a choice of parameters for $G\to S$
at~$e(s)$, this map $\Z_p \times G(\Z_p)_{e(s)}\to G(\Z_p)_{e(s)}$ is
given by $n+d$ elements of
$\calO(\A^1_{\Z_p}\times_{\Z_p}\wt{G}^p_{e(s)})^{\wedge_p}$.
The induced morphism
$\A^1_{\F_p}\times(\wt{G}^p_{e(s)})_{\F_p}\to
(\wt{G}^p_{e(s)})_{\F_p}$, where $(\wt{G}^p_{e(s)})_{\F_p}$
is viewed as the product of $T_{S_{\F_p}}(s)$ and~$T_{G/S}(e(s))$, is a morphism
over~$T_{S_{\F_p}}(s)$, bilinear in  $\A^1_{\F_p}$ and~$T_{G/S}(e(s))$.
\end{subLemma}
\begin{proof}
Let $t_1,\ldots,t_n$ be in $\calO_{S,s}$ such that $p, t_1,\ldots,t_n$
are parameters at~$s$. Then we have a bijection
\begin{subeqn}\label{eq:param_S}
  \tilde{t}\colon S(\Z_p)_s\to \Z_p^n, \quad
  a\mapsto p^{-1}{\cdot}(t_1(a),\ldots,t_n(a))\,.
\end{subeqn}
Similarly, let $x_1,\ldots,x_d$ be
generators for the ideal $I_{e(s)}$ of $e$ in $\calO_{G,e(s)}$. Then
$p$, the $t_i$ and the $x_j$ together are parameters
for~$\calO_{G,e(s)}$, and give the bijection
\begin{subeqn}\label{eq:param_G}
(t,x)^\sim\colon G(\Z_p)_{e(s)}\to \Z_p^{n+d}\,,\quad
b\mapsto p^{-1}{\cdot} (t_1(b),\ldots,x_d(b))\,.
\end{subeqn}
The $\de x_i$ form an $\calO_{S,s}$-basis of $\Omega^1_{G/S}(e)_s$,
and so give translation invariant differentials $\omega_i$
on~$G_{\calO_{S,s}}$. As $G$ is commutative, for all~$i$,
$\de \omega_i =0$ (\cite{Hon}, Proposition~1.3). We also have the dual
$\calO_{S,s}$-basis $\partial_i$ of~$T_{G/S}(e)$ and the bijection
\begin{subeqn}\label{eq:param_T}
(t,x)^\sim\colon (T_{G/S}(e))(\Z_p)_{0(s)}\to\Z_p^{n+d}\,,\quad
(a,\sum_iv_i\partial_i)\mapsto
p^{-1}{\cdot} (t_1(a),\ldots,t_n(a),v_1,\ldots,v_d)\,.
\end{subeqn}
Then $\log$ is given
by elements $\log_i$ in $(\Q\otimes\calO_{S,s})[[x_1,\ldots,x_d]]$
whose constant term is~$0$, uniquely determined (Proposition~1.1
in~\cite{Hon}) by the equality
\begin{subeqn}\label{eq:dlogi}
  \de\log_i = \omega_i \,, \quad \text{in}\,\,
  \oplus_j \calO_{S,s}[[x_1,\ldots,x_d]]{\cdot}\de x_j \,.
\end{subeqn}
Hence the formula from calculus,
$\log_i(x)-\log_i(0) = \int_0^1 (t\mapsto tx)^*\omega_i$, gives us
that, with
\begin{subeqn}\label{eq:coeff_logi}
\log_i=\sum_{J\neq0}\log_{i,J}x^J\quad\text{and}\quad \log_{i,J}\in(\Q\otimes
\calO_{S,s})\,, 
\end{subeqn}
we have, for all $i$ and $J$, with $|J|$ denoting the total degree of~$x^J$,
\begin{subeqn}\label{eq:logiJ}
  |J|{\cdot}\log_{i,J}\in\calO_{S,s}\,.
\end{subeqn}
The claim about convergence and definition of
$\log\colon G(\Z_p)_{e(s)}\to (T_{G/S}(e))(\Z_p)_{0(s)}$, is now
equivalent to having an analytic bijection $\Z_p^{n+d}\to\Z_p^{n+d}$
given by
\begin{subeqn}\label{eq:series_log}
\begin{tikzcd}
  G(\Z_p)_{e(s)} \arrow{rr}{?} \arrow{d}{(t,x)^\sim} & &
  (T_{G/S}(e))(\Z_p)_{0(s)}\arrow{d}{(t,x)^\sim}    \\
  \Z_p^{n+d}\arrow{rr}{?} & & \Z_p^{n+d} \\
  (a,b) \arrow[rr, mapsto, "?"] & &
  \left(a,p^{-1}{\cdot}\left(\sum_{J\neq0}\log_{i,J}(\tilde{t}^{-1}(a))(pb)^J\right)_i\right)\,. 
\end{tikzcd}
\end{subeqn}
We have, for each~$i$,
\begin{subeqn}\label{eq:logi}
  p^{-1}{\cdot}\sum_{J\neq0}\log_{i,J}(\tilde{t}^{-1}(a))(pb)^J
  = \sum_{J\neq0} \frac{p^{|J|-1}}{|J|}(|J|\log_{i,J})(\tilde{t}^{-1}(a))b^J\,.
\end{subeqn}
For each $i$, this expression is an element of
$\Z_p\langle\tilde{t}_1,\ldots,\tilde{t}_n,\tilde{x}_1,\ldots,\tilde{x}_d\rangle
= \calO(\wt{G}^p_{e(s)})^{\wedge_p}$, 
even when $p=2$, because for each~$J$, $|J|\log_{i,J}$ is in
$\calO_{S,s}$, which is contained in
$\Z_p\langle\tilde{t}_1,\ldots,\tilde{t}_n\rangle$, and the function
$\Z_{\geq1}\to\Q_p$, $r\mapsto p^{r-1}/r$ has values in $\Z_p$ and
converges to~$0$.  The existence and analyticity of $\log$ is now
proved (even for $p=2$). As $p>2$, the image of (\ref{eq:logi}) in
$\F_p\otimes \calO(\wt{G}^p_{e(s)})^{\wedge_p}$
is~$\tilde{x}_i$, and on the first $n$ coordinates, $\log$ is the
identity, so, by applying Hensel modulo powers of~$p$,
$\log$ is invertible, and the inverse is also given by $n+d$ elements
of $\calO(\wt{T}_{G/S}(e)_{0(s)}^p)^{\wedge_p}$.

The function $\Z_p \times G(\Z_p)_{e(s)}\to G(\Z_p)_{e(s)}$,
$(a,g)\mapsto \exp(a{\cdot}\log g)$ is a composition of maps given by
integral convergent power series, hence it is also of that form.
\end{proof}

\subsection{Parametrisation by power series}
\label{sec:param_power_series}

The notation and assumptions are as in the beginning of
Section~\ref{sec:closure-finiteness}, in particular, $p>2$ and $T$ is
as defined in~(\ref{eq:def_T}). We have a
$t$ in $T(\F_p)$, with image $j_b(u)$ in~$J(\F_p)$, and a $\tilde{t}$
in $T(\Z)$ lifting~$t$.  For every $Q$ in $T(\Z)$ mapping to $j_b(u)$
in $J(\F_p)$ there are unique $\eps\in\Z^{\times,\rho-1}$ and
$n\in\Z^r$ such that $Q=\eps{\cdot}D_{\tilde{t}}(n)$: the image of $Q$
in $J(\Z)$ is in~$J(\Z)_{j_b(u)}$, hence differs from the image
$x_{\tilde{t}}$ in $J(\Z)$ of $\tilde{t}$ by an element of $J(\Z)_0$
(with here $0\in J(\F_p)$), $\sum_i n_ix_i$ for a unique $n\in\Z^r$,
hence $D_{\tilde{t}}(n)$ and $Q$ are in $T(\Z)$ and have the same
image in $J(\Z)$, and that gives the unique~$\eps$. So we have a
bijection
\begin{subeqn}\label{eq:param_Zpoints_T}
  \Z^{\times,\rho-1}\times\Z^r \lto T(\Z)_{j_b(u)} =
  \{Q\in T(\Z) : Q\mapsto j_b(u)\in J(\F_p)\}\,,
  \quad
  (\eps,n)\mapsto \eps{\cdot}D_{\tilde{t}}(n)\,.
\end{subeqn}
But a problem that we are facing is that the map
$\Z^r\to T(\F_p)_{j_b(u)}$ sending $n$ to the image
of~$D_{\tilde{t}}(n)$ depends on the (unknown) images of the $P_{i,j}$,
$R_{i,\tilde{t}}$ and $S_{\tilde{t},j}$ from~(\ref{eq:Pij})
in~$P^{\times,\rho-1}(\F_p)$, and so we do not know for which $n$ and
$\eps$ the point $\eps{\cdot}D_{\tilde{t}}(n)$ is
in~$T(\Z)_t$. Luckily we have the $\Z_p^{\times,\rho-1}$-action
on~$T(\Z_p)$. Using that $\Z_p^\times=\F_p^\times\times (1+p\Z_p)$ we
have $\F_p^{\times,\rho-1}$ acting on $T(\Z_p)_{j_b(u)}$, compatibly
with the torsor structure on $T(\F_p)_{j_b(u)}$. So, for every $n$ in
$\Z^r$ there is a unique $\xi(n)$ in $\F_p^{\times,\rho-1}$ such that
$\xi(n){\cdot}D_{\tilde{t}}(n)$ is in~$T(\Z_p)_t$. We define
\begin{subeqn}\label{eq:Dprime}
D'(n):=\xi(n){\cdot}D_{\tilde{t}}(n)\,.
\end{subeqn}
Then for all $n$ in $\Z^r$,
\begin{subeqn}\label{eq:kappa_and_Dprime}
\kappa_\Z(n)=D_{\tilde{t}}((p-1){\cdot}n)=D'((p-1){\cdot}n) \,,
\end{subeqn}
because
$D_{\tilde{t}}((p-1){\cdot}n)$ maps to $t$ in~$T(\F_p)$.
Moreover for every $Q$ in $T(\Z)_t$ there is a unique $n\in\Z^r$ and
a unique $\eps\in\Z^{\times,\rho-1}$ such that
$Q=\eps{\cdot}D_{\tilde{t}}(n)=\xi(n){\cdot}D_{\tilde{t}}(n)=D'(n)$. Hence
\begin{subeqn}\label{eq:TZ_and_Dprime}
  T(\Z)_t\subset D'(\Z^r)\,.
\end{subeqn}
The following lemma proves the existence and uniqueness of the
$\kappa_i$ of Theorem~\ref{thm:p-adic_closure}, and the claims on the
degrees of the~$\ol{\kappa}_i$.

\begin{subLemma}\label{lem:A'C'D'}
After any choice of parameters of $\calO_{T,t}$
as in Theorem~\ref{thm:p-adic_closure}, $D'$ is given by elements
$\kappa_1',\ldots,\kappa_{g+\rho-1}'$ of
$\calO(\A^r_{\Z_p})^{\wedge_p}$, and then $\kappa_\Z$ is given by
$\kappa_1,\ldots,\kappa_{g+\rho-1}$ with, for all
$i\in\{1,\ldots,g+\rho-1\}$ and all $a\in\Z_p^r$,
\[
  \kappa_i(a)=\kappa_i'((p-1)a)\,.
\]
For all $i$ in $\{1,\dots,g+\rho-1\}$ we let $\ol{\kappa}_i'$ be the
reduction mod~$p$ of~$\kappa_j'$. Then 
$\ol{\kappa}_1',\ldots,\ol{\kappa}_g'$ are of degree at most~1,
and the remaining $\ol{\kappa}_j'$ are of degree at most~2.
\end{subLemma}
\begin{proof}
In order to get a formula for~$D'(n)$, we introduce variants of the
$P_{i,j}$, $R_{i,\wt{t}}$, and $S_{\wt{t},j}$ as
follows. The images in $(J\times (J^{\vee 0 })^{\rho -1})(\F_p)$ of
these points are of the form $(0,*)$, $(0,*)$, and $(*,0)$,
respectively. Hence the fibers over them of $P^{\times,\rho-1}$ are
rigidified, that is, equal to~$\F_p^{\times,\rho-1}$. We define their
variants $P_{i,j}'$, $R_{i,\wt{t}}'$, and
$S_{\wt{t},j}'$ in $P^{\times,\rho-1}(\Z_p)$ to be the unique
elements in their orbits under $\F_p^{\times,\rho-1}$ whose images in
$P^{\times,\rho-1}(\F_p)$ are equal to the element $1$
in~$\F_p^{\times,\rho-1}$. Replacing, in~(\ref{eq:A_and_B})
and~(\ref{eq:C}), these $P_{i,j}$, $R_{i,\wt{t}}$, and
$S_{\wt{t},j}$ by $P_{i,j}'$, $R_{i,\wt{t}}'$, and
$S_{\wt{t},j}'$ gives variants $A'$, $B'$ and~$C'$, and using
these in~(\ref{eq:D_ttilde_n}) gives a variant~$D_{\tilde{t}}'(n)$
of~\ref{eq:Dprime}.

Then, for all $n$ in $\Z^r$, $D_{\tilde{t}}'(n)$ and
$D'(n)$ (as in~(\ref{eq:Dprime})) are equal, because both are
in~$P^{\times,\rho-1}(\Z_p)_t$, and in the same
$\F_p^{\times,\rho-1}$-orbit. Hence we have, for all $n$ in~$\Z^r$:
\begin{subeqn}\label{eq:Dprime_formula}
\begin{aligned}
\displaystyle  A'(n) & =\sideset{}{_2}\sum_{j=1}^r n_j
\cdot_2 S_{\wt t, j}' \,, \quad 
\displaystyle B'(n) =
\sideset{}{_1}\sum_{i=1}^r n_i \cdot_1 R_{i,\wt t}'\,, \\
\displaystyle C' (n) & = \sideset{}{_1}\sum_{i=1}^r n_i \cdot_1 \left(
\sideset{}{_2}\sum_{j=1}^r n_j \cdot_2 P'_{i,j} \right)\,, \\
D'(n) & = \left(C'(n)+_2 B'(n)\right) +_1 \left(A'(n) +_2 \wt{t} \right)\,.
\end{aligned}
\end{subeqn}
This shows how the map $n\mapsto D'(n)$ is built up from the two
partial group laws $+_1$ and $+_2$ on~$P^{\times,\rho-1}$, and the
iterations ${\cdot}_1$ and~${\cdot}_2$. Lemma~\ref{lem:log_and_exp}
gives that the iterations are given by integral convergent power
series. The functoriality in Section~\ref{sec:formal_geometry} gives
that the maps induced by $+_1$ and $+_2$ on residue polydisks are
given by integral convergent power series. Stability under composition
then gives that $n\mapsto D'(n)$ is given by elements
$\kappa_1',\ldots,\kappa_{g+\rho-1}'$ of $\Z_p\langle
z_1,\ldots,z_r\rangle$. 

We call the $\kappa_i'$ the coordinate functions of the extension
$D'\colon\Z_p^r\to T(\Z_p)_t=\Z_p^{g+\rho-1}$, and their images
$\ol{\kappa}_1',\ldots,\ol{\kappa}_{g+\rho-1}'$ in
$\F_p[z_1,\ldots,z_r]$ the mod~$p$ coordinate functions, viewed as a
morphism $\ol{D}'_{\F_p}\colon\A^r_{\F_p}\to \A^{g+\rho-1}_{\F_p}$. 

The mod~$p$ coordinate functions of
$A'\colon\Z_p^r\to P^{\times,\rho-1}(\Z_p)=\Z_p^{\rho g+\rho-1}$
(after choosing the necessary parameters) are all of degree at
most~1. The same holds for~$B'$. We define 
\begin{subeqn}\label{eq:Cprime2}
C'_2\colon\Z^r\times\Z^r\lto P^{\times,\rho-1}(\Z_p),\quad 
\displaystyle C'_2 (n,m) = \sideset{}{_1}\sum_{i=1}^r n_i \cdot_1 \left(
\sideset{}{_2}\sum_{j=1}^r m_j \cdot_2 P'_{i,j} \right)\,.
\end{subeqn}
Then the mod~$p$ coordinate functions of~$C'_2$, elements of
$\F_p[x_1,\ldots,x_r,y_1,\ldots,y_r]$, are linear in the~$x_i$, and in
the~$y_j$. Hence of degree at most~2, and the same follows for the
mod~$p$ coordinate functions of~$C'$. However, as the first $\rho g$
parameters for $P^{\times,\rho-1}$ come from $J\times J^{\vee
  \rho-1}$, and the 1st and 2nd partial group laws there act on
different factors, the first $\rho g$ mod~$p$ coordinate functions of
$C'$ are in fact linear. As $D'$ is obtained by summing, using the
partial group laws, the results of~$A'$, $B'$ and~$C'$, we conclude
that $\ol{\kappa}_1',\ldots,\ol{\kappa}_g'$ are of degree at most~1,
and the remaining $\ol{\kappa}_j$ are of degree at most~2. The same
holds then for all~$\ol{\kappa}_j$.
\end{proof}

\subsection{The $p$-adic closure}\label{sec:p-adic_closure}
We know from~(\ref{eq:kappa_and_Dprime}) that 
$\kappa_\Z(\Z^r)=D'((p-1)\Z^r)$. From~(\ref{eq:kappa_Z}) we know that
$\kappa_\Z(\Z^r)\subset T(\Z)_t$. From~(\ref{eq:TZ_and_Dprime}) we
know that $T(\Z)_t\subset D'(\Z^r)$. So together we have:
\begin{subeqn}\label{eq:inclusions_kappa}
D'((p-1)\Z^r)=\kappa_\Z(\Z^r) \subset T(\Z)_t \subset D'(\Z^r)\,.
\end{subeqn}
We have extended $D'$ to a continuous map $\Z_p^r\to T(\Z_p)_t$. As
$\Z_p^r$ is compact, $D'(\Z_p^r)$ is closed in $T(\Z_p)_t$. As $\Z^r$
and $(p-1)\Z^r$ are dense in $\Z_p^r$, the closures of their images
under $D'$ are both equal to $D'(\Z_p^r)$, and equal to~$\kappa(\Z_p^r)$.
This finishes the proof of Theorem~\ref{thm:p-adic_closure}.

\section{Explicit description of the Poincar\'e torsor}
The aim of this section is to give explicit descriptions of the
Poincar\'e torsor $P^\times$ on $J\times J^{\vee,0}$ and its partial
group laws, to be used for doing computations when applying
Theorem~\ref{thm:finiteness}. The main results are as follows.
Proposition~\ref{prop_P_and_M_and_norm} describes the fibre of $P$
over a point of $J\times J^{\vee,0}$, say with values in $\Z/p^2\Z$
with $p$ not dividing~$n$ or in $\Z[1/n]$, when the corresponding
points of $J$ and $J^{\vee,0}$ are given by a line bundle on~$C$ (over
$\Z/p^2\Z$ or~$\Z[1/n]$, and rigidified at~$b$) and an effective
relative Cartier divisor on~$C$ (over $\Z/p^2\Z$ or~$\Z[1/n]$).  It
also translates the partial group laws of $P^\times$ in terms of such
data. Lemma~\ref{lem:norm_and_lin_equiv} shows how to deal with linear
equivalence of divisors. Lemma~\ref{lem:M_symmetry} makes the symmetry
of $P^\times$ explicit. Lemma~\ref{lem:parametrization_residue_disk_M}
gives parametrisations of residue polydisks of $P^\times(\Z/p^2\Z)$,
and Lemma~\ref{lem:part_group_laws_res_disks_M} gives partial group
laws on these residue polydisks. Proposition~\ref{prop:M_and_N_q}
describes the unique extension over $J\times J^{\vee,0}$ of the
Poincar\'e torsor on $(J\times J^{\vee,0})_{\Z[1/n]}$, in terms of
line bundles and divisors on~$C$. Finally,
Proposition~\ref{prop:P_over_Z-point} describes the fibres of $P$ over
$\Z$-points of~$J\times J^{\vee,0}$.

In this article, we have chosen to use line bundles and divisors on
curves for describing the jacobian and the Poincar\'e torsor. Another
option is to use line bundles on curves and the determinant of
coherent cohomology, as in Section~2 of~\cite{M-B_MP}. We note that in Section~2, only the restriction of $P$ to $J^0\times J^{\vee,0}$
is treated, and moreover, under the assumption that $C$ is nodal (that
is, all fibres $C_{\F_p}$ are reduced and have only the mildest
possible singularities). Another choice we have made is to develop the
basic theory of norms of $\Gm$-torsors under finite locally free
morphisms in this article
(Sections~\ref{sec:norms1}--\ref{sec:norms2}) and not to refer, for
example, to EGA or SGA, because we think this is easier for the
reader, and because this way we could adapt the definition directly to
our use of it.

\subsection{Norms}\label{sec:norms1}
Let $S$ be a scheme, $f\colon S'\to S$ be finite and locally free, say
of rank~$n$. Then $\calO_{S'}=f_*\calO_{S'}$ (we view $\calO_{S'}$ as
a sheaf on~$S$) is an $\calO_S$-algebra, locally free as
$\calO_S$-module of rank~$n$, and $\calO_{S'}^\times$ is a subsheaf of
groups of $\GL_{\calO_S}(\calO_{S'})$. Then the norm morphism is the
composition
\begin{subeqn}\label{eq:def_norm_groups}
  \begin{tikzcd}
    \calO_{S'}^\times \arrow[bend left]{rr}{\Norm_{S'/S}}  \arrow[r,hook] &
    \GL_{\calO_S}(\calO_{S'})\arrow{r}{\det} & \calO_S^\times \,.
\end{tikzcd}  
\end{subeqn}
Our viewing of $\calO_{S'}$ as a sheaf on~$S$ does not change the
notion of $\calO_{S'}^\times$-torsor, because of the equivalence with invertible
$\calO_{S'}$-modules: triviality locally on $S'$ implies triviality
locally on~$S$.

For $T$ an $\calO_{S'}^\times$-torsor, we let $\Norm_{S'/S}(T)$ be the 
$\calO_S^\times$-torsor
\begin{subeqn}\label{eq:def_norm_torsor}
  \Norm_{S'/S}(T):=\calO_S^\times\otimes_{\calO_{S'}^\times}T =
  \left(\calO_S^\times\times T\right)/\calO_{S'}^\times\,,  
\end{subeqn}
with, for every open $U$ of $S$, and every $u\in
\calO_{S'}^\times(U)$,  $u$ acting as $(v,t)\mapsto
(v{\cdot}\Norm_{S'/S}(u),u ^{-1}{\cdot}t)$. 
This is functorial in~$T$: a morphism $\phi\colon T_1\to T_2$ induces
an isomorphism~$\Norm_{S'/S}(\phi)$. It is also functorial for
cartesian diagrams $(S_2'\to S_2)\to (S_1'\to S_1)$.

For $U\subset S$ open, $T$ an $\calO_{S'}^\times$-torsor, and $t\in
T(U)$, we have the isomorphism of $\calO_{S'}^\times|_U$-torsors
$\calO_{S'}^\times|_U\to T|_U$ sending $1$ to~$t$. Functoriality gives
$\Norm_{S'/S}(t)$ in $(\Norm_{S'/S}(T))(U)$, also denoted~$1\otimes t$.

The norm functor~(\ref{eq:def_norm_torsor}) is multiplicative:
\begin{subeqn}\label{eq:norm_mult}
  \Norm_{S'/S}(T_1\otimes_{\calO_{S'}} T_2)=
  \Norm_{S'/S}(T_1)\otimes_{\calO_S}\Norm_{S'/S}(T_2)\,,
\end{subeqn}
such that, if $U\subset S$ is open and $t_1$ and $t_2$ are in $T_1(U)$
and $T_2(U)$, then
\begin{subeqn}\label{eq: norm_mult_elmts}
\Norm_{S'/S}(t_1\otimes t_2)\mapsto\Norm_{S'/S}(t_1)\otimes
\Norm_{S'/S}(t_2)\,.
\end{subeqn}

Let $\calL$ be an invertible $\calO_{S'}$-module; locally on~$S$, it
is free of rank $1$ as $\calO_{S'}$-module. This gives us the
$\calO_{S'}^\times$-torsor (on~$S$)
$\Isom_{\calO_{S'}}(\calO_{S'},\calL)$, which gives back $\calL$ as
$\calL=\calO_{S'}\otimes_{\calO_{S'}^\times}\Isom_{\calO_{S'}}(\calO_{S'},\calL)$. The
norm of $\calL$ via $f\colon S'\to S$ is then defined as
\begin{subeqn}\label{eq:def_norm_inv_sheaf}
  \Norm_{S'/S}(\calL) :=
  \calO_S\otimes_{\calO_S^\times}\Norm_{S'/S}(\Isom_{\calO_{S'}}(\calO_{S'},\calL))\,.
\end{subeqn}
This construction is functorial for isomorphisms of invertible
$\calO_{S'}$-modules.

\subsection{Norms along finite relative Cartier divisors}\label{sec:norms2}
This part is inspired by~\cite{Ka-Ma}, section~1.1.  Let $S$ be a
scheme, let $f\colon X\to S$ be an $S$-scheme of finite
presentation. A finite effective relative Cartier divisor on
$f\colon X\to S$ is a closed subscheme $D$ of $X$ that is finite and
locally free over~$S$, and whose ideal sheaf~$I_D$ is locally
generated by a non-zero divisor (equivalently, $I_D$ is locally free
of rank~$1$ as $\calO_X$-module). For such a $D$ and an invertible
$\calO_X$-module~$\calL$, the norm of $\calL$ along $D$ is defined,
using~(\ref{eq:def_norm_inv_sheaf}), as
\begin{subeqn}\label{eq:def_norm_along_eff_divisor}
\Norm_{D/S}(\calL) := \Norm_{D/S}(\calL|_D)\,.
\end{subeqn}
Then $\Norm_{D/S}(\calL)$ is functorial for cartesian diagrams $(X'\to
S',\calL')\to(X\to S,\calL)$.  

\begin{subLemma}\label{lem_norms}
  Let $f\colon X\to S$ be a morphism of schemes that is of finite
  presentation. For $D$ a finite effective relative Cartier divisor
  on~$f$, the norm functor $\Norm_{D/S}$
  in~(\ref{eq:def_norm_along_eff_divisor}) is multiplicative
  in~$\calL$:
\begin{subeqn}\label{eq:norm_add_L}
\Norm_{D/S}(\calL_1\otimes\calL_2) =
\Norm_{D/S}(\calL_1)\otimes_{\calO_S}\Norm_{D/S}(\calL_2)\,,  
\end{subeqn}
with, for $U\subset S$ open, $V\subset X$ open, containing $f^{-1}U\cap
D$ and $l_i\in\calL_i(V)$  generating $\calL_i|_V$,
\begin{subeqn}\label{eq:norm_add_L_elements}
\Norm_{D/S}(l_1\otimes l_2) =
\Norm_{D/S}(l_1)\otimes\Norm_{D/S}(l_2)\,.
\end{subeqn}
Let $D_1$ and $D_2$ be finite effective relative Cartier divisors
on~$f$. Then the ideal sheaf $I_{D_1}I_{D_2}\subset\calO_X$ is locally
free of rank~$1$, the closed subscheme $D_1+D_2$ defined by it is a
finite effective relative Cartier divisor on~$f$. The norm functor
in~(\ref{eq:def_norm_along_eff_divisor}) is additive in~$D$:
\begin{subeqn}\label{eq:norm_add_in_D}
\Norm_{(D_1+D_2)/S}(\calL) =
\Norm_{D_1/S}(\calL)\otimes_{\calO_S}\Norm_{D_2/S}(\calL)\,,
\end{subeqn}
with, for $U\subset S$ open, $V\subset X$ open, containing $f^{-1}U\cap
(D_1+D_2)$ and $l\in\calL(V)$  generating $\calL|_{D_1+D_2}$,
\begin{subeqn}
\label{eq:norm_add_in_D_elements}
\Norm_{(D_1+D_2)/S}(l) =
\Norm_{D_1/S}(l)\otimes\Norm_{D_2/S}(l)\,.
\end{subeqn}
\end{subLemma}
\begin{proof}
Let $D_1$ and $D_2$ be as stated. If $V\subset X$ is open, and $f_i$
generates $I_{D_i}|_V$, then $f_1f_2$ generates
$(I_{D_1}I_{D_2})|_V$, and this element of $\calO_X(V)$ is not a
zero-divisor because $f_1$ and $f_2$ are not. To show that $D_1+D_2$
is finite over~$S$, we replace $S$ by an affine open of it, and then
reduce to the noetherian case, using the assumption that $f$ is of
finite presentation. Then, $(D_1+D_2)_\red$ is the image of
$D_{1,\red}\coprod D_{2,\red}\to X$, and therefore is proper. Hence
$D_1+D_2$ is proper over $S$, and quasi-finite over~$S$, hence
finite over~$S$.  The short exact sequence
\begin{subeqn}
\label{eq:norm_s_e_s}
\begin{tikzcd}
I_{D_2}/I_{D_1+D_2} \arrow[d,equals] \arrow[r, hook] & \calO_{D_1+D_2}
\arrow[r, two heads] & \calO_{D_2} \\
(I_{D_2})|_{D_1}
\end{tikzcd}
\end{subeqn}
shows that $\calO_{D_1+D_2}$ is locally free as $\calO_S$-module, of
rank the sum of the ranks of the~$\calO_{D_i}$. So $D_1+D_2$ is a
finite effective relative Cartier divisor on $X\to S$.

We prove~(\ref{eq:norm_add_in_D}), by proving the required statement
about sheaves of groups. The diagram
\begin{subeqn}\label{eq:norm_comm_diag}
\begin{tikzcd}
  \calO_{D_1+D_2}^\times \arrow[bend left]{rrrrr}{\Norm_{(D_1+D_2)/S}} \arrow[r]
  & \calO_{D_1}^\times\times\calO_{D_2}^\times
  \arrow{rrr}{\Norm_{D_1/S}\times\Norm_{D_2/S}} & & & \calO_S ^\times\times\calO_S ^\times \arrow{r}{\cdot} & \calO_S ^\times  \\
u \arrow[rrrrr, mapsto] & & & & & \Norm_{D_1/S}(u)\Norm_{D_2/S}(u) \,.
\end{tikzcd}
\end{subeqn}
commutes, because multiplication by $u$ on $\calO_{D_1+D_2}$ preserves
the short exact sequence~(\ref{eq:norm_s_e_s}), multiplying on the sub
and quotient by its images in $\calO_{D_1}^\times$ and in
$\calO_{D_2}^\times$; note that the sub is an invertible
$\calO_{D_1}$-module.
\end{proof}

\subsection{Explicit description of the Poincar\'e torsor of a smooth curve}

Let $g$ be in $\Z_{\geq1}$, let $S$ be a scheme, and
$\pi\colon C\to S$ be a proper smooth curve, with geometrically
connected fibres of genus~$g$, with a section~$b\in C(S)$. Let
$J\to S$ be its jacobian. On $C\times_SJ$ we have~$\calL^\univ$, the
universal invertible $\calO$-module of degree zero on~$C$, rigidified
at~$b$.

Let $d\geq 0$, and $C^{(d)}$ the $d$th symmetric power of $C\to S$ (we
note that the quotient $C^d\to C^{(d)}$ is finite, locally free of
rank~$d!$, and commutes with base change on~$S$). Then on
$C\times_S C^{(d)}$ we have~$D$, the universal effective relative
Cartier divisor on~$C$ of degree~$d$. Hence, on
$C\times_SJ\times_SC^{(d)}$ we have their pullbacks $D_J$
and~$\calL^\univ_{C^{(d)}}$, giving us
\begin{subeqn}\label{eq:univ_norm_bundle}
\calN_d := \Norm_{D_J/(J\times_S C^{(d)})}(\calL^\univ_{C^{(d)}})\,.  
\end{subeqn}
This invertible $\calO$-module~$\calN_d$ on $J\times_S C^{(d)}$,
rigidified at the zero-section of~$J$, gives us a morphism of
$S$-schemes $C^{(d)}$ to~$\Pic_{J/S}$. The point $db$ (the divisor $d$
times the base point~$b$) in $C^{(d)}(S)$ is mapped to~$0$, precisely
because $\calL^\univ$ is rigidified at~$b$, and~\ref{eq:norm_add_in_D}. Hence
there is a unique morphism $\Box\colon C^{(d)}\to J^\vee=\Pic^0_{J/S}$
such that the pullback of the Poincar\'e bundle $P$ on
$J\times J^\vee$ by $(\id,\Box)\colon J\times C^{(d)}\to J\times J^\vee$,
with its rigidifications, is the same as~$\calN_d$. The following
proposition tells us what the morphism~$\Box$ is, and the next section
tells us what the induced isomorphism is between the fibres of $\calN_d$
at points of $J\times C^{(d)}$ with the same image in~$J\times_S J$.  

\begin{subProposition}\label{prop_P_and_M_and_norm}
  The pullback of $P$ by
  $(j_b,j_b^{*,-1})\colon C\times_S J\to J\times_S J^\vee$ together
  with its rigidifications at $b$ and~$0$, is equal to~$\calL^\univ$.
  
  Let $d$ be in~$\Z_{\geq0}$. The morphism
  $\Box\colon C^{(d)}\to J^\vee=\Pic^0_{J/S}$ is the composition of first
  $\Sigma\colon C^{(d)}\to J$, sending, for every $S$-scheme~$T$, each point
  $D \in C^{(d)}(T)$ to the class of $\calO_{C_T}(D-db)$ twisted
  by the pullback from $T$ that makes it rigidified at~$b$, followed
  by~$j_b^{*,-1}\colon J\to J^\vee$.  Summarised in a diagram, with
  $\calM:=(\id\times j_b^{*,-1})^*P$:
\begin{subeqn}\label{eq:P_and_M_and_norm}
\begin{tikzcd}
\calL^\univ && P \arrow{ll} \arrow{rr} & & \calM
\arrow{rr}{\wt{\id\times\Sigma}} & & \calN_d \\ 
C\times_S J \arrow{rr}{j_b\times j_b^{*,-1}} && J \times_S J^\vee & & 
J\times_S J \arrow{ll}[swap]{\id\times j_b^{*,-1}} & &
J\times_S C^{(d)} \arrow{ll}[swap]{\id\times\Sigma}\,.
\end{tikzcd}      
\end{subeqn}
Then $\calM$, with its rigidifications
at $\{0\}\times_SJ$ and $J\times_S\{0\}$, is symmetric.
For $T\to S$, $x$ in $J(T)$ given by an invertible $\calO$-module  $\calL$ on
$C_T$ rigidified at~$b$, and $y=\Sigma(D)$ in $J(T)$ given by an
effective relative divisor $D$ of degree $d$ on~$C_T$ we have
\begin{subeqn}\label{eq:P_and_M_and_N_on_points}
P\left(x,j_b^{*,-1}(y)\right) = \calM(x,y) = \Norm_{D/T}(\calL)\,.  
\end{subeqn}
For $c_1$ and $c_2$ in~$C(S)$, we have
\begin{subeqn}\label{eq:M_at_c1_and_c2}
\calM\left(j_b(c_1),j_b(c_2)\right) =
c_2^*\left(\calO_C(c_1-b)\right)\otimes b^*\left(\calO_C(b-c_1)\right)\,,
\end{subeqn}
and, as invertible $\calO$-modules on $C\times_SC$, with $\Delta$ the
diagonal and $\pr_{\emptyset}\colon C\times_SC\to S$ the structure
morphism, we have 
\begin{subeqn}\label{eq:jb_jb_pullback_M}
(j_b\times j_b)^*\calM =
\calO(\Delta)\otimes\pr_1^*\calO(-b)\otimes\pr_2^*\calO(-b)
\otimes\pr_{\emptyset}^*b^*T_{C/S}\,. 
\end{subeqn}
For $d>2g-2$, $\wt{\id\times\Sigma}$ gives $\calN_d$ a descent
datum along $\id\times\Sigma$ that gives~$\calM$ on~$J\times_S J$.
For $T$ an $S$-scheme, $x\in J(S)$ given by $\calL$ on $C_T$,
rigidified at~$b$, $D_1$ and $D_2$ in $C^{(d_1)}(S)$
and~$C^{(d_2)}(S)$, the isomorphism
\begin{subeqn}\label{eq:plus_2_M}
\calM(x,\Sigma(D_1+D_2)) =   \calM(x,\Sigma(D_1))\otimes \calM(x,\Sigma(D_2))
\end{subeqn}
corresponds, via $\wt{\id\times\Sigma}$, to
\begin{subeqn}\label{eq:plus_2_N}
\begin{aligned}
  \calN_{d_1+d_2}(x,D_1+D_2) & = \Norm_{(D_1+D_2)/T}(\calL) =
  \Norm_{D_1/T}(\calL)\otimes \Norm_{D_2/T}(\calL)\\
  & = \calN_{d_1}(x,D_1)\otimes\calN_{d_2}(x,D_2)\,,
\end{aligned}
\end{subeqn}
using Lemma~\ref{lem_norms}.

For $T$ an $S$-scheme and $x_1$ and $x_2$ in $J(T)$ given by
$\calO$-modules $\calL_1$ and $\calL_2$ on~$C_T$, rigidified at~$b$,
and $D$ in $C^{(d)}(T)$, the isomorphism
\begin{subeqn}\label{eq:plus_1_M}
\calM(x_1+x_2,\Sigma(D)) =   \calM(x_1,\Sigma(D))\otimes \calM(x_2,\Sigma(D))
\end{subeqn}
corresponds, via $\wt{\id\times\Sigma}$, to
\begin{subeqn}\label{eq:plus_1_N}
\begin{aligned}
  \calN_d(x_1+x_2,D) & = \Norm_{D/T}(\calL_1\otimes\calL_2) =
  \Norm_{D/T}(\calL_1)\otimes\Norm_{D/T}(\calL_2) \\
  & = \calN_d(x_1,D)\otimes\calN_d(x_2,D)\,,
\end{aligned}
\end{subeqn}
using Lemma~\ref{lem_norms}.
\end{subProposition}
\begin{proof}
  Let $T$ be an $S$-scheme, and $x$ be in~$J(T)$. Then $x$ corresponds
  to the invertible $\calO$-module $(\id\times x)^*\calL^\univ$
  on~$C_T$, rigidified at~$b$. Let $z:=j_b^{*,-1}(x)$
  in~$J^\vee(T)$. Then $j_b^*(z)=x$, meaning that the pullback of
  $(\id\times z)^*P$ on $J_T$ rigidified at~$0$ by $j_b$ equals
  $(\id\times x)^*\calL^\univ$ on~$C_T$ rigidified at~$b$. Taking
  $T:=J$ and $x$ the tautological point gives the first claim of the
  proposition.

  The symmetry of $\calM$ with its rigidifications follows from
  \cite{M-B_MP}, (2.7.1) and Lemma~2.7.5, and~(2.7.7),
  using~\ref{eq:lambda_and_jb}. 

  Now we prove~(\ref{eq:P_and_M_and_N_on_points}). So let $T$ and $x$
  be as above, and $y=\Sigma(D)$ in $J(T)$ given by a relative divisor
  $D$ of degree $d$ on~$C_T$. As $C^d\to C^{(d)}$ is finite and
  locally free of rank~$d!$, we may and do suppose that $D$ is a sum of
  sections, say $D=\sum_{i=1}^d(c_i)$, with $c_i\in C(T)$. Then we
  have, functorially:
  \begin{subeqn}\label{eq:P_and_norm}
\begin{aligned}
P(x,j_b^{*,-1}(y)) & = P(y, j_b^{*,-1} (x)) = P(\Sigma(D),
j_b^{*,-1}(x))\\
& =  P\left(\sum_i j_b(c_i), j_b^{*,-1}(x)\right) 
= \bigotimes_i P(j_b(c_i), j_b^{*,-1}(x))\\
& = \bigotimes_i \calL^\univ(c_i, x) = \bigotimes_i\calL(c_i) =
\Norm_{D/T}(\calL)\,. 
\end{aligned}        
\end{subeqn}
Identities~(\ref{eq:M_at_c1_and_c2}) and~(\ref{eq:jb_jb_pullback_M})
follow directly from~(\ref{eq:P_and_M_and_N_on_points}).

Now we prove the claimed compatibility between (\ref{eq:plus_1_M})
and~(\ref{eq:plus_1_N}).  We do this by considering the case where
$\calL$ is universal, that is, base changing to $J_T$ and $x$ the universal
point. Then, on~$J_T$, we have 2 isomorphisms from
$\Norm_{(D_1+D_2)/J_T}(\calL)$ to
$\Norm_{D_1/J_T}(\calL)\otimes \Norm_{D_2/J_T}(\calL)$. These differ by an
element of $\calO(J_T)^\times=\calO(T)^\times$. Hence it suffices to
check that this element equals $1$ at~$0\in J(T)$. This amounts to
checking that the 2 isomorphisms are equal for $\calL=\calO_{C_T}$ with
the standard rigidification at~$b$. Then, both isomorphisms are the
multiplication map $\calO_T\otimes_{\calO_T}\calO_T\to\calO_T$.

The compatibility between (\ref{eq:plus_2_M}) and~(\ref{eq:plus_2_N})
is proved analogously.
\end{proof}

\begin{subRemark}\label{rem_no_base_point_deg_0}
From Proposition~\ref{prop_P_and_M_and_norm} one easily deduces,
in that situation, for $T$ an $S$-scheme, $x$ in $J(T)$ given by an invertible
$\calO$-module $\calL$ on~$C_T$, and $D_1$ and $D_2$ effective relative
Cartier divisors on~$C_T$, of the same degree, a canonical isomorphism
\begin{subeqn}\label{eq:M_and_N_no_base}
\calM(x,\Sigma(D_1)-\Sigma(D_2)) = \Norm_{D_1/T}(\calL)\otimes \Norm_{D_2/T}(\calL)^{-1}\,,
\end{subeqn}
satisfying the analogous compatibilities as in
Proposition~\ref{prop_P_and_M_and_norm}. No rigidification of $\calL$
at~$b$ is needed. In fact, for $\calL_0$ an invertible
$\calO_T$-module, we have
$\Norm_{D_1/T}(\pi^*\calL_0)=\calL_0^{\otimes d}$, where $\pi\colon
C_T\to T$ is the structure morphism and $d$ is the degree
of~$D_1$. Hence the right hand side of~(\ref{eq:M_and_N_no_base}) is
independent of the choice of~$\calL$, given~$x$.
\end{subRemark}

\subsection{Explicit isomorphism for norms along equivalent
  divisors} \label{sec:iso_norm_along_equiv_D} Let $g$ be in
$\Z_{\geq1}$, let $S$ be a scheme, and $p\colon C\to S$ be a proper
smooth curve, with geometrically connected fibres of genus~$g$, with a
section~$b\in C(S)$.  Let $D_1,D_2$ be effective relative Cartier
divisors of degree $d$ on~$C$, that we also view as elements
of~$C^{(d)}(S)$. Recall from Proposition~\ref{prop_P_and_M_and_norm}
the morphism $\Sigma\colon C^{(d)}\to J$. Then
$\Sigma(D_1)=\Sigma(D_2)$ if and only if $D_1,D_2$ are linearly
equivalent in the following sense: locally on $S$, there exists an $f$
in $\calO_C(U)^\times$, with $U:=C\setminus(D_1 \cup D_2)$, such that
$f{\cdot}\colon\calO_U\to\calO_U$ extends to an isomorphism
$f{\cdot}\colon\calO_C(D_1)\to\calO_C(D_2)$. In this case, we define
$\divisor(f) = D_2-D_1$.  Proposition~\ref{prop_P_and_M_and_norm}
gives us, for each invertible $\calO$-module $\calL$ of degree $0$ on
$C$ rigidified at~$b$ (viewed as an element of~$J(S)$) specific
isomorphisms
\begin{subeqn} \label{eq:norm_iso_equiv_div_abstract}
\begin{aligned}
\Norm_{D_1/S}(\calL) & = \calN_d(\calL, D_1) =  \calM(\calL, \Sigma(D_1))
= \calM(\calL, \Sigma(D_2)) = \calN_d(\calL, D_2)\\
& =\Norm_{D_2/S}(\calL)\,.
\end{aligned}
\end{subeqn}
Now we describe explicitly this isomorphism $\Norm_{D_1/S}(\calL)\to
\Norm_{D_2/S}(\calL)$. To do so we first describe \emph{an} isomorphism 
\begin{subeqn}\label{eq:isom_norm_equiv_div}
\phi_{\calL, D_1,D_2}\colon \Norm_{D_1/S}(\calL) \lto \Norm_{D_2/S}(\calL)
\end{subeqn}
that is functorial for Cartesian diagrams
$(C'\to S',\calL', D_1', D_2')\to(C\to S,\calL,D_1,D_2)$ and then we
prove that \emph{this} isomorphism is the one
in~(\ref{eq:norm_iso_equiv_div_abstract}).

We construct $\phi_{\calL,D_1,D_2}$ locally on $S$ and the
functoriality of the construction takes care of making it global. So,
suppose that $f$ is as above: $f\in\calO_C(U)^\times$, and
$f{\cdot}\colon\calO_U\to\calO_U$ extends to an isomorphism
$f{\cdot}\colon\calO_C(D_1)\to\calO_C(D_2)$. Let $n\in\Z$ with
$n>2g-2+2d$. Then $p_*(\calL(nb))\to p_*\calL(nb)|_{D_1+D_2}$ and
$p_*(\calO_C(nb))\to p_*\calO_C(nb)|_{D_1+D_2}$ are surjective, and
(still localising on~$S$) $p_*(\calL(nb))$ and $p_*(\calO_C(nb))$ are
free $\calO_S$-modules and $\calL(nb)|_{D_1+D_2}$ and
$\calO_C(nb)|_{D_1+D_2}$ are free $\calO_{D_1+D_2}$-modules of
rank~1. Then we have $l_0$ in $(\calL(nb))(C)$ and $l_1$ in
$(\calO_C(nb))(C)$ restricting to generators on $D_1+D_2$. Let
$D^-:=\divisor(l_1)$ and $D^+:=\divisor(l_0)$, and let
$V:=C\setminus(D^++D^-)$. Note that $V$ contains $D_1+D_2$ and that
$U$ contains $D^++D^-$. Then, on~$V$, $l:=l_0/l_1$ is in $\calL(V)$,
generates $\calL|_{D_1+D_2}$, and multiplication by~$l$ is an
isomorphism ${\cdot}l\colon\calO_C(D^+-D^-)\to\calL$, that is,
$\divisor(l)=D^+-D^-$. Let
\begin{subeqn}\label{eq:f_div_l}
f(\divisor(l)) = f(D^+ - D^-):= \Norm_{D^+/S}(f|_{D^+})\cdot
\Norm_{D^-/S}(f|_{D^-})^{-1} \in \calO_S(S)^\times \,,
\end{subeqn}
and let $\phi_{\calL, l,f}$ be the isomorphism, given in terms of generators
\begin{subeqn}\label{eq:phi_dirty}
\begin{aligned}
  \phi_{\calL,l,f}\colon \Norm_{D_1/S}(\calL) & \lto
  \Norm_{D_2/S}(\calL) \\
  \Norm_{D_1/S}(l) &\longmapsto f(\divisor(l))^{-1} \cdot
  \Norm_{D_2/S}(l) \,.
\end{aligned} 
\end{subeqn}
Now suppose that we made other choices $n'$, $l_0'$,~$l_1'$. Then we
get ${D^-}'$, ${D^+}'$, $V'$, $l'$ and~$\phi_{\calL, l',f}$. Then
there is a unique function $g\in \calO_C(V\cap V')^\times$ such that
$l' = gl$ in $\calL(V\cap V')$. Then 
\begin{subeqn}\label{eq:phi_indep_l}
\begin{aligned}
\phi_{\calL,l',f}(\Norm_{D_1/S}(l)) & =
\phi_{\calL,l',f}(\Norm_{D_1/S}(g^{-1}l')) \\
& = \phi_{\calL,l',f}(g^{-1}(D_1)\Norm_{D_1/S}(l')) \\
& = g^{-1}(D_1){\cdot}\phi_{\calL,l',f}(\Norm_{D_1/S}(l')) \\
& = g^{-1}(D_1){\cdot}f(\divisor(l'))^{-1}{\cdot}\Norm_{D_2/S}(l') \\
& = g^{-1}(D_1){\cdot}f(\divisor(gl))^{-1}{\cdot}\Norm_{D_2/S}(gl) \\
& =
g^{-1}(D_1){\cdot}f(\divisor(g)+\divisor(l))^{-1}{\cdot}g(D_2){\cdot}\Norm_{D_2/S}(l)
\\
& =
g^{-1}(D_1){\cdot}f(\divisor(g))^{-1}{\cdot}g(D_2){\cdot}f(\divisor(l))^{-1}{\cdot}\Norm_{D_2/S}(l)\\
& =
g(\divisor(f)){\cdot}f(\divisor(g))^{-1}{\cdot}\phi_{\calL,l,f}(\Norm_{D_1/S}(l))\\
& = \phi_{\calL,l,f}(\Norm_{D_1/S}(l))\,,
\end{aligned}
\end{subeqn}
where, in the last step, we used Weil reciprocity, in a generality for
which we do not know a reference. The truth in this generality is
clear from the classical case by reduction to the universal case, in
which the base scheme is integral: take a suitable level structure
on~$J$, then consider the universal curve with this level structure,
and the universal 4-tuple of effective divisors with the necessary
conditions.  We conclude that $\phi_{\calL,l,f}=\phi_{\calL,l',f}$.

Now suppose that $f'$ is in $\calO_C(U)^\times$ with
$\divisor(f') = \divisor(f)$.  Then there is a unique
$u \in \calO_S(S)^\times$ such that $f' = u{\cdot}f$, and since $\calL$
has degree~$0$ on~$C$
\begin{subeqn}\label{eq:phi_indep_f}
\begin{aligned}
\phi_{\calL,l,f'}\left(\Norm_{D_1/S}(l)\right) &=
(u{\cdot}f)(\divisor(l))^{-1} {\cdot} \Norm_{D_2/S}(l) \\
& = u^{-\deg(\divisor(l))}f(\divisor(l))^{-1} {\cdot} \Norm_{D_2/S}(l) \\ 
& = f(\divisor(l))^{-1} {\cdot} \Norm_{D_2/S}(l)  = \phi_{\calL,
  l,f}\left(\Norm_{D_1/S}(l)\right) \,.
\end{aligned}
\end{subeqn}
Hence $\phi_{\calL,l,f'} = \phi_{\calL, l,f}$. We define
\begin{subeqn}\label{eq:explicit_iso_equiv_div}
  \phi_{D_1,D_2,\calL}\colon \Norm_{D_1/S}(\calL) \lto \Norm_{D_2/S}(\calL)
\end{subeqn}
as the isomorphism $\phi_{\calL,l,f}$ in~(\ref{eq:phi_dirty}) for any
local choice of~$f$ and~$l$.

\begin{subLemma}\label{lem:norm_and_lin_equiv}
With the assumptions as in the beginning of
Section~\ref{sec:iso_norm_along_equiv_D}, the isomorphism
$\phi_{\calL,D_1,D_2}$ in~(\ref{eq:explicit_iso_equiv_div}) is equal to the
isomorphism in~(\ref{eq:norm_iso_equiv_div_abstract}). 
\end{subLemma}
\begin{proof}
We do this, as
in the proof of Proposition~\ref{prop_P_and_M_and_norm}, by
considering the case of the universal~$\calL$, that is, we base change
via $J\to S$, and then restricting to~$0\in J(S)$. This amounts to
checking that the 2 isomorphisms are equal for $\calL=\calO_C$ with
the standard rigidification at~$b$. In this case,
$\Norm_{D_i/S}(\calO_C)=\calO_S$, with $\Norm_{D_i/S}(1)=1$. Hence
$\phi_{D_1,D_2,\calO_C}=\phi_{\calO_C,1,f}$ is the identity
on~$\calO_S$ (use~(\ref{eq:phi_dirty})). The other isomorphism is the
identity on $\calO_S$ because of the rigidifications of $\calM$ and
$\calN_d$ on $0\times J$ and~$0\times C^{(d)}$.
\end{proof}

\subsection{Symmetry of the Norm for divisors on smooth
  curves}\label{sec:symmetry_norms} 
Let $C\to S$ be a proper and smooth curve with geometrically connected
fibres. For $D_1$, $D_2$ effective relative Cartier divisors on~$C$ we
define an isomorphism
\begin{subeqn}\label{eq:norm_symm_isom}
\phi_{D_1,D_2}\colon \Norm_{D_1/S}(\calO_C(D_2)) \lto \Norm_{D_2/S}(\calO_C(D_1))  
\end{subeqn}
that is functorial for cartesian diagrams $(C'/S',D'_1,D'_2)\to (C/S,D_1,D_2)$. 

If suffices to define this isomorphism in the universal case, that is,
over the scheme that parametrises all $D_1$ and~$D_2$. Let $d_1$ and
$d_2$ be in $\Z_{\geq0}$, and let $U:=C^{(d_1)}\times_S C^{(d_2)}$,
and let $D_1$ and $D_2$ be the universal divisors on~$C_U$. Then we
have the invertible $\calO_U$-modules $\Norm_{D_1/U}(\calO_C(D_2))$
and $\Norm_{D_2/U}(\calO_C(D_1))$. The image of $D_1\cap D_2$ in $U$
is closed, let $U^0$ be its complement. Then, over~$U^0$, $D_1$ and
$D_2$ are disjoint, and the restrictions of
$\Norm_{D_1/U}(\calO_C(D_2))$ and $\Norm_{D_2/U}(\calO_C(D_1))$ are
generated by $\Norm_{D_1/U}(1)$ and $\Norm_{D_2/U}(1)$, and there is a
unique isomorphism $(\phi_{D_1,D_2})_{U^0}$ that sends $\Norm_{D_1/U}(1)$
to $\Norm_{D_2/U}(1)$.

We claim that this isomorphism extends to an isomorphism over~$U$. To
see it, we base change by $U'\to U$, where $U'=C^{d_1}\times_S
C^{d_2}$, then $U'\to U$ is finite, locally free of
rank~$d_1!{\cdot}d_2!$. Then $D_1=P_1+\cdots+P_{d_1}$ and
$D_2=Q_1+\cdots+Q_{d_2}$ with the $P_i$ and $Q_j$ in~$C(U')$. The
complement of the inverse image $U'^0$ in $U'$ of $U^0$ is the union
of the pullbacks $D_{i,j}$ under $\pr_{i,j}\colon U'\to C\times_S C$ of the
diagonal, that is, the locus where $P_i=Q_j$. Each $D_{i,j}$ is an
effective relative Cartier divisor on~$U'$, isomorphic as $S$-scheme
to $C^{d_1+d_2-1}$, hence smooth over~$S$. Now
\begin{subeqn}\label{eq:norm_Pi_Qj}
  \Norm_{D_1/U'}(\calO(D_2)) = \bigotimes_{i,j}P_i^*\calO(Q_j)\,,\quad
  \Norm_{D_2/U'}(\calO(D_1)) = \bigotimes_{i,j}Q_j^*\calO(P_i)\,,
\end{subeqn}
and, on~$U'^0$,
\begin{subeqn}\label{eq:norm_of_1_is_1}
  \Norm_{D_1/U'}(1) =  \bigotimes_{i,j} 1\,,\quad
\Norm_{D_2/U'}(1) =  \bigotimes_{i,j} 1\,,
  \quad\text{in $\calO(U'^0)$.}
\end{subeqn}
The divisor on $U'$ of the tensor-factor $1$ at $(i,j)$, both in
$\Norm_{D_1/U'}(1)$ and in $\Norm_{D_2/U'}(1)$,
is~$D_{i,j}$. Therefore, the isomorphism $(\phi_{D_1,D_2})_{U^0}$
extends, uniquely, to an isomorphism $\phi_{D_1,D_2}$ over $U'$, which
descends uniquely to~$U$.

Our description of $\phi_{D_1,D_2}$ allows us to compute it in the
trivial case where $D_1$ and $D_2$ are disjoint. One should be a bit
careful in other cases. For example, when $d_1=d_2=1$ and $P=Q$, we
have $P^*\calO_C(Q)=P^*\calO_C(P)$ is the tangent space of $C\to S$
at~$P$, and hence also at~$Q$, but $\phi_{P,Q}$ is multiplication
by~$-1$ on that tangent space. The reason for that is that the switch
automorphism on $C\times_S C$ induces $-1$ on the normal bundle of the
diagonal.

\begin{subLemma}\label{lem:M_symmetry}
Let $b$ be an $S$-point on~$C$. Because of the symmetry in
Proposition~\ref{prop_P_and_M_and_norm},
using~(\ref{eq:M_and_N_no_base}), for $D_1$, $D_2$ relative effective
divisors on $C$ of degree $d_1, d_2$ over $S$ we have the following
diagram of isomorphisms defining~$\psi_{D_1, D_2}$
\[
\begin{tikzcd}
 \calM(\Sigma(D_2), \Sigma(D_1))\arrow[equal]{d}\arrow[equal]{r} &
 \Norm_{D_1/S}(\calO_C(D_2 - d_2 b)) \otimes
 b^*\calO_C(D_2-d_2b)^{-d_1} \arrow{d}{\psi_{D_1,D_2}}\\ 
 \calM(\Sigma(D_1), \Sigma(D_2)) \arrow[equal]{r} &
 \Norm_{D_2/S}(\calO_C(D_1 - d_1 b)) \otimes b^*\calO_C(D_1-d_1b)^{-d_2} \,.  
\end{tikzcd}    
\]
Then
\begin{subeqn}\label{eq:symms_are_the_same}
\psi_{D_1,D_2} = \phi_{D_1,D_2} \otimes \phi_{D_1,d_2 b}^{-1} \otimes
\phi_{d_1 b, D_2}^{-1} \otimes \phi_{d_1 b, d_2 b}\,. 
\end{subeqn}
Moreover the isomorphisms $\phi_{D_1,D_2}$, and consequently
$\psi_{D_1,D_2}$, are compatible with addition of divisors, that is,
under~(\ref{eq:plus_1_N}) and~(\ref{eq:plus_2_N}), for every triple
$D_1,D_2,D_3$ of relative Cartier divisors on $C$ we have 
\begin{subeqn}\label{eq:phi_and_addition}
\phi_{D_1 + D_2, D_3} = \phi_{D_1, D_3} \otimes \phi_{D_2, D_3}\,,
\quad \phi_{D_1, D_2 + D_3} = \phi_{D_1, D_2} \otimes \phi_{D_1, D_3}\,. 
\end{subeqn}
\end{subLemma}
\begin{proof}
It is enough to prove it in the universal case, that is, when $D_1$ and
$D_2$ are the universal divisors on $C_U$, and there we know that
there exists a $u$ in $\calO_U(U)^\times = \calO_S(S)^\times$
such that  
\begin{subeqn}
u \cdot \psi_{D_1,D_2} = \phi_{D_1,D_2} {\otimes} \phi_{D_1,d_2
  b}^{-1} {\otimes} \phi_{d_1 b, D_2}^{-1} {\otimes} \phi_{d_1 b, d_2
  b}\,. 
\end{subeqn}
Since the symmetry in Proposition~\ref{prop_P_and_M_and_norm} is
compatible with the rigidification at $(0,0) \in (J\times J)(S)$ then
$\psi_{d_1 b,d_2 b}$ is the identity on~$\calO_U$, as well as the
right hand side of~(\ref{eq:symms_are_the_same}) when $D_i =
d_ib$. Hence $u = u(d_1 b, d_2 b) = 1$,
proving~(\ref{eq:symms_are_the_same}). 

Now we prove~(\ref{eq:phi_and_addition}). As for
(\ref{eq:symms_are_the_same}), it is enough to prove it in the
universal case and then we can reduce to the case where $D_1 = d_1 b$,
$D_2 = d_2 b$ and $D_3 = d_3 b$ for $d_i$ positive integers where we
have
\begin{subeqn}
\begin{aligned}
\phi_{d_1b + d_2b, d_3b} & = \phi_{d_1b, d_3b} \otimes \phi_{d_2b, d_3b} =
(-1)^{(d_1 + d_2)d_3}\,, \\
\phi_{d_1b, d_2b + d_3b} & = \phi_{d_1b, d_2b}
\otimes \phi_{d_1b, d_3b} = (-1)^{d_1(d_2 + d_3)}\,.     
\end{aligned}
\end{subeqn}
\end{proof}

\subsection{Explicit residue disks and partial group laws}
\label{sec:parameters_on_P_mod_p2}

Let $C$ be a smooth, proper, geometrically connected curve
over~$\Zpp$, with a $b\in C(\Zpp)$, let $g$ be the genus, and let
$\calM$ be as in Proposition~\ref{prop_P_and_M_and_norm}. Let
$D = D^+ - D^-$ and $E= E^+ - E^-$ be relative Cartier divisors of
degree~$0$ on~$C$. For each $\alpha$ in $\calM^\times(\F_p)$ whose
image in $(J\times J)(\F_p)$ is given by $(D, E)$ we
parametrise~$\calM^\times(\Zpp)_\alpha$, under the assumption that
there exists a non-special split reduced divisor of degree~$g$
on~$C_{\F_p}$.

Let $b_1, \ldots, b_g $ in $C(\Zpp)$ have distinct images $\ol{b}_i$
in~$C(\F_p)$ such that $h^0(C_{\F_p}, \ol{b}_1+\cdots+ \ol{b}_g)=1$,
and let $b_{g+1},\ldots, b_{2g}$ in $C (\Zpp)$ be such that
the $\ol{b}_{g+i}$ are distinct and $h^0(C_{\F_p},
\ol{b}_{g+1}+\cdots+ \ol{b}_{2g})=1$. 
Then the maps
\begin{subeqn}\label{def_f1_f2}
\begin{aligned}
f_1 \colon C^g &\lto J\,,  \quad (c_1, \ldots , c_g) \longmapsto
\left[ \calO_C( c_1 + \cdots + c_g - (b_1 +\cdots +  b_g) + D) \right] \\ 
f_2 \colon C^g &\lto J\,, \quad (c_{1}, \ldots , c_{g}) \longmapsto
\left[ \calO_C( c_1 + \cdots + c_g - (b_{g+1} +\cdots + b_{2g}) + E)
\right]\,, 
\end{aligned}
\end{subeqn}
are \'etale respectively in the points
$( \ol{b_1},\ldots, \ol{b_g}) \in C^g(\F_p)$ and
$( \ol{b_{g+1}},\ldots, \ol{b_{2g}}) \in C^g(\F_p)$ and consequently give
bijections $C^g(\Zpp)_{(\ol{b_1},\ldots, \ol{b_g})} \to J(\Zpp)_{\ol D}$ and
$C^g(\Zpp)_{(\ol{b_{g+1}},\ldots, \ol{b_{2g}})} \to J(\Zpp)_{\ol E}$. For each
point $ c \in C(\F_p)$ we choose 
\begin{subeqn} \label{eq:par_choice}
  x_{D,c}\in \calO_C(-D)_c \text{ a generator}\,,
  \quad \text{and $x_c \in \calO_{C, c}$,}
\end{subeqn}
such that $p$ and $x_c$ generate the maximal ideal of~$\calO_{C, c}$.

For each $i =1, \ldots, 2g$ we choose $x_{b_i}$
so that $x_{b_i}(b_i)=0$. For each $(\Zpp)$-point $c \in C(\Zpp)$ with
image $\ol c$ in $C(\F_p)$ and for each $\lambda \in \F_p$ let
$c_\lambda$ be the unique point in $C(\Zpp)_{\ol c}$ with
$x_{\ol c}(c_\lambda) = \lambda p$. Then the map
$\lambda \mapsto c_{\lambda}$ is a bijection $\F_p \to C(\Zpp)_{\ol c}$
hence the maps $f_1, f_2$ induce bijections
\begin{subeqn}\label{def_Dlambda_Emu}
\begin{aligned}
\F_p^g & \lto J(\Zpp)_{\ol D}\,, \quad  \lambda \longmapsto D_\lambda := D +
(b_{1,\lambda_1}  - b_1) + \cdots + (b_{g,\lambda_g}  - b_g) \\
\F_p^g & \lto J(\Zpp)_{\ol E}\,, \quad  \mu \longmapsto E_\mu := E +
(b_{g+1,\mu_1}  - b_{g+1}) + \cdots + (b_{2g,\mu_g}  - b_{2g})\,.
\end{aligned}
\end{subeqn}
Hence $\calM^\times(\Zpp)_{\ol{D}, \ol{E}}$ is the union of
$\calM^\times(D_{\lambda}, E_{\mu})$ as $\lambda$ and $\mu$
vary in $\F_p^g$ and by Proposition~\ref{prop_P_and_M_and_norm} and
Remark~\ref{rem_no_base_point_deg_0} we have
\begin{subeqn}\label{eq:M_and_norms_modulo_p2}
\begin{aligned}
\calM(D_{\lambda}, E_{\mu})= & \Norm_{E^+/(\Zpp)}(\calO_C(D_\lambda
)) \otimes \Norm_{E^-/(\Zpp)}(\calO_C(D_\lambda))^{-1} \otimes \\ 
& \otimes  \bigotimes_{i=1}^g \left( b_{g+i,
    \mu_i}^*\calO_C(D_{\lambda}) \otimes  b_{g+i}^*
  \calO_C(D_{\lambda})^{-1} \right)\,. 
\end{aligned}
\end{subeqn}
For each $i\in\{1,\ldots, g\}$, $c \in C(\Zpp)$ and $\lambda \in \F_p$
we define $x_i(c,\lambda) := 1$ if $\ol{c}\neq\ol{b}_i$ and
$x_i(c,\lambda) := x_{b_i} - \lambda p$ if $\ol{c}=\ol{b}_i$, so that
$c^*x_i(c, \lambda)^{-1}$ generates $c^* \calO(b_{i,\lambda})$.  Then,
for each $c \in C(\Zpp)$ and each $\lambda \in \F_p^g$,
\begin{subeqn}\label{eq:gen_mod_p2_variable_part}
c^* \left( x_{D,c}^{-1} \cdot \prod_{i=1}^g  \frac{x_i(c,0)}{x_i(c,\lambda_i)} \right) \quad \text{generates }c^* \calO_C(D_\lambda) \,.
\end{subeqn}
We write $E^\pm = E^{0,\pm} + \cdots + E^{g,\pm}$ so that $E^{0, \pm}$
is disjoint from $\{\ol{b}_1,\ldots,\ol{b}_g\}$, and $E^{ i,\pm}$,
restricted to~$C_{\F_p}$, is supported on~$\ol{b}_i$. Let
$x_{D,E}$ be a generator of $\calO_{C}(-D)$ in a neighborhood of
$E^+ \cup E^-$. Then, for each $\lambda$ in $\F_p^g$,
\begin{subeqn}\label{eq:gen_mod_p2_constant_part}
\Norm_{E^{0,\pm}/(\Zpp)} (x_{D,E}^{-1}) \otimes \bigotimes_{i=1}^g
\Norm_{E^{i,\pm}/(\Zpp)}\left( x_{D,E}^{-1} \cdot
  \frac{x_{b_i}}{x_{b_i} - \lambda_i p} \right)  
\end{subeqn}
generates $\Norm_{E^{\pm}/(\Zpp)}(\calO_C(D_\lambda))$. By
(\ref{eq:M_and_norms_modulo_p2}), (\ref{eq:gen_mod_p2_variable_part})
and~(\ref{eq:gen_mod_p2_constant_part}) we see that, for $\lambda$ and
$\mu$ in~$\F_p^g$,
\begin{subeqn}\label{eq:def_s_DE_lambda}
\begin{aligned}
s_{D,E}(\lambda,& \mu) :=  \Norm_{E^{0,+}/(\Zpp)} (x_{D,E}^{-1})
\otimes \bigotimes_{i=1}^g \Norm_{E^{i,+}/(\Zpp)}\left( x_{D,E}^{-1}
  \cdot \frac{x_{b_i}}{x_{b_i} - \lambda_i p} \right) \otimes \\ 
&  \otimes \Norm_{E^{0,-}/(\Zpp)} (x_{D,E}^{-1})^{-1} \otimes
\bigotimes_{i=1}^g \Norm_{E^{i,-}/(\Zpp)}\left( x_{D,E}^{-1}
  \cdot \frac{x_{b_i}}{x_{b_i} - \lambda_i p} \right)^{-1}  \otimes \\ 
&\otimes  \bigotimes_{i = 1}^g \left( b_{g+i, \mu_i}^* \left(
    x_{D,b_{g+i}}^{-1} \cdot \prod_{j=1}^g  \frac{x_j(b_{g+i,
        \mu_i},0)}{x_j(b_{g+i, \mu_i},\lambda_j)} \right) \otimes
  b_{g+i}^* \left( x_{D,b_{g+i}}^{-1} \cdot \prod_{j=1}^g
    \frac{x_j(b_{g+i},0)}{x_j(b_{g+i},\lambda_j)} \right)^{-1} \right)  
\end{aligned}
\end{subeqn}
generates the free rank one $\Zpp$-module~$\calM(D_{\lambda}, E_{\mu})$.
The fibre $\calM^\times(\ol{D}, \ol{E})$ over $(\ol{D}, \ol{E})$ in
$(J\times J)(\F_p)$ is an $\F_p^\times$-torsor, containing
$\ol{s_{D,E}(0,0)}$, hence in bijection with $\F_p^\times$ by sending $\xi$ in
$\F_p^\times$ to $\xi{\cdot}\ol{s_{D,E}(0,0)}$. Using that
$(\Zpp)^\times = \F_p^\times \times (1 + p \F_p)$, we conclude the
following lemma.
\begin{subLemma}\label{lem:parametrization_residue_disk_M}
  With the assumptions and definitions from the start of
  Section~\ref{sec:parameters_on_P_mod_p2}, we have, for each
  $\xi \in \F_p^\times$, a parametrisation of the mod~$p^2$ residue
  polydisk of $\calM^\times$ at $\xi {\cdot}\ol{s_{D,E}(0,0)}$ by the
  bijection
\[
\begin{aligned}
\F_p^g \times \F_p^g \times \F_p \lto
\calM^\times(\Zpp)_{\xi{\cdot}\ol{s_{D,E}(0,0)}}  \,, \quad  (\lambda,
\mu, \tau) \longmapsto (1 + p\tau) {\cdot} \xi {\cdot} s_{D,E}
(\lambda, \mu) \,. 
\end{aligned}
\]
\end{subLemma}
Using this parametrization it easy to describe the two partial group
laws on $\calM^\times(\Zpp)$ when one of the two points we are summing
lies over $(\ol D, \ol E)$ and the other lies over $(\ol D,0)$ or
$(0, \ol E)$.  To compute the group law in $J(\Zpp)$ we notice that
for each $c \in C(\Zpp)$ such that $x_c(c)=0$ and for each
$\lambda, \mu \in \F_p$ we have
\begin{subeqn}\label{eq:rat_fun:life_mod_p2_is_easy}
\frac{x_c^2}{(x_c {-} \lambda p)(x_c {-} \mu p)} = \frac{x_c^2}{x_c^2
  - \lambda p x_c - \mu p x_c} =
\frac{x_c}{x_c - (\lambda {+} \mu)p} 
\end{subeqn}
and since these rational functions generate
$\calO_C( c_{\lambda} - c + c_{\mu} - c )$ and
$\calO_C( c_{\lambda + \mu} - c )$ in a neighborhood of~$c$, we have
the \emph{equality} of relative Cartier divisors on~$C$
\begin{subeqn}\label{eq:div:life_mod_p2_is_easy}
(c_{\lambda} - c) + (c_{\mu} - c) = c_{\lambda + \mu} - c \,.
\end{subeqn}
Hence, under the definition for $\lambda \in \F_p^g$ of 
\begin{subeqn}
D_{\lambda}^0 := (b_{1,\lambda_1}-b_1) +\cdots+ (b_{g,\lambda_g}-b_g) \,,
\quad
E_{\lambda}^0 := (b_{g+1,\lambda_1}-b_{g+1})+\cdots+
(b_{2g,\lambda_g}-b_{2g}) \,, 
\end{subeqn}
we have, for all $\lambda, \mu \in \F_p^g$, that 
$D_{\lambda} + D_\mu^0 = D_{\lambda + \mu}$ and
$E_{\lambda} + E_\mu^0 = E_{\lambda +
  \mu}$. Definition~\ref{eq:def_s_DE_lambda}, applied with $(D,0)$ and
$(0,E)$, with $x_{0,E}=1$ and, for every
$c \in C(\F_p)$, with $x_{0,c} = 1$, gives, for all $\lambda , \mu$
in~$\F_p^g$, the elements
\begin{subeqn}
s_{D, 0}(\lambda, \mu) \in \calM^\times(D_\lambda, E_\mu^0) \,, \quad 
s_{0, E}(\lambda, \mu) \in \calM^\times(D_\lambda^0, E_\mu) \,.
\end{subeqn}
With these definitions, we have the following lemma for the partial
group laws of~$\calM$.
\begin{subLemma}\label{lem:part_group_laws_res_disks_M}
With the assumptions and definitions from the start of
Section~\ref{sec:parameters_on_P_mod_p2}, we have,
for all $\lambda, \lambda_1, \lambda_2, \mu, \mu_1, \mu_2$
in~$\F_p^g$, that
\[
\begin{aligned}
s_{D, 0}(\lambda, \mu_1) +_2 s_{D,E}(\lambda, \mu_2) & = s_{D,
  0}(\lambda, \mu_1) \otimes s_{D,E}(\lambda, \mu_2) =
s_{D,E}(\lambda,\mu_1 + \mu_2) \\ 
s_{0, E}(\lambda_1, \mu) +_1 s_{D,E}(\lambda_2, \mu) & = s_{D,
  0}(\lambda_1, \mu) \otimes s_{D,E}(\lambda_2, \mu) =
s_{D,E}(\lambda_1 + \lambda_2,\mu) \,, 
\end{aligned}
\]
and, consequently, for all $\tau_1,\tau_2\in\F_p$ and
$\xi_1,\xi_2\in\F_p^\times$, that
\begin{subeqn}\label{eq:explicit_partial_group_laws}
\begin{aligned}
\xi_1(1 {+} \tau_1 p ) {\cdot} s_{D, 0}(\lambda, \mu_1) +_2 \xi_2(1
{+} \tau_2 p ) {\cdot} s_{D,E}(\lambda, \mu_2) & =\xi_1(1 {+} \tau_1 p
) \xi_2(1 {+} \tau_2 p ) {\cdot} s_{D,E}(\lambda,\mu_1 {+} \mu_2) \\ 
& = \xi_1\xi_2(1 {+} (\tau_1 {+} \tau_2)p ) {\cdot}
s_{D,E}(\lambda,\mu_1 {+} \mu_2) \,,  \\ 
\xi_1(1 {+} \tau_1 p ) {\cdot} s_{0, E}(\lambda_1, \mu) +_1 \xi_2(1
{+} \tau_2 p ) {\cdot} s_{D,E}(\lambda_2, \mu)  
& = \xi_1 \xi_2 (1 {+} (\tau_1 {+} \tau_2)p) {\cdot} s_{D,E}(\lambda_1
{+} \lambda_2,\mu) \,. 
\end{aligned}
\end{subeqn}
\end{subLemma}
\begin{proof}
This follows from~(\ref{eq:rat_fun:life_mod_p2_is_easy}) and
(\ref{eq:div:life_mod_p2_is_easy}), together with the equivalence of
(\ref{eq:plus_2_M}) and~(\ref{eq:plus_2_N}) and the equivalence of
(\ref{eq:plus_1_M}) and~(\ref{eq:plus_1_N})  in
Proposition~\ref{prop_P_and_M_and_norm}. 
\end{proof}
We end this section with one more lemma. 
\begin{subLemma}\label{lem:par_res_disk_M_by_parameters}
The parametrization in
Lemma~\ref{lem:parametrization_residue_disk_M} is the inverse of a 
bijection given by parameters on $\calM^\times$ analogously
to~(\ref{eq:param_bijection}). 
\end{subLemma}
\begin{proof}
Let $\calQ$ be the pullback of $\calM$
by~$f_1 {\times} f_2$ with $f_1$ and $f_2$ as
in~(\ref{def_f1_f2}). Then the lift
$\wt{f_1 {\times} f_2}\colon \calQ^\times \to \calM^\times$ is
\'etale at any point $\beta\in\calQ(\F_p)$ lying over
$\ol{b}=(b_1,\ldots,b_{2g})\in (C^{2g})(\F_p)$ and induces a bijection
between $\calQ^\times(\Zpp)_{\ol{b}}$
and~$\calM^\times(\Zpp)_{(\ol D, \ol E)}$. In particular we can
interpret $s_{D,E}(\lambda, \mu)$ as a section of
$\calQ(b_{1, \lambda_1}, \ldots b_{2g, \mu_g})$ and we can interpret
the parametrization in Lemma~\ref{lem:parametrization_residue_disk_M}
as a parametrization 
of~$Q^\times(\Zpp)_{\xi \ol{s_{D,E}(0,0)}}$. It is then enough to
prove that the parametrization in
Lemma~\ref{lem:parametrization_residue_disk_M} is the inverse of 
a bijection given by parameters on~$\calQ^\times$. It comes from the
definition of $c_\nu$ for $c\in C(\Zpp)$ and $\nu\in\F_p$, that the
maps $\lambda_i\, \mu_i \colon C^{2g}(\Zpp)_{\ol b} \to \F_p$ are given by
parameters in~$\calO_{C^{2g}, \ol b}$ divided by~$p$. In order to see
that also the coordinate
$\tau \colon \calQ^\times(\Zpp)_{\xi s_{D,E}(0,0)}\to \F_p$ is given by
a parameter divided by~$p$ it is enough to prove that there is an open
subset $U\subset C^{2g}$ containing $\ol b$ and a section $s$
trivializing $\calQ|_U$ such that
$s_{D,E}(\lambda, \mu) = s(b_{1,\lambda_1},\ldots, b_{2g, \mu_g})$.
Remark~\ref{rem_no_base_point_deg_0} and~(\ref{eq:norm_symm_isom})
give that
\begin{subeqn}\label{eq:justifying_parameters}
\begin{aligned}
\calQ =& \bigotimes_{i,j=1}^g   \Big( (\pi_{i}, \pi_{g+j})^*\calO_{C
  \times C}( \Delta) \Big)\\
& \otimes \bigotimes_{i=1}^g \Big( \pi_{i}^*
\calO_C(E-(b_{g+1} +\cdots+b_{2g}))
\otimes \pi_{g+i}^* \calO_C(D-(b_1+\cdots+b_g)) \Big)  \\ 
& \otimes \Norm_{E/\Zpp} (\calO_C(D-(b_1+\cdots+b_g)))  \otimes
\bigotimes_{i=1}^g  b_{g+i}^*\calO_C(D-(b_1+\cdots+ b_g))^{-1}  
\end{aligned} 
\end{subeqn}
where $\Delta \subset C {\times} C$ is the diagonal and $\pi_i$ is the
$i$-th projection $C^g \times C^g \to C$.  We can prove that there is
an open subset $U\subset C^g{\times} C^g$ containing $b$ and a section
$s$ trivializing $\calQ|_U$ such that
$s_{D,E}(\lambda, \mu ) = s(b_{1,\lambda_1},\ldots, b_{2g, \mu_g})$,
by trivializing each factor of the above tensor product in a
neighborhood of~$b$. Let us see it, for example, for the pieces of the
form $(\pi_i, \pi_{g+j})^*\calO_{C \times C}(\Delta)$.  Let
$\pi_1, \pi_2$ be the two projections $C \times C \to C$ and let us
consider the divisor~$\Delta$: for each pair of points
$c_1, c_2\in C(\F_p)$ the invertible $\calO$-module
$\calO_{C\times C}(-\Delta)$ is generated by the section
$x_{\Delta, c_1, c_2} : =1$ in a neighborhood of $(c_1,c_2)$ if
$c_1 \neq c_2$, while it is generated by the section
$x_{\Delta, c_1, c_2}:= \pi_1^* x_{ c_1} - \pi_2 ^* x_{ c_2}$ in a
neighborhood of $(c_1,c_2)$ if $ c_1 = c_2$. If we now take
$c_1 = b_i, c_2 = b_{g+j} \in C(\F_p)$ we deduce there is a
neighborhood $U$ of $(b_i, b_{g+j})$ such that
$x_{\Delta, b_i, b_{g+j}}^{-1}$ generates
$\calO_{C \times C}(\Delta)|_U$. For each $\lambda, \mu \in \F_p^g$
the point $(b_{i,\lambda_i},b_{g+j, \mu_j})$ lies in $U(\Zpp)$ and the
canonical isomorphism
$(b_{i,\lambda_i},b_{g+j, \mu_j})^*\calO_{C \times C}(\Delta) =
b_{g+j, \mu_j}^* \calO(b_{i, \lambda_i})$ sends the generating section
$(b_{i,\lambda_i},b_{j, \mu_j})^*x_{\Delta, c_1, c_2}^{-1}$ to
$b_{j, \mu_j}^* x_{i}(b_{g+j}, \lambda_i)^{-1}$, which is a factor in
(\ref{eq:def_s_DE_lambda}). This gives a section $s_{i,j}$
trivializing
$\Big( (\pi_{i}, \pi_{g+j})^*\calO_{C \times C}( \Delta) \Big)$ in a
neighborhood of $b$. With similar choices we can find sections
trivializing the other factors in~(\ref{eq:justifying_parameters}) in
a neighborhood of $b$ and tensoring all such sections we get a section
$s$ such that
$s_{D,E}(\lambda, \mu ) = s(b_{1,\lambda_1},\ldots, b_{2g, \mu_g})$.
\end{proof}

\subsection{Extension of the Poincar\'e biextension over N\'eron models}
\label{sec:extension_Poincare_Neron}

Let $C$ over $\Z$ be a curve as in Section~\ref{sec:alg_geometry}. Let
$q$ be a prime number that divides~$n$. We also write $C$
for~$C_{\Z_q}$. Let $J$ be the N\'eron model over $\Z_q$ of
$\Pic^0_{C/\Q_q}$, and $J^0$ its fibre-wise connected component
of~$0$. On $(J\times_{\Z_q} J)_{\Q_q}$ we have $\calM$ as in
Proposition~\ref{prop_P_and_M_and_norm}, rigidified at
$0\times J_{\Q_q}$ and at~$J_{\Q_q}\times 0$.
\begin{subProposition}
  The invertible $\calO$-module $\calM$ on
  $(J\times_{\Z_q} J)_{\Q_q}$, with its rigidifications, extends
  uniquely to an invertible $\calO$-module $\wt{\calM}$ with
  rigidifications on~$J\times_{\Z_q}J^0$. The biextension structure on
  $\calM^\times$ extends uniquely to a biextension structure
  on~$\wt{\calM}^\times$. 
\end{subProposition}
\begin{proof}
  First of all, $J\times_{\Z_q}J^0$ is regular, hence Weil divisors
  and Cartier divisors are the same, and every invertible
  $\calO$-module on $(J\times_{\Z_q}J^0)_{\Q_q}$ has an extension to
  an invertible $\calO$-module on $J\times_{\Z_q}J^0$. So let $\calM'$
  be an extension of~$\calM$. Any extension $\calM''$ of $\calM$ is
  then of the form $\calM'(D)$, with $D$ a divisor on
  $J\times_{\Z_q}J^0$ with support in
  $(J\times_{\Z_q}J^0)_{\F_q}$. Such $D$ are $\Z$-linear combinations
  of the irreducible components of the $D_i\times_{\F_q}J^0_{\F_q}$,
  where the $D_i$ are the irreducible components of~$J_{\F_q}$. Now
  $\calM'|_{J\times 0}$ extends $\calM|_{J_{\Q_q}\times 0}$, hence the
  rigidification of $\calM|_{J_{\Q_q}\times 0}$ is a rational section
  of $\calM'|_{J\times 0}$ whose divisor is a $\Z$-linear combination
  of the~$D_i$. It follows that there is exactly one $D$ as above such
  that the rigidification of $\calM$ extends to a rigidification
  of~$\calM'(D)$ on $J\times 0$. That rigidification is compatible
  with a unique rigidification of $\calM'(D)$ on $0\times J^0$. We
  denote this extension $\calM'(D)$ of $\calM$ to $J\times_{\Z_q} J^0$
  by~$\wt{\calM}$.

  Let us now prove that the $\Gm$-torsor $\wt{\calM}^\times$ on
  $J\times_{\Z_q} J^0$ has a unique biextension structure, extending
  that of~$\calM^\times$. Over $J\times_{\Z_q}J\times_{\Z_q}J^0$ we
  have the invertible $\calO$-modules whose fibres, at a point
  $(x,y,z)$ (with values in some $\Z_q$-scheme) are
  $\wt{\calM}(x+y,z)$ and $\wt{\calM}(x,z)\otimes\wt{\calM}(y,z)$. The
  biextension structure of~$\calM^\times$ gives an isomorphism between
  the restrictions of these over~$\Q_q$, that differs from an
  isomorphism over $\Z_q$ by a divisor with support over~$\F_q$. The
  compatibility with the rigidification of $\wt{\calM}$ over
  $J\times_{\Z_q}0$ proves that this divisor is zero. The other
  partial group law, and the required properties of them follow in the
  same way. We have now shown that $\wt{\calM}^\times$ extends the
  biextension~$\calM^\times$.
\end{proof}

\subsection{Explicit description of the extended Poincar\'e bundle}
\label{sec:explicit_P_over_Zq}
Let $C$ over $\Z$ be a curve as in Section~\ref{sec:alg_geometry}. Let
$q$ be a prime number that divides~$n$. We also write $C$
for~$C_{\Z_q}$. By~\cite{Liu}, Corollary~9.1.24, $C$ is
cohomologically flat over~$\Z_q$, which means that for all
$\Z_q$-algebras~$A$, $\calO(C_A)=A$. Another reference for this
is~\cite{Ray}, (6.1.4), (6.1.6) and~(7.2.1). 

The relative Picard functor $\Pic_{C/\Z_q}$ sends a $\Z_q$-scheme $T$
to the set of isomorphism classes of $(\calL,\rig)$ with $\calL$ an
invertible $\calO$-module on $C_T$ and $\rig$ a rigidification
at~$b$. By cohomological flatness, such objects are rigid. But if the
action of $\Gal(\Fbar_q/\F_q)$ on the set of irreducible components
of~$C_{\Fbar_q}$ is non-trivial, then $\Pic_{C/\Z_q}$ is not
representable by a $\Z_q$-scheme, only by an algebraic space
over~$\Z_q$ (see~\cite{Ray}, Proposition~5.5). Therefore, to not be
annoyed by such inconveniences, we pass to $S:=\Spec(\Z_q^\unr)$, the
maximal unramified extension of~$\Z_q$. Then $\Pic_{C/S}$ is
represented by a smooth $S$-scheme, and on $C\times_S\Pic_{C/S}$ there
is a universal pair $(\calL^\univ,\rig)$ (\cite{Ray}, Proposition~5.5,
and Section~8.0). We note that $\Pic_{C/S}\to S$ is separated if and
only if $C_{\Fbar_q}$ is irreducible.

Let $\Pic^{[0]}_{C/S}$ be the open part of $\Pic_{C/S}$ where
$\calL^\univ$ is of total degree zero on the fibres of $C\to S$. It contains
the open part $\Pic^0_{C/S}$ where $\calL^\univ$ has degree zero on
all irreducible components of~$C_{\Fbar_q}$.

Let $E$ be the closure of the $0$-section of $\Pic_{C/S}$, as
in~\cite{Ray}. It is contained in~$\Pic^{[0]}_{C/S}$.  By~\cite{Ray},
Proposition~5.2, $E$ is represented by an $S$-group scheme, \'etale.

By~\cite{Ray}, Theorem~8.1.4, or \cite{B-L-R}, Theorem~9.5.4, the
tautological morphism $\Pic^{[0]}_{C/S}\to J$ is surjective (for the
\'etale topology) and its kernel is~$E$, and so
$J=\Pic^{[0]}_{C/S}/E$. Also, the composition
$\Pic^0_{C/S}\to \Pic^{[0]}_{C/S} \to J$ induces an isomorphism
$\Pic^0_{C/S}\to J^0$.

Let $C_i$, $i\in I$, be the irreducible components
of~$C_{\Fbar_q}$. Then, as divisors on $C$, we have
\begin{subeqn}\label{eq:spec_fib_divisor}
C_{\Fbar_q} = \sum_{i\in I} m_iC_i\,.
\end{subeqn}
For $\calL$ an invertible $\calO$-module on $C_{\Fbar_q}$, its
multidegree is defined as
\begin{subeqn}\label{eq:multidegree}
\multideg(\calL)\colon I\to\Z,\quad i\mapsto \deg_{C_i}(\calL|_{C_i})\,,
\end{subeqn}
and its total degree is then
\begin{subeqn}\label{eq:totaldeg}
\deg(\calL) = \sum_{i\in I} m_i \deg_{C_i}(\calL|_{C_i})\,.
\end{subeqn}
The multidegree induces a surjective morphism of groups
\begin{subeqn}
\multideg\colon \Pic_{C/S}(S) \to \Z^I\,.
\end{subeqn}
Now let $d\in \Z^I$ be a sufficiently large multidegree so that every
invertible $\calO$-module $\calL$ on $C_{\Fbar_q}$ with
$\multideg(\calL)=d$ satisfies $\rmH^1(C_{\Fbar_q},\calL)=0$ and has a
global section whose divisor is finite. Let $\calL_0$ be an invertible
$\calO$-module on~$C$, rigidified at~$b$, with
$\multideg(\calL_0)=d$. Then over $C\times_S J^0$ we have the
invertible $\calO$-module $\calL^\univ\otimes\calL_0$, and its
pushforward $\calE$ to~$J^0$. Then $\calE$ is a locally free
$\calO$-module on $J^0$. Let $E$ be the geometric vector bundle over
$J^0$ corresponding to~$\calE$. Then over~$E$, $\calE$ has its
universal section. Let $U\subset E$ be the open subscheme where the
divisor of this universal section is finite over~$J^0$. The
$J^0$-group scheme $\Gm$ acts freely on~$U$. We define $V:=U/\Gm$. As
the $\Gm$-action preserves the invertible $\calO$-module and its
rigidification, the morphism $U\to J^0$ factors through $U\to V$ and
gives a morphism $\Sigma_{\calL_0}\colon V\to J^0$. Then on
$C\times_SV$ we have the universal effective relative Cartier divisor
$D^\univ$ on~$C\times_SV\to V$ of multidegree~$d$, and
$\calL^\univ\otimes\calL_0$ together with its rigidification at~$b$ is
(uniquely) isomorphic to
$\calO_{C\times_SV}(D^\univ)\otimes_{\calO_V}b^*\calO_{C\times_SV}(-D^\univ)$
with its tautological rigidification at~$b$, in a diagram:
\begin{subeqn}\label{eq:universal_L_and_D}
\begin{tikzcd}
  \calL^\univ\otimes\calL_0  \arrow[r, equals] &
  \calO_{C\times_SV}(D^\univ)\otimes_{\calO_V}b^*\calO_{C\times_SV}(-D^\univ)\,.
\end{tikzcd}
\end{subeqn}

Then $\Sigma_{\calL_0}$ sends, for $T$ an $S$-scheme, a $T$-point $D$
on $C_T$ to
$\calO_{C_T}(D)\otimes_{\calO_T}
b^*\calO_{C_T}(-D)\otimes_{\calO_C}\calL_0^{-1}$ with its rigidification
at~$b$. Let $s_0$ be in $\calL_0(C)$ such that its divisor $D_0$ is
finite over~$S$, and let $v_0\in V(S)$ be the corresponding point.

On $\Pic^{[0]}_{C/S}\times_SV\times_SC$ we have the universal
$\calL^\univ$ from $\Pic^{[0]}_{C/S}$ with rigidification at~$b$, and
the universal divisor~$D^\univ$. Then on $\Pic^{[0]}_{C/S}\times_SV$
we have the invertible $\calO$-module $\calN_{q,d}$ whose fibre at a
$T$-point $(\calL,\rig,D)$ is
$\Norm_{D/T}(\calL)\otimes_{\calO_T}\Norm_{D_0/T}(\calL)^{-1}$,
canonically trivial on~$\Pic^{[0]}_{C/S}\times_S v_0$:
\begin{subeqn}\label{eq:def_Nq}
\begin{tikzcd}
  \calN_{q,d}\colon \left(\Pic^{[0]}_{C/S}\times_SV\right)(T)\ni
  (\calL,\rig,D)  \arrow[r, mapsto] &
  \Norm_{D/T}(\calL)\otimes_{\calO_T}\Norm_{D_0/T}(\calL)^{-1}\,.
\end{tikzcd}
\end{subeqn}
Any global regular function on the integral scheme
$\Pic^{[0]}_{C/S}\times_SV$ is constant on the generic fibre, hence
in~$\Q_q^\unr$, and restricting it to $(0,v_0)$ shows that it is
in~$\Z_q^\unr$, and if it is $1$ on $\Pic^{[0]}_{C/S}\times_S v_0$, it
is equal to~$1$. Therefore trivialisations on
$\Pic^{[0]}_{C/S}\times_S v_0$ rigidify invertible $\calO$-modules on
$\Pic^{[0]}_{C/S}\times_SV$.

The next proposition generalises~\cite{M-B_MP}, Corollary~2.8.6 and
Lemma~2.7.11.2: there, $C\to S$ is nodal (but not necessarily
regular), and the restriction of $\calM$ to $J^0\times_S J^0$ is described.

\begin{subProposition}\label{prop:M_and_N_q}
In the situation of Section~\ref{sec:explicit_P_over_Zq}, the pullback
of the invertible $\calO$-module $\calM$ on $J\times_{\Z_q^\unr} J^0$ to
$\Pic^{[0]}_{C/{\Z_q^\unr}}\times_{\Z_q^\unr} V$ by the product of the
quotient map $\quot\colon \Pic^{[0]}_{C/{\Z_q^\unr}}\to J$ and the map
$\Sigma_{\calL_0}\colon V\to J^0$ is $\calN_{q,d}$, compatible with their
rigidifications at $J\times 0$ and $\Pic^{[0]}_{C/{\Z_q^\unr}}\times v_0$.
In a diagram:
\begin{subeqn}\label{eq:N_q}
\begin{tikzcd}
  P^\times \arrow{d} & & \calM^\times \arrow{ll} \arrow{d} &
  & \calN_{q,d}^\times \arrow{ll} \arrow{d} \\
J \times_{\Z_q^\unr} J^{\vee,0} & & J\times_{\Z_q^\unr} J^0 \arrow{ll}{\id\times j_b^{*,-1}} & &
\Pic^{[0]}_{C/{\Z_q^\unr}}\times_{\Z_q^\unr} V \arrow{ll}{\quot\times\Sigma_{\calL_0}}\,.
\end{tikzcd}  
\end{subeqn}
For $T$ any $\Z_q^\unr$-scheme, for $x$ in $J(T)$ given by an
invertible $\calO$-module $\calL$ on $C_T$ rigidified at~$b$, and
$y$ in $J^0(T)=\Pic^0_{C/\Z_q^\unr}(T)$ given by the difference
$D=D^+-D^-$ of effective relative Cartier
divisors on~$C_T$ of the same multidegree, we have
\[
  P(x,j_b^{*,-1}(y)) = \calM(x,y) =
  \Norm_{D^+/T}(\calL)\otimes_{\calO_T}\Norm_{D^-/T}(\calL)^{-1}\,.
\]
\end{subProposition}
\begin{proof}
The scheme $\Pic^{[0]}_{C/{\Z_q^\unr}}\times_{\Z_q^\unr} V$ is smooth
over~$\Z_q^\unr$, hence regular, it is connected, hence integral, and
since $V_{\Fbar_q}$ is irreducible, the irreducible components of
$(\Pic^{[0]}_{C/{\Z_q^\unr}}\times_{\Z_q^\unr} V)_{\Fbar_q}$ are the
$P^i\times_{\Fbar_q}V_{\Fbar_q}$, with $P^i$ the
irreducible components of~$(\Pic^{[0]}_{C/{\Z_q^\unr}})_{\Fbar_q}$,
with $i$ in $\pi_0((\Pic^{[0]}_{C/{\Z_q^\unr}})_{\Fbar_q})$, which, by
the way, equals the kernel of $\Z^I\to\Z$, $x\mapsto \sum_{j\in
  I}m_jx_j$. 

We prove the first claim. Both $\calN_{q,d}$ and the pullback of $\calM$
are rigidified on $\Pic^{[0]}_{C/{\Z_q^\unr}}\times v_0$. Below we
will give, after inverting~$q$, an isomorphism $\alpha$ from $\calN_{q,d}$
to the pullback of $\calM$ that is compatible with the
rigidifications. Then there is a unique divisor $D_\alpha$ on
$\Pic^{[0]}_{C/{\Z_q^\unr}}\times_{\Z_q^\unr} V$, supported on
$(\Pic^{[0]}_{C/{\Z_q^\unr}}\times_{\Z_q^\unr} V)_{\Fbar_q}$, such
that $\alpha$ is an isomorphism from $\calN_{q,d}(D_\alpha)$ to the
pullback of~$\calM$. Let $i$ be in
$\pi_0((\Pic^{[0]}_{C/{\Z_q^\unr}})_{\Fbar_q})$, and let $x$ be in
$\Pic^{[0]}_{C/{\Z_q^\unr}}(\Z_q^\unr)$ specialising to an
$\Fbar_q$-point of~$P^i$, then restricting $\alpha$ to $(x_i,v_0)$ and
using the compatibility of $\alpha$ (over~$\Q_q^\unr$) with the
rigidifications, gives that the multiplicity of
$P^i\times V_{\Fbar_q}$ in $D_\alpha$ is zero. Hence $D_\alpha$ is
zero.

Let us now give, over
$(\Pic^{[0]}_{C/\Z_q^\unr}\times_{\Z_q^\unr}V)_{\Q_q^\unr}$, an
isomorphism $\alpha$ from $\calN_{q,d}$ to the pullback of~$\calM$. Note
that $(\Pic^{[0]}_{C/\Z_q^\unr})_{\Q_q^\unr}=J_{\Q_q^\unr}$, and that
$V _{\Q_q^\unr}=C_{\Q_q^\unr}^{(|d|)}$, where $|d|=\sum_i m_id_i$ is
the total degree given by the multidgree~$d$. For $T$ a
$\Q_q^\unr$-scheme, $x\in J(T)$ given by $\calL$ an invertible $\calO_{C_T}$-module
rigidified at~$b$, and $v\in V(T)$ given by a relative Cartier divisor $D$ of degree $|d|$
on~$C_T$, we have, using Proposition~\ref{prop_P_and_M_and_norm}
and~(\ref{eq:def_Nq}), the following isomorphisms (functorial in~$T$),
respecting the rigidifications at $v=v_0$:
\begin{subeqn}
\begin{aligned}
\calM(x,\Sigma_{\calL_0}(v)) & = \calM(x,\Sigma(v)-\Sigma(v_0))
= \calM(x,\Sigma(v))\otimes\calM(x,\Sigma(v_0))^{-1}  \\
& = \Norm_{D/T}(\calL)\otimes_{\calO_T}\Norm_{D_0/T}(\calL)^{-1}
= \calN_{q,d}(x,v)\,.
\end{aligned}
\end{subeqn}
This finishes the proof of the first claim of the Proposition. The
second claim follows directly from the definition of~$\calN_{q,d}$, plus
the compatibility at the end of Proposition~\ref{prop_P_and_M_and_norm}.
\end{proof}

\subsection{Integral points of the extended Poincar\'e torsor}
\label{sec:Z_points_on_P}

Let $C$ over $\Z$ be a curve as in
Section~\ref{sec:alg_geometry}. Given a point
$(x,y) \in (J\times J^0)(\Z)$ we want to describe explicitly the free
$\Z$-module $\calM(x,y)$ when $x$ is given by an invertible
$\calO$-module $\calL$ of total degree $0$ on $C$ rigidified at $b$
and $y$ is given as a relative Cartier divisor $D$ on $C$ of total
degree $0$ with the property that there exists a unique divisor $V$
whose support is disjoint from $b$ and contained in the bad fibres of
$C \to \Spec(\Z)$ such that $\calO(D{+}V)$ has degree zero when
restricted to every irreducible component of any fibre of
$C \to \Spec(\Z)$.  Since $\calM(x,y)$ is a free $\Z$-module of rank
$1$ then it is a submodule of $\calM(x,y)[1/n]$ and writing
$D=D^+-D^-$ as a difference of relative effective Cartier divisors,
Proposition~\ref{prop_P_and_M_and_norm}, with $S=\Spec(\Z[1/n])$,
gives
\begin{subeqn}
\calM(x,y)[1/n] = \left(\Norm_{D^+/\Z}(\calL)
  \otimes_\Z \Norm_{D^-/\Z}(\calL)^{-1} \right)[1/n] 
\end{subeqn}
and consequently there exist unique integers~$e_q$, for $q$ varying
among the primes dividing~$n$, such that, as submodules of
$\left(\Norm_{D^+/\Z}(\calL) \otimes_\Z \Norm_{D^-/\Z}(\calL)^{-1}\right)[1/n]$,
\begin{subeqn}\label{eq:P_over_Z_no_details}
\calM(x,y)= \left( \prod_{q|n}q^{e_q} \right) \cdot
\left(\Norm_{D^+/\Z}(\calL) \otimes_\Z \Norm_{D^-/\Z}(\calL)^{-1}
\right) \,. 
\end{subeqn}
We write $V = \sum_{q|n}V_q$ where $V_q$ is a divisor supported on
$C_{\F_q}$. For every prime $q$ dividing $n$ let $C_{i,q}, i \in I_q$
the  irreducible components of $C_{\F_q}$ with multiplicity $m_{i,q}$
and let $V_{i,q}$ be the integers so that $V_{q} = \sum_{i \in I_q} V_{i,q} C_{i,q}$. 

\begin{subProposition}\label{prop:P_over_Z-point}
  The integers in (\ref{eq:P_over_Z_no_details}) are given by
\[
  e_q = -\sum_{i \in I_q}  V_{i,q} \deg_{\F_q}(\mathcal L|_{C_{i,q}}) \,.
\]
\end{subProposition}
\begin{proof}
For every $q$ dividing $n$ let $H_q$ be an effective relative Cartier
divisor on $C_{\Z_q}$ whose complement $U_q$ is affine (recall that
$C$ is projective over~$\Z$, take a high degree embedding and a
hyperplane section that avoids chosen closed points $c_{i,q}$ on
the~$C_{i,q}$). The Chinese remainder theorem, applied to the
$\calO_C(U_q)$-module $(\calO_C(D+V))(U_q)$ and the (distinct) closed
points~$c_{i,q}$, provides an element $f_q$ of $(\calO_C(D+V))(U_q)$
that generates $\calO_C(D+V)$ at all~$c_{i,q}$. Let
$D_q=D_q^+-D_q^-$ be the divisor of $f_q$ as rational section
of~$\calO_C(D+V)$. Then $D_q^+$ and $D_q^-$ are finite over~$\Z_q$,
and $f_q$ is a rational function on $C_{\Z_q}$ with
\begin{subeqn}
\divisor(f_q) =( D_q^+ - D_q^-) - (D + V) = (D_q^+ + D^-) - (D^+ + D_q^-) - V \,.
\end{subeqn}
This linear equivalence, restricted to~$\Q_q$, gives, via the
definition in~(\ref{eq:explicit_iso_equiv_div}), the isomorphism
\begin{subeqn}
  \phi \colon \Norm_{(D^+ + D_q^-)/\Q_q} (\calL) \lto \Norm_{(D_q^+ +
    D^-)/\Q_q} (\calL) \,.
\end{subeqn}
Tensoring with $\Norm_{(D^- + D^-_q)/\Q_q} (\calL)^{-1}$ we
obtain the isomorphism
\begin{subeqn}
\begin{tikzcd}
  \phi\otimes\id\colon \Norm_{D^+/\Q_q} (\calL) \otimes
  \Norm_{D^-/\Q_q} (\calL)^{-1} \arrow{r} &
  \Norm_{D_q^+/\Q_q} (\calL) \otimes \Norm_{ D^-_q/\Q_q}
  (\calL)^{-1}
\end{tikzcd}
\end{subeqn}
using the identifications
\begin{subeqn}
\begin{aligned}
  \Norm_{D^+/\Q_q} (\calL) \otimes \Norm_{D^-/\Q_q} (\calL)^{-1} &=
  \Norm_{(D^+ + D_q^-)/\Q_q} (\calL) \otimes \Norm_{(D^- +
    D^-_q)/\Q_q} (\calL)^{-1} \\
  \Norm_{D_q^+/\Q_q} (\calL) \otimes \Norm_{ D^-_q/\Q_q} (\calL)^{-1}
  &= \Norm_{(D_q^+ + D^-)/\Q_q} (\calL) \otimes \Norm_{(D^- +
    D^-_q)/\Q_q} (\calL)^{-1}\,.
\end{aligned}
\end{subeqn}
Using the same method as for getting the rational section~$f_q$
of~$\calO_C(D+V)$, we get a rational section~$l$ of $\calL$ with the
support of $\divisor(l)$ finite over~$\Z_q$ and disjoint from the
supports of $D$ and~$D_q$, and from the intersections of different
$C_{i,q}$ and~$C_{j,q}$.
By Proposition~\ref{prop:M_and_N_q}, and the choice of~$l$,
\begin{subeqn}
  \calM(x,y)_{\Z_q} =
  \Norm_{D_q^+/\Z_q}(\calL) \otimes \Norm_{D_q^-/\Z_q}(\calL)^{-1}
  =\Z_q{\cdot}\Norm_{D_q^+/\Z_q}(l) \otimes \Norm_{D_q^-/\Z_q}(l)^{-1}\,,
\end{subeqn}
and
\begin{subeqn}
\Norm_{D^+/\Z_q}(\calL) \otimes \Norm_{D^-/\Z_q}(\calL)^{-1}
  =\Z_q{\cdot}\Norm_{D^+/\Z_q}(l) \otimes \Norm_{D^-/\Z_q}(l)^{-1}\,.
\end{subeqn}
By~(\ref{eq:phi_dirty}), we have that $\phi\otimes\id$ maps
\[
\Norm_{D^+/\Q_q} (l) \otimes \Norm_{D^-/\Q_q} (l)^{-1}
\]
to
\begin{subeqn}
f_q(\divisor(l))^{-1}\cdot
  \Norm_{D_q^+/\Q_q} (l) \otimes \Norm_{D_q^-/\Q_q} (l)^{-1} \,.
\end{subeqn}
Comparing with~(\ref{eq:P_over_Z_no_details}), we conclude that
\begin{subeqn}
e_q = v_q(f_q(\divisor(l)))\,.
\end{subeqn}
We write $\divisor(l)=\sum_j n_jD_j$ as a sum of prime divisors. These
$D_j$ are finite over $\Z_q$, disjoint from the support of the
horizontal part of $\divisor(f_q)$, that is of~$D_q-D$, and each of
them meets only one of the~$C_{i,q}$, say~$C_{s(j),q}$. Then, for
each~$j$, $f_q^{m_{s(j),q}}$ and $q^{-V_{s(j),q}}$ have the same
multiplicity along~$C_{s(j),q}$, and consequently they differ multiplicatively by
a unit on a neighborhood of~$D_j$. Then we have
\begin{subeqn}
\begin{aligned}
  v_q(f_q(D_j)) & = \frac{v_q(f_q^{m_{s(j),q}}(D_j))}{m_{s(j),q}}
  = \frac{v_q(q^{-V_{s(j),q}}(D_j))}{m_{s(j),q}} = \frac{v_q\left(
      \Norm_{D_j/\Z_q}(q^{-V_{s(j),q}}) \right)}{m_{s(j),q}} \\
  & = \frac{-V_{s(j),q}\deg_{\Z_q}(D_j)}{m_{s(j),q}}
  = \frac{-V_{s(j),q}{\cdot}(D_j\cdot C_{\F_q})}{m_{s(j),q}}
  = \frac{-V_{s(j),q}{\cdot}(D_j\cdot
    m_{s(j),q}C_{s(j),q})}{m_{s(j),q}} \\
  & = - V_{s(j),q} (D_j \cdot C_{s(j)}) = -V_q\cdot D_j\,.
\end{aligned}
\end{subeqn}
We get
\begin{subeqn}\label{eq:e_q}
\begin{aligned}
  e_q & = v_q(f_q(\divisor(l))) = -V_q\cdot \divisor(l) =
  -\sum_{i \in I_q} V_{i,q} (C_i\cdot \divisor(l)) \\
  & = -\sum_{i \in I_q}  V_{i,q} \deg_{\F_q}(\mathcal L|_{C_{i,q}}) \,.    
\end{aligned}
\end{subeqn}
\end{proof}

\section{Description of the map from the curve to the torsor}
\label{sec:explicit_j_b_tilde}

The situation is as in Section~\ref{sec:alg_geometry}. The aim of this
section is to give descriptions of all morphisms in the
diagram~(\ref{eq:def_T}), in terms of invertible $\calO$-modules on
$(C\times C)_\Q$ and extensions of them over~$C\times U$, to be used
for doing computations when applying Theorem~\ref{thm:finiteness}. The
main point is that each $\tr_{c_i}\circ f_i$ is described
in~(\ref{eq:jbtilde_L}) as a morphism (of schemes)
$\alpha_{\calL_i}\colon J_\Q\to J_\Q$ with $\calL_i$ an invertible
$\calO$-module on $C\times U$, and that
Proposition~\ref{prop:jbtilde_precise} describes
$(\wt{j_b})_i\colon C_{\Z[1/n]}\to T_i$. For finding the line required
line bundles, see~\cite{CMSV}. 

We describe the morphism $\wt{j_b}\colon U\to T$ in terms of
invertible $\calO$-modules on $C\times C^\sm$.  Since $T$ is the
product, over~$J$, of the $\Gm$-torsors
$T_i :=(\id, m{\cdot}\circ \tr_{c_i} \circ f_i)^*P^\times$ this
amounts to describing, for each~$i$, the morphism
$(\wt{j_b})_i\colon U \to T_i$. Note that
$\tr_{c_i} \circ f_i\colon J_\Q\to J_\Q$ is a morphism of groupschemes
composed with a translation, and that all morphisms of schemes
$\alpha\colon J_\Q\to J_\Q$ are of this form. From now on we fix one
such~$i$ and omit it from our notation.

Let $\alpha\colon J_\Q\to J_\Q$ be a morphism of schemes, let
$\calL_\alpha$ be the pullback of $\calM$
(see~(\ref{eq:P_and_M_and_norm})) to $C_\Q \times C_\Q$ via
$j_b \times (\alpha\circ j_b)$, and let
$T_\alpha:=(\id,\alpha)^*\calM^\times$ on~$J_\Q$:  
\begin{eqn}\label{eq:L_alpha_diagram}
\begin{tikzcd}
& T_\alpha \arrow{d}\arrow{rr} & & \calM^\times\arrow{ld}\\    
  C_\Q \arrow{r}{j_b}\arrow{d}{\diag}  & J_\Q \arrow{r}{(\id,\alpha)}
  & (J\times J)_\Q & \\
  (C\times C)_\Q \arrow{r}{\id\times j_b}
  & (C\times J)_\Q \arrow{r}{\id\times\alpha}
  & (C\times J)_\Q \arrow{u}{j_b\times\id} & \\
  \calL_\alpha^\times\arrow{u}\arrow{rrr} & &
  & \calL^{\univ,\times} \arrow{uuu}\arrow{lu} \,.
\end{tikzcd}
\end{eqn}
Then $(b,\id)^*\calL_\alpha=\calO_{C_\Q}$, $\calL_\alpha$ is
of degree zero on the fibres of $\pr_2\colon (C\times C)_\Q\to C_\Q$,
and: $j_b^*T_\alpha$ is trivial if and only if $\diag^*\calL_\alpha$
is trivial.  Note that diagram~(\ref{eq:L_alpha_diagram}) without the $\Gm$-torsors is
commutative.

Conversely, let $\calL$ be an invertible $\calO$-module on
$(C\times C)_\Q$, rigidified on~$\{b\}\times C_\Q$, and of degree~$0$
on the fibres of $\pr_2\colon(C\times C)_\Q\to C_\Q$.  The universal
property of $\calL^\univ$ gives a unique
$\beta_\calL\colon C_\Q\to J_\Q$ such that
$(\id\times\beta_\calL)^*\calL^\univ=\calL$ (compatible with
rigidification at~$b$).  The Albanese property of
$j_b\colon C_\Q\to J_\Q$ then gives that $\beta_\calL$ extends to a
unique $\alpha_\calL\colon J_\Q\to J_\Q$ such
that~$\alpha_\calL\circ j_b=\beta_\calL$. Then $j_b^*T_{\alpha_\calL}$
is trivial if and only if $\diag^*\calL$ is trivial. We have proved
the following proposition.
\begin{Proposition}\label{prop:alpha_and_L}
In the situation of Section~\ref{sec:alg_geometry}, the above maps
$\alpha\mapsto\calL_\alpha$ and $\calL\mapsto \alpha_\calL$ are
inverse maps between the sets 
\begin{quote}
  \{scheme morphisms $\alpha\colon J_\Q\to J_\Q$ such
that $j_b^*(\id,\alpha)^*\calM$ is trivial\}
\end{quote}
and
\begin{quote}
  \{invertible $\calO$-modules $\calL$ on $(C\times C)_\Q$,
  rigidified on~$\{b\}\times C_\Q$, of degree $0$ on the fibres of
  $\pr_2\colon(C\times C)_\Q\to C_\Q$, and such that $\diag^*\calL$ is
  trivial\}. 
\end{quote}
\end{Proposition}
Now let $\calL$ be in the second set of
Proposition~\ref{prop:alpha_and_L}. Then $\diag^*\calL=\calO_{C_\Q}$,
compatible with rigidifications at~$b$. Let
\begin{eqn}\label{eq:ell_on_diag}
\ell\in(\diag^*\calL^\times)(C_\Q)
\end{eqn}
correspond to~$1$. Then $m{\cdot}\circ\alpha_\calL$ extends over $\Z$
to $m{\cdot}\circ\alpha_\calL\colon J\to J^0$, and the restriction of
$j_b^*(m{\cdot}\circ\alpha_\calL)^*\calM$ on $C^\sm$ to~$U$ is
trivial, giving a lift $\wt{j_b}$, unique up to sign:
\begin{eqn}
\begin{tikzcd}\label{eq:jbtilde_L}
  & & T_{m{\cdot}\circ\alpha_\calL} \arrow{rr} \arrow{d}
  & & \calM^\times \arrow{d} \\
  U \arrow{rru}{\wt{j_b}}\arrow{r}
  & C^\sm \arrow{r}{j_b}
  & J \arrow{rr}{(\id, m{\cdot}\circ\alpha_\calL)} & & J\times J^0\,. 
\end{tikzcd}
\end{eqn}
The invertible $\calO$-module $\calL$ on $(C\times C)_\Q$ with its
rigidification of $(b,\id)^*\calL$, extends uniquely to an invertible
$\calO$-module 
on $(C\times C)_{\Z[1/n]}$, still denoted~$\calL$. 
\begin{Proposition}\label{prop:M_L_D_E}
Let $S$ be a $\Z[1/n]$-scheme, let $d$ and $e$ be in $\Z_{\geq0}$, and
let $D\in
C^{(d)}(S)$ and $E\in C^{(e)}(S)$. Then we have:
\[
\calM(\Sigma(D),\alpha_\calL(\Sigma(E))) =
\left(\Norm_{D/S}(\id,b)^*\calL\right)^{\otimes(1-e)}\otimes
\Norm_{(D\times E)/S}(\calL) \,.
\]
For $x\in C(S)$ we have
\[
T_{m{\cdot}\circ\alpha_\calL}(j_b(x)) =
\calM^\times(j_b(x),m{\cdot}\alpha_\calL(j_b(x)))
= \calL^{\otimes m}(x,x)^\times = (\Gm)_S \,.
\]
\end{Proposition}
\begin{proof}
We may and do assume (finite locally free base change
on~$S$) that we have $x_i$ and $y_j$ in~$C(S)$, such that
$D=\sum_ix_i$ and $E=\sum_j y_j$. Recall that, for $c\in C(S)$,
$\beta_\calL(c)$ in $J(S)$ is $(\id,c)^*\calL$ on~$C_S$, with its
rigidification at~$b$. Then we have:
\begin{subeqn}\label{eq:M_L_D_E_proof}
\begin{aligned}
\calM(\Sigma(D),\alpha_\calL(\Sigma(E)))
& = \calM(\alpha_\calL(\Sigma(E)),\Sigma(D)) \\
& =
\calM\left(\beta_\calL(b)+\sum_j(\beta_\calL(y_j)-\beta_\calL(b)),\sum_i
j_b(x_i)\right) \\
& = \left(\bigotimes_i\calL(x_i,b)^{\otimes(1-e)}\right) \otimes
\bigotimes_{i,j}\calL(x_i,y_j)\,.
\end{aligned}
\end{subeqn}
from which the desired equality follows.

Now we prove the second claim. Let $x$ be in~$C(S)$. The first
equality holds by definition. Taking $D=E=x$ in what we just proved,
gives the second equality, and the third comes from the rigidification
at~$b$. 
\end{proof}
Now let $\calL$ be any extension of $\calL$ with its rigidification
of~$(b,\id)^*\calL$ from $(C\times C)_{\Z[1/n]}$ to~$C\times U$.
For $q$ dividing~$n$, let $W_q$ be the valuation along~$U_{\F_q}$
of the rational section $\ell$ of $\diag^*\calL$ on~$U$. Then
$\ell$, multiplied by the product, over the primes $q$ dividing~$n$,
of~$q^{-W_q}$, generates $\diag^*\calL$ on~$U$:
\begin{eqn}\label{eq:W_l}
\left(\prod_{q|n}q^{-W_q}\right){\cdot}\ell \in (\diag^*\calL^\times)(U)\,.
\end{eqn}
There is a unique divisor $V$ on $C\times U$ with support disjoint
from $(b,\id)U$ and contained in the $(C\times U)_{\F_q}$ with $q$
dividing~$n$, such that
\begin{eqn}\label{eq:modif_calLm_multideg_0}
\calL^m:=\calL^{\otimes m}(V)\quad\text{on $C\times U$}
\end{eqn}
has multidegree~$0$ on the fibres of $\pr_2\colon C\times U\to U$. Then
$\calL^m$ is the pullback of $\calL^\univ$ via
$\id\times(m{\cdot}\circ\alpha_\calL\circ j_b)\colon C\times U\to
C\times J^0$. Its restriction $\calL^m|_{C^\sm\times U}$ is then the
pullback of $\calM$ via
$j_b\times(m{\cdot}\circ\alpha_\calL\circ j_b)\colon C^\sm\times U\to
J\times J^0$, because on $C^\sm\times J^0$ the restriction of
$\calL^\univ$ and $(j_b\times\id)^*\calM$ are equal (both are
rigidified after $(b,\id)^*$ and equal over~$\Z[1/n]$; here we use
that, for all~$q|n$, $J^0_{\F_q}$ is geometrically connected). Hence,
on~$U$ we have
$j_b^*T_{m{\cdot}\circ\alpha_\calL}=\diag^*(\calL^{\otimes m}(V)^\times)$,
compatible with rigidifications at~$b\in U(\Z[1/n])$. Our trivialisation
$\wt{j_b}$ on~$U$ of~$T_{m{\cdot}\circ\alpha_\calL}$ is
therefore a generating section of~$\calL^{\otimes m}$, multiplied by
the product over the $q$ dividing~$n$, of the factors~$q^{-V_q}$,
where $V_q$ is the multiplicity in~$V$ of the prime divisor
$(U\times U)_{\F_q}$. This means that we have proved the following
proposition. 
\begin{Proposition}\label{prop:jbtilde_precise}
For $x$ and $S$ as in Proposition~\ref{prop:M_L_D_E}, we have the
following description of~$\wt{j_b}$:
\[
  \wt{j_b}(x) =
  \left(\prod_{q|n}q^{-mW_q-V_q}\right){\cdot}\ell^{\otimes m}
  \quad\text{in $(T_{m{\cdot}\circ\alpha_\calL}(j_b(x)))(S)
    =\calL^{\otimes m}(x,x)^\times(S)$.}
\]
\end{Proposition}

\section{An example with genus~2, rank~2, and 14 points}
\label{sec:example}
The example that we are going to treat is the quotient of the modular
curve $X_0(129)$ by the action of the group of order $4$ generated by
the Atkin-Lehner involutions $w_3$ and~$w_{43}$. An equation for this
quotient is given in the table in~\cite{Has}, and Magma has shown that
that equation and the equations below give isomorphic curves
over~$\Q$.

Let $C_0$ be the curve over $\Z$ obtained from the following closed
subschemes of~$\A^2_\Z$ 
\[
\begin{aligned}
V_1:& \quad y^2 +  y  = x^6 - 3 x^5 + x^4 + 3 x^3 - x^2 - x\,, \\
V_2:& \quad w^2 +  z^3 w = 1 - 3 z + z^2 + 3 z^3 - z^4 - z^5
\end{aligned}
\]
by glueing the open subset of $V_1$ where $x$ is invertible with the
open subset of $V_2$ where $z$ is invertible using the identifications
$z = 1/x$, $w = y/x^3$. The scheme $C_0$ can be also described as a
subscheme of the line bundle $\calL_3$ associated to the invertible
$\calO$-module $\calO_{\P_\Z^1}(3)$ on $\P^1_\Z$ with homogeneous coordinates
$X,Z$: the map $\calO_{\P_\Z^1}(3) \to \calO_{\P_\Z^1}(6)$ sending a section $Y$ to
$Y\otimes Y  + Z^3 \otimes Y$ induces a map $\phi$ from $\calL_3$ to
the line bundle $\calL_6$ associated to $\calO(6)$; then $C_0$ is
isomorphic to the inverse image by $\phi$ of the section $s := X^6 {-}
3 X^5Z {+} X^4Z^2 {+} 3 X^3Z^3 {-} X^2Z^4 {-} XZ^5$ of $\calL_6$ and
since the map $\phi$ is finite of degree $2$ then $C_0$ is finite of
degree $2$ over~$\P^1_\Z$. 
Hence $C_0$ is proper over $\Z$ and it is moreover smooth over
$\Z[1/n]$ with $n = 3 \cdot 43$. The generic fiber of $C_0$ is a curve
of genus $g =2$, labeled 5547.b.16641.1 on \url{www.lmfdb.org}. The
only point where  $C_0$ is not regular is the point $ P_0 = (3,
x-2,y-1)$ contained in $V_1$ and the blow up $C$ of $C_0$ in $P_0$ is
regular. 

In the rest of this section we apply our geometric method to the curve
$C$ and we prove that $C(\Z)$ contains exactly $14$ elements. We use
the same notation as in Sections~\ref{sec:alg_geometry}
and~\ref{sec:closure-finiteness}.  

The fiber $C_{\F_{43}}$ is absolutely irreducible while $C_{\F_3}$ is the union of two geometrically irreducible curves, a curve of genus $0$ that lies above the point $P_0$ and that we call $K_0$, and a curve of genus $1$ that we call $K_1$.
We define $U_0 := C\setminus K_1$ and $U_1 := C \setminus K_0$ so that $C(\Z) = C^\sm(\Z) = U_0(\Z) \cup U_1(\Z)$ and both $U_0$ and $U_1$ satisfy the hypothesis on $U$ in Section~\ref{sec:alg_geometry}.
We have $K_0 \cdot K_1 = 2$ and consequently the self-intersections of $K_0$ and $K_1$ are both equal to ${-}2$. 
We deduce that all the fibers of $J$ over $\Z$ are connected except for $J_{\F_3}$ which has group of connected components equal to~$\Z/2\Z$. Hence,
\begin{eqn}\label{eq:m}
m = 2\,.
\end{eqn}
The automorphism group of $C$ is isomorphic to $(\Z/2\Z)^2$, generated by the automorphisms $\iota$ and $\eta$ lifting the extension to $C_0$ of 
\[
\iota, \eta  \colon V_1 \lto V_1 \,, \quad \iota \colon (x,y) \longmapsto (x, -1-y)\,, \quad \eta \colon (x,y) \longmapsto (1-x, -1-y) \,.
\]
The quotients $E_1 :=C_\Q/ \eta$ and $E_2 := C_\Q/(\iota \circ \eta)$ are curves of genus $1$ and the two projections $C \to E_i$ induce an isogeny $J \to \Pic^0(E_1) \times \Pic^0(E_2)$. The elliptic curves $\Pic^0(E_i)$ are not isogenous and $\rho = 2$.

\subsection{The torsor on the jacobian}\label{sec:example_f_and_parameters_on_T}
Let $\infty, \infty_- \in C(\Z)$ be the lifts of $(0,1), (0,-1) \in V_2(\Z) \subset C_0(\Z)$ and let us fix the base point $b = \infty$ in $C(\Z)$.
Following Section~\ref{sec:explicit_j_b_tilde} we describe a $\Gm$-torsor $T\to J$ and maps $\wt{j_{b,i}} \colon U_i \to T$ using invertible $\calO$-modules on $C \times C^\sm$.
The torsor $T = (\id, m{\cdot}\circ \alpha)^* \calM^\times$ only depends on the scheme morphism $\alpha\colon J_\Q\to J_\Q$, which, by Proposition~\ref{prop:alpha_and_L}, is uniquely determined by an invertible $\calO$-module $\calL$ on $(C\times C)_\Q$, rigidified on~$\{b\}\times C_\Q$, of degree $0$ on the fibres of $\pr_2\colon(C\times C)_\Q\to C_\Q$, and such that $\diag^*\calL$ is trivial. 

We now look for a non-trivial $\calO$-module $\calL$ with these properties using the homomorphism $\eta^*\colon J_\Q \to J_\Q$, which does not belong to $\Z \subset \mathrm{End}(J_\Q)$. We can take $\alpha$ of the form $ \tr_c \circ( n_1{\cdot} \eta^* + n_2 {\cdot}\id)$, where $\id\colon J_\Q \to J_\Q$ is the identity map, $n_i$ are integers and $c$ lies in $J(\Q)$. Using the map $\alpha \mapsto \calL_\alpha := (j_b\times (j_b{\circ} \alpha))^* \calM$ in 
Proposition~\ref{prop:alpha_and_L}, the $\calO$-module $\calL_{\tr_c}$ is isomorphic to $\calO_{C_\Q \times C_\Q}(\pr_1^*D)$, the $\calO$-module $\calL_{\eta^*}$ is isomorphic to $\calO_{C_\Q \times C_\Q}(\Gamma_{\eta,\Q} - \pr_1^*\eta^*(b) -\pr_2^*\eta(b))$ and the $\calO$-module $\calL_{\id}$ is isomorphic to $\calO_{C_\Q \times C_\Q}(\diag(C_\Q) - \pr_1^*(b) -\pr_2^*(b))$, where $D$ is a divisor on $C_\Q$ representing $c$, the maps $\pr_i$ are the projections $C_\Q \times C_\Q \to C_\Q$ and $\Gamma_{\eta}$ is the graph of the map $\eta \colon C \to C$. Hence, we can take $\calL$ of the form $\calO_{C_\Q \times C_\Q}(n_1 \Gamma_{\eta, \Q } + n_2 \diag(C_\Q) + \pr_1^*D_1 + \pr_2^*D_2)$ for some integers $n_i$ and some divisors $D_i$ on $C_\Q$. Among the $\calO$-modules of this form satisfying the needed properties, we choose 
\[
\calL := \calO_{C_\Q \times C_\Q}(\Gamma_{\eta, \Q } - \pr_1^* (\infty_-) - \pr_2^*(\infty) )  =\calO_{C_\Q \times C_\Q}(\Gamma_{\eta, \Q } - \infty_- \times C_\Q - C_\Q \times\infty ) 
\]
trivialised on $b\times C_\Q$ through the section
\[
l_b := 2  \quad \text{ in } ((b,\id)^*\calL)(C_\Q) = \calO_{C_\Q}(\eta(b) - b)(C_\Q) = \calO_{C_\Q}(C_\Q)\,.
\]
For every $\ol \Q$-point $Q$ on $C_\Q$ the invertible $\calO_{C_{\ol \Q}}$-module $(\id, Q)^*\calL$ is isomorphic to $\calO_{C_{\ol \Q}}(\eta(Q) -\infty_-)$, hence 
\[
\alpha_\calL = \tr_c \circ f\,, \quad \text{with } f = \eta_*  \text{ and }c = [D_0] \,, D_0:= \infty - \infty_- \,.
\]
When restricted to the diagonal $\calL$ is trivial since, compatibly with the trivialisation at $(b,b)$,
\[
\diag^*\calL = \calO_{C_\Q}(\infty_- + \infty - \infty_- - \infty) = \calO_{C_\Q}\,.
\]
In particular, the global section $l := 1$ of $\calO_{C_\Q}$ gives a rigidification of $\diag^*\calL$ that we write as
\[
\diag^*\calL = l\cdot \calO_{C_\Q}\,.
\]
Following Proposition~\ref{prop:jbtilde_precise} and the discussion preceding it, we choose the extension of $\calL$ over $C \times C^\sm$
\[
\calL := \calO(\Gamma_{\eta}|_{C \times C^\sm} - \infty_- \times C^\sm - C \times \infty)\,,
\]
trivialised along $b\times C^\sm$ through the section $l_b=2$ (the
points $\infty_-$ and $b$ have a simple intersection over the prime
$2$). By Proposition~\ref{prop:M_L_D_E}, the torsor $T:= T_{m{\cdot} \circ \alpha_\calL}$ on $J$, with $m=2$ as explained before Equation~(\ref{eq:m}), satisfies, for $S$ a $\Z[1/n]$-scheme and $x$ in $C(S)$, using the trivialisation given by $l$ and $l_b$
\begin{subeqn}\label{eq:example_fibers_T}
\begin{aligned}
T(j_b(x)) & = \calM^\times (j_b(x), m{\cdot}\alpha_\calL(j_b(x))) =  \calM^\times (j_b(x), (\id, x)^* \calL^{\otimes m}) \\
&= x^*(\id, x)^* \calL^{\otimes m, \times} \otimes b^* (\id, x)^* \calL^{\otimes -m, \times} \\
& = \calL^{\otimes m, \times}(x,x) \otimes \calL^{\otimes m, \times}(b,x)^{-1} = \calL^{\otimes m, \times}(x,x)  = \calO_S^\times \,.
\end{aligned}
\end{subeqn}
Using Proposition~\ref{prop:jbtilde_precise} we now compute $\wt{j_{b,0}}$ and $\wt{j_{b,1}}$. 
Since $l$ generates $\diag^*(\calL)$ on the whole $C^\sm$, we have $W_3 = W_{43} = 0$.
The invertible $\calO$-module $\calL^{\otimes m}$ has multidegree $0$ over all the fibers $C\times U_1 \to U_1$, hence in order to compute $\wt{j_{b,1}}$ we must take $V = 0$ in~(\ref{eq:modif_calLm_multideg_0}), giving $V_3 = V_{43} = 0$. Hence for $S$ and $x$ as in~(\ref{eq:example_fibers_T}), assuming moreover that $2$ is invertible on $S$,
\begin{subeqn}\label{eq:example_tilde_j_b_1}
\wt{j_{b,1}}(x) = l^2 \otimes l_b^{-2} = \frac 1 4 (x^* 1) \otimes (b^* 1)^{-1} \quad \text{in } 
\end{subeqn}
\[
T(j_b(x)) = x^*(\id, x)^* \calL^{\otimes m, \times} \otimes b^* (\id, x)^* \calL^{\otimes -m, \times} =x^* \calO_{C_S}(\eta(x) -\infty_-)^\times \otimes b^* \calO_{C_S} (\eta(x) - \infty_-)^\times \,,
\]
where the last equality in~(\ref{eq:example_tilde_j_b_1}) makes sense if the image of $x$ is disjoint from $\infty, \infty_-$ in $C_S$.

The restriction $\calL^{\otimes m}$ to $C \times U_0$ has multidegree $0$ over all the fibers $C\times U_0 \to U_0$ of characteristic not $3$, while if we consider a fiber of characteristic $3$ it has degree $2$ over $K_0$ and degree ${-}2$ over $K_1$. Hence for computing $\wt{j_{b,0}}$ we take $V = K_0 \times (K_0\cap U_0)$ in~(\ref{eq:modif_calLm_multideg_0}) giving $V_{43} = 0$, $V_3 = 1$. Hence for $S$ and $x$ as in~(\ref{eq:example_fibers_T}), assuming moreover that $2$ is invertible on $S$,
\begin{subeqn}\label{eq:example_tilde_j_b_0}
\wt{j_{b,0}}(x) = \frac 1 3 l^2 \otimes l_b^{-2} = \frac{1}{12}(x^* 1) \otimes (b^* 1)^{-1} \quad \text{in } 
\end{subeqn}
\[
T(j_b(x)) = x^*(\id, x)^* \calL^{\otimes m, \times} \otimes b^* (\id, x)^* \calL^{\otimes -m, \times} =x^* \calO_{C_S}(\eta(x) -\infty_-)^\times \otimes b^* \calO_{C_S} (\eta(x) - \infty_-)^\times \,,
\]
where the last equality in~(\ref{eq:example_tilde_j_b_0}) makes sense if the image of $x$ is disjoint from $\infty, \infty_-$ in $C_S$.

\subsection{Some integral points on the biextension}
\label{subsec:known_points}
On $C_0$ we have the following integral points that lift uniquely to elements of $C(\Z)$
\[
\begin{aligned}
&\infty = (0, 1)  \,, \quad  \infty_-: = (0,-1) \quad \text{in } V_2(\Z) \,, \\  
&\alpha := (1,0) \,, \quad
\beta := \eta(\alpha) = (0, -1) \,, \quad  \gamma := (2,1) \,, \quad  \delta := \eta(\gamma) = (-1,-2) \quad \text{in } V_1(\Z) \,.
\end{aligned}
\]
Computations in Magma confirm that $J(\Z)$ is a free $\Z$-module of rank $r = 2$ generated by 
\[
G_1 := \gamma  - \alpha\,, \quad  G_2 := \alpha + \infty_- - 2 \infty \,.
\]
The points in $T(\Z)$  are a subset of points of $\calM^\times(\Z)$ that can be constructed, using the two group laws, from the points in $\calM^\times(G_i, m{\cdot}f(G_j))(\Z)$ and $\calM^\times(G_i, m{\cdot}D_0)(\Z)$ for $i,j \in \{1,2 \}$. 
Let us compute in detail $\calM^\times(G_1, m \cdot f(G_1))(\Z)$.  
As explained in Proposition~\ref{prop:P_over_Z-point}, we have
\[
\begin{aligned}
\calM(G_1, m {\cdot} f(G_1))^\times & = \calM^\times(\gamma - \alpha, 2\delta - 2\beta) \\
&= 3^{e_3}  43^{e_{43}} \cdot \Norm_{(2\delta)/\Z}(\calO_C(\gamma {-} \alpha)) \otimes \Norm_{(2\beta)/\Z}(\calO(\gamma {-} \alpha))^{-1} \\
 & = 3^{e_3} 43^{e_{43}} \cdot (2\delta - 2\beta)^*\calO_C(\gamma - \alpha)
\end{aligned}
\]
where, given a scheme $S$, an invertible $\calO$-module $\calL$ on $C_S$ and a divisor $D_+ {-} D_- = \sum_{i} n_iP_i$ on $C_S$ that is sum of $S$-points, we define the invertible $\calO_S$-module
\[
\left( \sum_{i} n_iP_i \right)^* \calL := \bigotimes_i P_i^*\calL^{n_i} = \Norm_{D_+/S} (\calL) \otimes \Norm_{D_-/S}(\calL)^{-1} \,.
\]
Since $C_{\F_{43}}$ is irreducible then $2f(G_1)$ has already multidegree $0$ over $43$, hence $e_{43} = 0$. If we look at $C_{\F_3}$ then $2f(G_1)$ does not have multidegree $0$, while $2f(G_1) + K_0$ has multidegree $0$; hence, by Proposition~\ref{prop:P_over_Z-point},
\[
e_3 = - \deg_{\F_3} \calO_C(\gamma - \alpha)|_{K_0} = -1 \,.
\]
Notice that over $\Z[\frac 1 2]$ the divisor $G_1$ is disjoint from $\beta$ and $\delta$  (to see that it is disjoint from $\delta = (-1,-2,1)$  over the prime $3$ one needs to look at local equations of the blow up) thus $\beta ^* \calO_C(\gamma - \alpha)$ and $\delta ^* \calO_C(\gamma - \alpha)$ are generated by $\beta^* 1$ and $\delta^*1$ over $\Z[\frac 1 2]$. Thus there are integers $e_\beta, e_{\delta}$ such that $\beta ^* \calO_C(\gamma - \alpha)$ and $\delta ^* \calO_C(\gamma - \alpha)$ are generated by $\beta^* 2^{e_\beta}$ and $\delta^*2^{e_{\delta}}$ over $\Z$. Looking at the intersections between $\beta, \gamma, \alpha$ and $\delta$ we compute that $e_\beta = -1$ and $e_{\delta} = 1$, hence 
\[
\begin{aligned}
& \calM(G_1, m \cdot f(G_1))  = 3^{-1} {\cdot}(\delta^*2)^2 \otimes (\beta^* 2^{-1})^{-2} \cdot \Z = 2^4 {\cdot }3^{-1} {\cdot} (\delta^*1)^2 \otimes (\beta^* 1) \cdot \Z \quad \text{and} \\
& Q_{1,1} := \pm  2^4 {\cdot }3^{-1} {\cdot} (\delta^*1)^2 \otimes (\beta^* 1)^{-2}  \in \calM_{G_1, m \cdot f(G_1)}^\times(\Z) \,.
\end{aligned}
\]
With analogous computations we see that 
\[
\begin{aligned}
Q_{2,1} & :=  2^{-2} {\cdot} (\delta^*1)^2 {\otimes} (\beta^* 1)^{-2}  \quad &\text{generates }  \calM_{G_2, m \cdot f(G_1)} 
\\ Q_{1,2} & := 2^{-2} {\cdot} (\beta^*1)^2 {\otimes} (\infty_-^* 1)^2 {\otimes} (\infty ^* 1)^{-4} \quad  &\text{generates }  \calM_{G_1, m \cdot f(G_2)}  &
\\ Q_{2,2} & := 2^{18} {\cdot} (\beta^*1)^2 {\otimes} (\infty_-^* x)^2{\otimes} (\infty ^* z^2)^{-4} \quad &\text{generates } \calM_{G_2, m \cdot f(G_2)} 
\\ Q_{1,2} & := (\infty ^* 1)^2{\otimes} (\infty_- ^* 1)^{-2} \quad  &\text{generates } \calM_{G_1, m \cdot D_0} 
\\ Q_{2,0} & := 2^{-12} {\cdot} (\infty ^* z^2)^2{\otimes} (\infty_- ^* x)^{-2} \quad  &\text{generates } \calM_{G_2, m \cdot D_0} \,.
\end{aligned}
\]

\subsection{Some residue disks of the biextension}
Let $p$ be a prime of good reduction for~$C$.
Given the divisors 
\[ 
D := \alpha - \infty \,, \quad E := 2\beta - 2\infty_- = ( m {\cdot} \circ \tr_c \circ \eta_*)(D) \quad \text{in } \mathrm{Div}(C_\Zpp)
\]
we use Lemma~\ref{lem:parametrization_residue_disk_M} to give parameters on the residue disks in $\calM^\times(\Zpp)_{\ol{D}, \ol{E}}$ and $T(\Zpp)_{\ol{D}}$, with $\ol D, \ol E$ the images of $D, E$ in $\mathrm{Div}(C_{\F_p})$.

We choose the ``base points'' $b_1 = \alpha, b_2 = \infty$, $b_3 = \beta, b_4 = \infty$, so that $b_1 \neq b_2$, $b_3 \neq b_4$ and $h^0(C_{\F_p}, b_1 + b_2) = h^0(C_{\F_p}, b_3 + b_4) = 1$. As in Equation~(\ref{eq:par_choice}), we define $x_\alpha = x{-}1$, $x_\infty = z$, $x_\beta = x$ and $x_{D, \beta} = x_{D, \infty_-} = 1$, $x_{D, \infty} = z^{-1}$. 
For $Q$ in $\{ \infty, \beta, \alpha \}$ and $a \in \F_p$ let $Q_a$ be the unique $\Zpp$-point of $C$ that is congruent to $Q$ modulo $p$ and such that $x_Q(Q_a) = ap \in \Zpp$. We have the bijections
\[
\begin{aligned}
\F_p^2 & \lto J(\Zpp)_{\ol{D}}\,, \quad \lambda \longmapsto D_\lambda:= D + \alpha_{\lambda_1} - \alpha + \infty_{\lambda_2}- \infty  = \alpha_{\lambda_1}  + \infty_{\lambda_2} - 2 \infty \\
\F_p^2 & \lto J(\Zpp)_{\ol{E}}\,, \quad \mu \longmapsto E_\mu:= E + \beta_{\mu_1} - \beta + \infty_{\mu_2}- \infty  = \beta + \beta_{\mu_1} + \infty_{\mu_2} - \infty - 2 \infty_- \,.
\end{aligned}
\]
Following (\ref{eq:def_s_DE_lambda}) for $\lambda, \mu \in \F_p^2$ we define
\[
s_{D,E}(\lambda, \mu) := (\beta^* 1) \otimes (\beta_{\mu_1}^* 1) \otimes (\infty_{\mu_2}^* \frac{z^2}{z-\lambda_2 p}) \otimes (\infty^* \frac{z^2}{z-\lambda_2 p})^{-1} \otimes (\infty_-^* 1)^{-2}
\]
that, by Proposition~\ref{prop_P_and_M_and_norm} and Remark~\ref{rem_no_base_point_deg_0}, generates $E_\mu^*\calO_{C_\Zpp}(D_\lambda) = \calM_{D_\lambda, E_\mu}$.
The points in $\calM^\times(\F_p)$ projecting to $(\ol D, \ol E)$ are
in bijection with the elements $\xi$ in $\F_p^\times$ and are exactly
the points $\xi \cdot {s_{D,E}(0,0)}$. Using $(\Zpp)^\times =
\F_p^\times \times (1 {+} p \F_p)$, for each $\xi \in \F_p^\times$ we
parametrise the residue disk of $\xi \cdot {s_{D,E}(0,0)}$ using
bijection in Lemma~\ref{lem:parametrization_residue_disk_M}
\[
\F_p^5 \lto \calM^\times(\Zpp)_{\xi \cdot {s_{D,E}(0,0)}}  \,, \quad  (\lambda_1, \lambda_2, \mu_1, \mu_2, \tau) \longmapsto (1 + p\tau) \xi {\cdot} s_{D,E} ( (\lambda_1, \lambda_2), (\mu_1, \mu_2)) \,.
\]
Since $(m{\cdot} \circ \tr_c \circ f)(D_\lambda) = E_{-2\lambda}$ then we have 
\[ 
T(\Zpp)_{\ol{D}} = \bigcup_{\lambda \in \F_p^2} T_{D_\lambda}(\Zpp) = \bigcup_{\lambda \in \F_p^2} \calM_{D_\lambda, E_{-2\lambda}}^\times(\Zpp)	\,.
\]
As $\xi$ varies in $ \F_p^\times$ the point $\xi {\cdot} s_{D, E}(0,0)$ varies in all the points in  $\calM^\times(\F_p)$ projecting to $(\ol D, \ol E)$ and we have the following bijection induced by parameters in $\xi {\cdot} s_{D, E}(0,0)$
\begin{subeqn} \label{eq:param_T_modulo_p2_example}
\F_p^3 \lto T(\Z_p)_{\xi s_{D,E}(0,0)}\,, \quad (\lambda_1,\lambda_2, \tau ) \longmapsto (1 + \tau p) {\cdot} \xi {\cdot} s_{D,E}((\lambda_1, \lambda_2), (-2\lambda_1, -2\lambda_2)) \,.
\end{subeqn}
If we apply (\ref{eq:example_tilde_j_b_1}) and~(\ref{eq:example_tilde_j_b_0}) to $Q = \alpha_\lambda$ and we use the symmetry of the Poincar\'e torsor explained in Proposition~\ref{prop_P_and_M_and_norm} and made explicit in Lemma~\ref{lem:M_symmetry} we obtain the following description of $\wt{j_{b,i}}$ on $C(\Zpp)_{\alpha_{\F_p}}$ when $p \neq 2$
\[
\wt{j_{b,1}}(\alpha_\lambda) = (1/4) \cdot s_{D,E}((\lambda, 0), (-2\lambda, 0)) \,,  \quad  \wt{j_{b,0}}(Q)  = (1/12) \cdot s_{D,E}((\lambda, 0), (-2\lambda, 0)) \,.
\]
If $p=5$ then $18$ and  ${-}1$ are $(p-1)$-th roots of unity in $(\Zpp)^\times$, thus $1/4 = (-1) (1 + p)$ and $1/12 = 3 (1+2p)$ in $(\Zpp)^\times = \F_p^\times \times (1 {+} p \F_p)$, hence 
\begin{subeqn}
\wt{j_{b,1}}(\alpha_\lambda) = -(1+p) \cdot s_{D,E}((\lambda, 0), (-2\lambda, 0)) \,,  \quad  \wt{j_{b,0}}(Q)  =  18\cdot (1+2p) \cdot s_{D,E}((\lambda, 0), (-2\lambda, 0)) \,.
\end{subeqn}

Since it is useful for computing the map $\kappa_\Z$ in the residue disks of $T(\Zpp)$ projecting to $\ol D$, we also apply Lemma~\ref{lem:parametrization_residue_disk_M} to the residue disks of $\calM^\times(\Zpp)$ lying over $(\ol D, 0)$, $(0, \ol E)$ and $(0,0)$. Hence for $\lambda, \mu \in \F_p^2$  we define the divisors on $C_{\Zpp}$
\[
D_\lambda^0:= \alpha_{\lambda_1} - \alpha + \infty_{\lambda_2}- \infty \,, \quad  
E^0_{\mu}:= \beta_{\mu_1} - \beta + \infty_{\mu_2}- \infty 
\]
and the sections
\[
\begin{aligned}
s_{D,0}(\lambda, \mu) & := (\beta_{\mu_1}^* 1) {\otimes} (\infty_{\mu_2}^*  \frac{z^2}{z{-}\lambda_2 p}){\otimes} (\beta^* 1)^{-1} {\otimes} (\infty^*  \frac{z^2}{z{-}\lambda_2 p} )^{-1} & \text{in }\calM^\times(D_\lambda, E_\mu^0)(\Zpp) \\
s_{0,E}(\lambda, \mu) & := (\beta^* 1) {\otimes} (\beta_{\mu_1}^* 1) {\otimes} (\infty_{\mu_2}^* \frac{z}{z{-}\lambda_2 p}) {\otimes} (\infty^* \frac{z}{z{-}\lambda_2 p})^{-1} {\otimes} (\infty_-^* 1)^{-2} & \text{in }\calM^\times(D_\lambda^0, E_\mu)(\Zpp) \\
s_{0,0}(\lambda, \mu) & := (\beta_{\mu_1}^* 1) {\otimes} (\infty_{\mu_2}^*  \frac{z}{z{-}\lambda_2 p} ){\otimes} (\beta^* 1)^{-1} {\otimes} (\infty^* \frac{z}{z{-}\lambda_2 p} )^{-1} & \text{in }\calM^\times(D_\lambda^0, E_\mu^0)(\Zpp)\,. \\
\end{aligned}
\]

\subsection{Geometry mod~$p^2$ of integral points}\label{sec:example_kappa}
From now on $p=5$. Let $\ol \alpha \in C(\Zpp)$ be the image of
$\alpha \in C(\Z)$. In this subsection we compute the composition
$\ol \kappa \colon \Z^2 \to T(\Zpp)_{\wt{j_{b,1}}(\ol \alpha)}$
of the map
$\kappa_\Z \colon \Z^2 \to T(\Z_p)_{\wt{j_{b,1}}(\ol \alpha)}$
in~(\ref{eq:kappa_Z}) and the reduction map
$T(\Z_p)_{\wt{j_{b,1}}(\ol \alpha)} \to
T(\Zpp)_{\wt{j_{b,1}}(\ol \alpha)}$. With a suitable choice of
parameters in $\calO_{T, \wt{j_{b,1}}(\ol \alpha)}$, the map
$\kappa_\Z$ is described by integral convergent power series
$\kappa_1, \kappa_2, \kappa_3 \in \Z_p \langle z_1, z_2 \rangle$
and~$\ol \kappa$, composed with the inverse of the
parametrization~(\ref{eq:param_T_modulo_p2_example}), is given by the
images $\ol {\kappa_1}, \ol {\kappa_2}, \ol {\kappa_3}$ of
$\kappa_1, \kappa_2, \kappa_3$ in $\F_p[z_1,z_2]$.

The divisor $j_b(\ol \alpha)$ is equal to the image of 
\[
\wt{G_t} := e_{0,1}G_1 + e_{0,2} G_2  \text{ with }e_{0,1}:= 6 \,, e_{0,2}:= 3
\]
in $J(\F_p)$ and 
\[
\tilde t :=  Q\strut_{1,0}^{6} \otimes Q\strut_{2,0}^{3} \otimes Q\strut_{1,1}^{6{\cdot}6} \otimes Q\strut_{1,2}^{6{\cdot}3} \otimes Q\strut_{2,1}^{3{\cdot}6} \otimes Q\strut_{2,2}^{3{\cdot}3}  \quad \text{in }\calM^\times(\wt{D_1}\,, m{\cdot}(D_0 + \eta_*\wt{G_t}))(\Z)  
\]
is a lift of $ \wt{j_{b,1}}(\ol \alpha)$.
The kernel of $J(\Z) \to J(\F_p)$ is a free $\Z$-module generated by 
\[
\wt{G_1} := e_{1,1} G_1 + e_{1,2}G_2\,, \quad \wt{G_2} := e_{2,1} G_1 + e_{2,2}G_2\,, \text{ with }e_{1,1}:= 16 \,, e_{1,2}:= 2 \,, e_{2,1}:= 0 \,, e_{2,2}:= 5\,.
\]
Let $\wt{G_{t,2}}$ be the divisor $m (D_0{+}\eta_*(\wt{G_t}))$ representing $(m{\cdot}{\circ} \tr_c {\circ} f)(\wt{G_t}) \in J^0(\Z)$. Following (\ref{eq:Pij}) for $i,j \in \{1,2 \}$ we define 
\[
\begin{tikzcd}
\displaystyle
P_{i,j} := \bigotimes_{m,l = 1}^2 Q\strut_{l,m}^{e_{i,l} {\cdot} e_{j,m}} \arrow[d,mapsto] & \displaystyle
R_{i, \tilde t} :=  \bigotimes_{l=1}^2 Q\strut_{l,0}^{e_{i,l}} 	\otimes \bigotimes_{m,l = 1}^2 Q \strut_{l,m}^{e_{i,l} {\cdot} e_{0,m}} \arrow[d,mapsto]
& \displaystyle
S_{\tilde t,j} := \bigotimes_{m,l = 1}^2 Q\strut_{l,m}^{e_{0,l} {\cdot} e_{j,m}} \arrow[d,mapsto] \\
(\wt{G_i},f(m\wt{G_j}))
& (\wt{G_i}, \wt{G_{t,2}})) 
& ( \wt{G_t},f(m\wt{G_j})) \,.
\end{tikzcd}
\]
Computations in $C_{\Zpp}$ show the following linear equivalences of divisors
\[
\wt{G_t} \sim D_{0,3}\,, \quad \wt{G_1} \sim D^0_{4,0} \,, \quad \wt{G_2} \sim D^0_{0,3}
\]
and applying Lemma~\ref{lem:norm_and_lin_equiv} and the functoriality of the norm we compute
\begin{subeqn}\label{eq:example_Pij}
\begin{aligned}
P_{1,1} & = (1 + 4p){\cdot}s_{0,0}((4,0),(2,0)) \quad&  \text{in } \calM^\times(\wt{G_1}, \wt{G_1})(\Zpp) = \calM^\times(D_{4,0}^0, E_{2,0}^0)(\Zpp) \,, \\
P_{1,2} & = (1 + 4p){\cdot}s_{0,0}((4,0),(0,4)) \quad&  \text{in } \calM^\times(\wt{G_1}, \wt{G_2})(\Zpp) = \calM^\times(D_{4,0}^0, E_{0,4}^0)(\Zpp) \,, \\
P_{2,1} & = (1 + 4p){\cdot}s_{0,0}((0,3),(2,0)) \quad&  \text{in } \calM^\times(\wt{G_2}, \wt{G_1})(\Zpp) = \calM^\times(D_{0,3}^0, E_{2,0}^0)(\Zpp) \,, \\
P_{2,2} & = (-1){\cdot} (1 + 2p){\cdot}s_{0,0}((0,3),(0,4)) \quad&  \text{in } \calM^\times(\wt{G_2}, \wt{G_2})(\Zpp) = \calM^\times(D_{0,3}^0, E_{0,4}^0)(\Zpp) \,, \\
R_{1,\tilde t} & = s_{0,E}((4,0),(0,4)) \quad&  \text{in } \calM^\times(\wt{G_1}, \wt{G_{t,2}})(\Zpp) = \calM^\times(D_{4,0}^0, E_{0,4})(\Zpp) \,, \\
R_{2,\tilde t} & = (1 + 4p){\cdot}s_{0,E}((0,3),(0,4)) \quad&  \text{in } \calM^\times(\wt{G_2}, \wt{G_{t,2}})(\Zpp) = \calM^\times(D_{0,3}^0, E_{0,4})(\Zpp) \,, \\
S_{\tilde t,1} & = s_{D,0}((0,3),(2,0)) \quad&  \text{in } \calM^\times(\wt{G_t}, \wt{G_1})(\Zpp) = \calM^\times(D_{0,3}, E_{2,0}^0)(\Zpp) \,, \\
S_{\tilde t,2} & = (-1) (1 + 4p){\cdot}s_{D,0}((0,3),(0,4)) \quad&  \text{in } \calM^\times(\wt{G_t}, \wt{G_2})(\Zpp) = \calM^\times(D_{0,3}, E_{0,4}^0)(\Zpp) \,, \\
\tilde t & = (-1){\cdot} (1 + 2p){\cdot}s_{D,E}((0,3),(0,4)) \quad&  \text{in } \calM^\times(\wt{G_t}, \wt{G_{t,2}})(\Zpp) = \calM^\times(D_{0,3}, E_{0,4})(\Zpp) \,. \\
\end{aligned}
\end{subeqn}
We now show these computations in the cases of $\wt{G_t}$ and $\tilde t$. The Riemann-Roch space relative to the divisor $\wt{G_t} {+} \infty {+} \alpha {-} D$ on $C_\Zpp$ is generated by the inverse of the rational function 
\[
h_1 := \frac{ x^9 {-} 5x^8 {-} 2 x^7 {+} 7x^6 {-} 9x^5 {-} 5x^4 {+} 14x^3  {+} 7x^2 {+} 13x {+} 1 {+} (x^6  {+} 9x^5 {-}  5 x^4 {+} 15x^3 {-} 5x^2 {+} 4x {+}14) y } {15 x^5 - x^4 + 4 x^3 + 19x^2 + 4 x + 9 }
\]
and indeed 
\[
\divisor( h_1 ) = \wt{G_t} - D_{0,3} = (6\gamma + 3 \infty_- -3 \alpha - 6\infty ) - ( \alpha + \infty_3 - 2 \infty )
\quad  \text{in }\mathrm{Div}(C_\Zpp)\,.
\]
Hence multiplication by $h_1$ gives an isomorphism $\calO_{C_\Zpp}(\wt{G_t}) \to \calO_{C_\Zpp}(D_{0,3})$ and by functoriality of the norm we get
\[
\begin{aligned}
\delta^* \calO_{C}(\wt{G_t}) &\to \delta^* \calO_{C_\Zpp}(D_{0,3})\,, \quad  & \delta^*1  \mapsto \delta^*(h_1) = h_1(\delta){\cdot} \delta^*1 = 12 {\cdot} \delta^*1 \,, \\
\beta^* \calO_{C}(\wt{G_t}) & \to \beta^* \calO_{C_\Zpp}(D_{0,3}) \,, \quad & \beta^*1 \mapsto \beta^*(h_1) = h_1(\beta) {\cdot} \beta^*1 = 18 {\cdot} \beta^*1 \,, \\
\infty ^* \calO_{C}(\wt{G_t}) & \to \infty^* \calO_{C_\Zpp}(D_{0,3}) \,, \quad & \infty^*z^6 \mapsto \infty^*(z^6h_1) = 
13 {\cdot} \infty^*\frac{z^2}{z-3p}\,, \\
\infty_- ^* \calO_{C}(\wt{G_t}) & \to \infty_-^* \calO_{C_\Zpp}(D_{0,3}) \,, \quad & \infty_-^*z^{-3} \mapsto \infty_-^*(z^{-3}h_1) = (z^{-3}h_1)(\infty_-) {\cdot} \infty_-^*1 = 
6 {\cdot} \infty_-^* 1 \,.
\end{aligned}
\]
Since $\wt{G_{t,2}} = 12 \delta + 4 \infty_- - 6 \beta - 10 \infty$, the above isomorphisms, tensored with the exponents, give the canonical isomorphism
\begin{subeqn}\label{eq:example_functoriality_norm}
\calM(\wt{G_t}, \wt{G_{t,2}}) = \wt{G_{t,2}}^*  \calO_{C_\Zpp}(\wt{G_t})  \to \wt{G_{t,2}}^*  \calO_{C_\Zpp}(D_{0,3}) = \calM(D_{0,3}, \wt{G_{t,2}}) 
\end{subeqn}
\[
\tilde t {=} 14{\cdot} (\delta^*1)^{12}{\otimes} (\beta^*1)^{-6} {\otimes}  (\infty^*z^6)^{-10} {\otimes} (\infty_-^*z^{-3})^4  \mapsto 14{\cdot} (\delta^*1)^{12}{\otimes} (\beta^*1)^{-6} {\otimes} 
(\infty^*\frac{z^2}{z{-}3p})^{-10} {\otimes} (\infty_-^*1)^4\,.
\]
The Riemann-Roch space relative to the divisor $\wt{G_{t,2}} {+} \infty {+} \alpha {-} E$ on $C_\Zpp$ is generated by the inverse of the rational function 
\[
\begin{aligned}
h_2 {:}{=} \frac{ x^{17}  {-} 8 x^{16} {+} x^{15} {-} 4x^{14} {+} 7x^{13} {+} 4x^{12} {+} 12 x^{11} {+} x^{10} {+}2 x^{9} -5x^8 {+} x^7 {+} 3 x^6 {+} 12 x^5 {-} 6 x^4 {-} 6 x^3 {+} 4 x^2 {-}6} { 20 x^8 - 6 x^7 }\\
{+} \frac{ 10 x^2 {+} (x^{15} {+} 6x^{14} {-} 5x^{13} {-} x^{12} {-} 2x^{11} {+} 14x^{10} {-} 4x^9 {+} 14x^8 {+} 3x^7 {+} 8x^6 {-} 6x^5 {-} 3x^4 {+} 4x^3 {+} 13x^2 {-} x {-} 7 )y}{ 20 x^9 - 6 x^8 }
\end{aligned}
\]
and indeed 
\[
\divisor( h_2 ) = \wt{G_{t,2}} - E_{0,4} = (12 \delta + 4 \infty_- - 6 \beta - 10 \infty) - ( 2 \beta + \infty_4 - \infty - \infty_- )
\quad  \text{in }\mathrm{Div}(C_\Zpp)\,.
\]
Following the recipe in Section~\ref{sec:iso_norm_along_equiv_D} that describes the map~(\ref{eq:phi_dirty}), we consider the following rational section of $\calO_{C_\Zpp}(D_{0,3})$
\[
l := \frac{ 10x^4 + x^3 + 17x + 14 + (15x + 9)y }{ 10x^4 + 16x^3 + 7x^2 + 7x + 10 } \,.
\]
since it generates $\calO_{C_\Zpp}(D_{0,3})$ in a neighborhood of the supports of $\wt{G_{t,2}}$ and $E_{0,4}$. Then
\[
\divisor(l) = 3{\cdot}(-1,1) {+} (17,23) {+} (15,10) {-} 2{\cdot}(12,23) {-} 2 {\cdot}(5,20) {-} (0,1) \text{ in } \mathrm{Div}(V_{1,\Zpp}) {\subset} \mathrm{Div}(C_{\Zpp})\,.
\]
Hence by Lemma~\ref{lem:norm_and_lin_equiv} the canonical isomorphism 
\[
\calM(D_{0,3}, \wt{G_{t,2}})  = \wt{G_{t,2}}^* \calO_{C_\Zpp}(D_{0,3}) \lto E_{0,4}^* \calO_{C_\Zpp}(D_{0,3}) = \calM(D_{0,3}, E_{0,4}) 
\]
described in Equation~(\ref{eq:norm_iso_equiv_div_abstract}) sends
\begin{subeqn}\label{eq:example_iso_norm_along_equiv_D}
\wt{G_{t,2}}^* l  \longmapsto h_2(\divisor(l)) \cdot E_{0,4}^* l = 14 \cdot E_{0,4}^* l \,.
\end{subeqn}
where
\[
\begin{aligned}
\wt{G_{t,2}}^* l & := (\delta^*l)^{12}{\otimes} (\beta^*l)^{-6} {\otimes} 
(\infty^*l)^{-10} {\otimes} (\infty_-^*l)^4  = - (\delta^*1)^{12}{\otimes} (\beta^*1)^{-6} {\otimes} 
(\infty^*\frac{z^2}{z{-}3p})^{-10} {\otimes} (\infty_-^*1)^4 \,, \\ 
E_{0,4}^* l & := (\beta^* l)^2 {\otimes} (\infty_4^* l) {\otimes}(\infty^* l)^{-1} {\otimes} (\infty^*_- l)^{-2} = 16 {\cdot}  (\beta^* 1)^2 {\otimes} (\infty_4^*\frac{z^2}{z{-}3p} ) {\otimes}(\infty^* \frac{z^2}{z{-}3p})^{-1} {\otimes}(\infty_-^* 1)^{-2} \,.
\end{aligned}
\]
Equations~(\ref{eq:example_functoriality_norm}) and~(\ref{eq:example_iso_norm_along_equiv_D}) 
imply that $\tilde t = -(1 + 2p) {\cdot} s_{D,E}((0,3),(0,4))$. 

Let $\ol{A_{\tilde t}}, \ol{B_{\tilde t}}, \ol C$ and $\ol{D_{\tilde t}}$ be the compositions of the reduction map $\calM^\times(\Z_p) \to \calM(\Zpp)$ and respectively $A_{\tilde t}, B_{\tilde t}, C$ and $D_{\tilde t}$, defined in~(\ref{eq:A_and_B}), (\ref{eq:C}) and (\ref{eq:D_ttilde_n}). Using (\ref{eq:explicit_partial_group_laws}) and (\ref{eq:example_Pij}) we get, for $n$ in $\Z^2$,
\begin{subeqn}\label{eq:example_ABCD}
\begin{aligned}
\ol{A_{\tilde t}}(n) & = (-1)^{n_2}(1+(4n_2)t) \cdot s_{D,0}((0,3),(2n_1 ,4n_2 ))      \,, \\
\ol{B_{\tilde t}}(n) & = (1 + (4n_2)p)s_{0,E}((4n_1 ,3n_2 ), (0,4) )   \,, \\
\ol C (n) & =   (-1)^{n_2^2} (1 + (4n_1^2 + (4+4) n_1n_2  + 2 n_2^2)p) \cdot  s_{0,0}((4n_1,3n_2),(2n_1, 4n_2))    \,,\\
\ol{D_{\tilde t}}(n) & = - (1 + (4n_1^2 + 3 n_1n_2  + 2 n_2^2 + 3 n_2 + 2)p) \cdot s_{D,E}((4n_1,3 + 3n_2),(2n_1,4+4n_2))\,, \\
\ol{\kappa}(n) & = - (1 + (4n_1^2 + 3 n_1n_2  + 2 n_2^2 + 2 n_2 + 2)p) \cdot s_{D,E}((n_1,3 + 2n_2),(3n_1,4+n_2)) \,,
\end{aligned}
\end{subeqn}
hence, using the bijection~(\ref{eq:param_T_modulo_p2_example}),
\begin{subeqn}\label{eq:example_kappa}
 \ol{\kappa_1} = z_1 \,, \quad  \ol{\kappa_2} = 3 + 2z_2 \,, \quad \ol{\kappa_3} =4z_1^2 + 3 z_1z_2  + 2 z_2^2 + 2 z_2 + 2  \,.
\end{subeqn}

\subsection{The rational points with a specific image mod~5.}
By~(\ref{eq:example_ABCD}) the image in $T(\F_p)$ of a point $\pm \ol{D_{\tilde t}}(n)$ for $n \in \Z^2$ is always of the form $\pm s_{D,E}(0,0)$, hence, looking at (\ref{eq:example_tilde_j_b_0}) we see that there is no point $T(\Z)$ with reduction $\wt{j_{b,0}}(\ol \alpha) \in T(\F_p)$. Hence $C(\Z)_{\ol \alpha} = U_1(\Z)_{\ol \alpha}$.

Let $f_1,f_2 \in \calO(\wt{T}_t^p)^{\wedge_p}$ be generators of the kernel of 
$\wt{j_{b,1}}^*\colon \calO(\wt{T}_t^p)^{\wedge_p}\to
\calO(\wt{U}_u^p)^{\wedge_p}$ as in Section~\ref{sec:closure-finiteness}.
The bijection (\ref{eq:param_T_modulo_p2_example}) gives an isomorphism  $\F_p\otimes\calO(\wt{T}_t^p) = \F_p[\lambda_1, \lambda_2, \tau]$
and since the images $\ol{f_1}, \ol{f_2}$ of $f_1,f_2$ in $\F_p\otimes\calO(\wt{T}_t^p)$ are generators of the kernel of
$\wt{j_{b,1}}^*\colon \F_p \otimes \calO(\wt{T}_t^p)^{\wedge_p}\to
\F_p \otimes \calO(\wt{U}_u^p)^{\wedge_p}$ we can suppose that 
\[
\ol{f_1} = \lambda_2 \,, \quad \ol{f_2} = \tau - 1 \,.
\]
By (\ref{eq:example_kappa}) we have 
\[
\kappa^*\ol{f_1} = \ol{\kappa_2} = 3 + 2\\z_2  \,, \quad \kappa^*\ol{f_2} = \ol{\kappa_3}-1 = 4z_1^2 + 3 z_1z_2  + 2 z_2^2 + 2 z_2 + 1  \,.
\]
Let $A$ be $\Z_p\langle z_1,z_2\rangle/(\kappa^*f_1, \kappa^*f_2)$. Then the ring
\begin{subeqn}\label{eq:example_F_p_algebra_A}
\ol{A} := A/pA = \F_p[z_1,z_2]/(\kappa^*\ol{f_1}, \kappa^*\ol{f_2}) =
\F_p[z_1,z_2] /(z_2 -1 \,,  4z_1^2 + 3z_1 ) 
\end{subeqn}
has dimension $2$ over $\F_p$, hence by Theorem~\ref{thm:finiteness} $U(\Z)_{\ol \alpha}$ contains at most $2$ points. Since both 
\[
\alpha \quad \text{and} \quad (12/7 , 20/7) \in V_1(\Z[1/7]) 
\]
reduce to $\ol \alpha$ we deduce that $C(\Z)_{\ol \alpha} = U_1(\Z)_{\ol \alpha}$ is made of the these two points.
\subsection{Determination of all rational points}
Denoting $(3,-1) \in V_1(\F_p) \subset C(\F_p)$ as $\eps$ we have
\[
C(\F_p) = \{ \ol{\infty}\,, \ol{\infty_-}\,, \ol \alpha\,, \iota(\ol \alpha)\,, \eta(\ol \alpha)\,, (\iota \circ \eta)(\ol \alpha)\,,\ol \gamma\,, \iota(\ol \gamma)\,, \eta(\ol \gamma)\,, (\iota \circ \eta)(\ol \gamma)\,, \eps\,, \iota(\eps)  \}\,.
\]
Using that for any point $Q$ in $C(\F_p)$ the condition $T(\Z)_{\wt{j_{b,i}}(Q)} = \emptyset$ implies $U_i(\Z)_Q = \emptyset$ we get
\[
U_0(\Z)_{\ol{\infty}} = U_0(\Z)_{\ol{\infty_-}}= U_1(\Z)_{\eps} = U_1(\Z)_{\iota(\eps)} = U_1(\Z)_{\ol \gamma} =U_1(\Z)_{\eta (\ol\gamma)} = U_1(\Z)_{\eta (\ol \gamma)} = U_1(\Z)_{\iota\eta(\ol \gamma)} = \emptyset \,.
\]
Applying our method to $\ol{\infty}$ we discover that $U_1(\Z)_{\ol \infty}$ contains at most $2$ points and the same holds for $U_1(\Z)_{\ol{\infty_-}}$. Moreover the action of $\langle \eta, \iota \rangle $ on $C(\Z)$ tells that $U_1(\Z)_{\iota(\ol \alpha)}$, $U_1(\Z)_{\eta(\ol \alpha)}$ and $U_1(\Z)_{\eta \iota(\ol \alpha)}$ are sets containing exactly $2$ elements. Hence
\[
U_1(\Z) = U_1(\Z)_{\ol \alpha} \cup U_1(\Z)_{\iota(\ol \alpha)}\cup U_1(\Z)_{\eta(\ol \alpha)} \cup U_1(\Z)_{\eta \iota(\ol \alpha)} \cup U_1(\Z)_{\ol{\infty_-}} \cup U_1(\Z)_{\ol{\infty}}
\]
contains at most $12$ elements. Looking at the orbits of the action of  $\langle \eta, \iota \rangle $ on $U_1(\Z)$ we see that $\# U_1(\Z) \equiv 2 \pmod 4$, hence $\# U_1(\Z) \leq 10$. Since $U_1(\Z)$ contains $\infty, \infty_-$ and all the images by $\langle \eta, \iota\rangle$ of $U_1(\Z)_{\ol \alpha}$ we conclude that $\# U_1(\Z) = 10$.

Applying our method to the point $\ol \gamma$ we see that $U_0(\Z)_{\ol \gamma}$ contains at most two points, one of them being $\gamma$. Moreover solving the equations $\kappa^*\ol{f_i} = 0$ we see that if there is another point $\gamma'$ in $U_0(\Z)_{\ol \gamma}$ then there exist $n_1, n_2 \in \Z$ such that
\[
j_b(\gamma') = 39 G_1 + 17 G_2 + 5n_1\wt{G_1} + 5 n_2 \wt{G_2}\,.
\]  
Using the Mordell-Weil sieve (see \cite{Muller}) we derive a contradiction: for all integers $n_1, n_2$, the image in $J(\F_7)$ of  $39 G_1 {+} 17 G_2 {+} 5n_1\wt{G_1} {+} 5 n_2 \wt{G_2}$ is not contained in $j_b(C(\F_7))$. We deduce that
\[
U_0(\Z)_{\ol \gamma} = \{ \gamma \}\,.
\]

Applying our method to to $\eps$ we see that $U_0(\Z)_\eps$ contains at most $2$ points corresponding to two different solutions to the equations $\kappa^*\ol{f_i} = 0$. We can see that one of the two solutions does not lift to a point in $U_0(\Z)_\eps$ in the same way we excluded the existence of $\gamma' \in U_0(\Z)_{\ol \gamma}$. Hence $U_0(\Z)_\eps$ has cardinality at most $1$. Using that for every $Q \in C(\F_p)$ and every automorphism $\omega$ of $C$ we have $\# U_0(\Z)_{Q} = \# U_0(\Z)_{\omega(Q)}$, we deduce that 
\[
U_0(\Z) = U_0(\Z)_{\ol \gamma} \cup U_0(\Z)_{\iota(\ol \gamma)}\cup U_0(\Z)_{\eta(\ol \gamma)} \cup U_0(\Z)_{\eta \iota(\ol \gamma)} \cup U_0(\Z)_{\eps} \cup U_0(\Z)_{\iota(\eps)}
\]
contains at most $6$ points. Looking at the orbits of the action of  $\langle \eta, \iota \rangle $ on $U_0(\Z)$ we see that $\# U_0(\Z) \equiv 0 \pmod 4$, hence $\# U_4(\Z) \leq 4$, and since $U_0(\Z)$ contains the orbit of $\gamma$ we conclude that $\# U_0(\Z) = 4$. Finally
\[
\# C(\Z) = \# U_0(\Z) + \# U_1(\Z) = 4 + 10 = 14 \,.
\]

\section{Some further remarks}\label{sec:rems}

\subsection{Complex uniformisations of some of the objects involved}
\label{sec:fund_grps}
Let $C$ be a projective curve over~$\Q$, smooth, and geometrically
irreducible, and let $g$ be its genus. The universal cover of
$P^\times(\C)$ is described in~\cite{Be-Ed}, Propositions~4.5
and~4.6. The covering space, denoted~$D_\tau$, is
$\rmM_{1,g}(\C)\times\rmM_{g,1}(\C)\times\C $, hence a $\C$-vector
space of dimension $2g+1$. The biextension structure on
$\rmM_{1,g}(\C)\times\rmM_{g,1}(\C)\times\C $  is trivial, that is,
for all $x$, $x_1$, $x_2$ in~$\rmM_{1,g}(\C)$, all $y$, $y_1$, $y_2$
in~$\rmM_{g,1}(\C)$, and all $z_1$, $z_2$ in~$\C$, we have:
\begin{subeqn} \label{eq:trivial_biext}
\begin{aligned}
  (x_1,y,z_1) +_1 (x_2,y,z_2) & = (x_1+x_2,y,z_1+z_2)\,,\\
  (x,y_1,z_1) +_2 (x,y_2,z_2) & = (x,y_1+y_2,z_1+z_2)\,.
\end{aligned}
\end{subeqn}
The fundamental group $\pi_1(P^\times(\C),1)$ is
\begin{subeqn}
\label{eq:pi_1_Px}
Q^u(\Z) := 
\left\{
\begin{pmatrix}
1 & x & z\\
0 & 1_{2g} & y\\
0 & 0 & 1    
\end{pmatrix} :
\text{$x\in\rmM_{1,2g}(\Z)$, $y\in\rmM_{2g,1}(\Z)$, $z\in \Z$}
\right\},
\end{subeqn}
also known as a Heisenberg group. Its action on $D_\tau$ is given
in~\cite[(4.5.3)]{Be-Ed}.

Now recall the definition of $T$ in~(\ref{eq:def_T}). As
$\rmM_{2g,1}(\Z)$ is the lattice of~$J(\C)$, and $\rmM_{1,2g}(\Z)$ the
lattice of~$J^\vee(\C)$, each $f_i$ is given by an antisymmetric
matrix $f_{i,\Z}$ in $\rmM_{2g,2g}(\Z)$ such that for all $y$ in
$\rmM_{2g,1}(\Z)$ we have $f_i(y) = y^t{\cdot}f_{i,\Z}$, and by a
complex matrix $f_{i,\C}$ in $\rmM_{g,g}(\C)$ such that for all $v$ in
$\rmM_{g,1}(\C)$, for each $i$ we have $f_i(v)=v^t{\cdot}f_{i,\C}$
in~$\rmM_{1,g}(\C)$. For more details about this description of the
$f_i$ see the beginning of~\cite [\S4.7]{Be-Ed}. Then we have
\begin{subeqn}\label{eq:pi_1_T}
\pi_1(T(\C)) =
\left\{
\begin{pmatrix}
1_{\rho-1} & m{\cdot}f(y) & z\\
0 & 1_{2g} & y\\
0 & 0 & 1    
\end{pmatrix} :
\text{$y\in\rmM_{2g,1}(\Z)$, $z\in \rmM_{\rho-1,1}(\Z)$}
\right\},  
\end{subeqn}
with $m{\cdot}f(y)\in\rmM_{\rho-1,2g}(\Z)$ with rows the
$m{\cdot}y^t{\cdot}f_{i,\Z}$.  So, $\pi_1(T(\C))$ is a central
extension of $\rmM_{2g,1}(\Z)$ by $\rmM_{\rho-1,1}(\Z)$, with
commutator pairing sending $(y,y')$ to
$(2my^t{\cdot}f_{i,\Z}{\cdot}y')_i$.

The universal covering $\wt{T(\C)}$ is given by
\begin{subeqn}\label{eq:univ_cover_T}
  \begin{aligned}
\wt{T(\C)} & =
  \{(m{\cdot}(c+f(v)),v,w) :
  \text{$v\in \rmM_{g,1}(\C)$, $w\in \rmM_{\rho-1,1}(\C)$}\}\\
  & \subset
  \rmM_{\rho-1,g}(\C)\times\rmM_{1,g}(\C)\times\rmM_{\rho-1,1}(\C)\,,    
  \end{aligned}
\end{subeqn}
with $m{\cdot}(c+f(v))\in \rmM_{\rho-1,g}(\C)$ with rows the
$m{\cdot}(\wt{c_i}+v^t{\cdot}f_{i,\C})$ with $\wt{c_i}$ a lift of
$c_i$ in~$\rmM_{1,g}(\C)$. The action of $\pi_1(T(\C),1)$ on
$\wt{T(\C)}$ is given again, with the necessary changes, by~\cite[(4.5.3)]{Be-Ed}.

Now that we know $\pi_1(T(\C),1)$ we investigate which quotient of
$\pi_1(C(\C),b)$ it is, via $\wt{j_b}\colon C(\C)\to T(\C)$. We
consider the long exact sequence of homotopy groups induced by the
$\C^{\times,\rho-1}$-torsor $T(\C)\to J(\C)$, taking into account that
$\C^{\times,\rho-1}$ is connected and that $\pi_2(J(\C))=0$:
\begin{subeqn}\label{eq:hom_top_les}
\begin{tikzcd}
  \pi_1(\C^{\times,\rho-1},1) \arrow[r, hook] &
  \pi_1(T(\C),1) \arrow[r, two heads] &
  \pi_1(J(\C),0)\,.
\end{tikzcd}
\end{subeqn}
Again we see that $\pi_1(T(\C),1)$ is a central extension of the free
abelian group $\pi_1(J(\C),0)$ by~$\Z^{\rho-1}$, and from the matrix
description we know that the $i$th coordinate of the commutator
pairing is given by
$mf_i\colon \rmH_1(J(\C),\Z)\to
\rmH_1(J^\vee(\C),\Z)=\rmH_1(J(\C),\Z)^\vee$.  The $\Z$-module of
antisymmetric $\Z$-valued pairings on $\rmH_1(J^\vee(\C),\Z)$ is
$\bigwedge^2\rmH^1(J(\C),\Z)=\rmH^2(J(\C),\Z)$, and $mf_i$ is the
cohomology class (first Chern class) of the $\C^\times$-torsor~$T_i$:
\begin{subeqn}\label{eq:c1_Ti}
  mf_i = c_1(T_i) \quad\text{in $\rmH^2(J(\C),\Z)$.}
\end{subeqn}
There is a central extension
\begin{subeqn}\label{eq:univ_centr_ext}
\begin{tikzcd}
  \rmH_2(J(\C),\Z) \arrow[r, hook]&
  E \arrow[r, two heads] &
  \pi_1(J(\C),0)
\end{tikzcd}
\end{subeqn}
that is universal in the sense that every central extension of
$\pi_1(J(\C),0)$ by a free abelian group arises by pushout from
$\rmH_2(J(\C),\Z)$. We denote
\begin{subeqn}\label{eqn:G=pi_1_C}
G:=\pi_1(C(\C),b)\,.
\end{subeqn}
The map $j_b\colon C\to J$ gives $G\to\pi_1(J(\C),0)$, and this is the
maximal abelian quotient.  The second quotient in the descending
central series of $G$ gives the central extension:
\begin{subeqn}\label{eqn:2nd_dcs_G}
\begin{tikzcd}
  {}[G,G]/[G,[G,G]] \arrow[r, hook]&
  G/[G,[G,G]] \arrow[r, two heads]&
  G/[G,G] = G^\mathrm{ab} = \pi_1(J(\C),0)\,.
\end{tikzcd}
\end{subeqn}
This extension (\ref{eqn:2nd_dcs_G}) arises from
(\ref{eq:univ_centr_ext}) by pushout via a morphism from
$\rmH_2(J(\C),\Z)$ to $[G,G]/[G,[G,G]]$:
\begin{subeqn}\label{eqn:push-out-pi_1}
\begin{tikzcd}
  \rmH_2(J(\C),\Z) \arrow[r, hook] \arrow[d] &
  E \arrow[r, two heads] \arrow[d] &
  G^\mathrm{ab} \arrow[d, equals] \\
  {}[G,G]/[G,[G,G]] \arrow[r, hook] &
  G/[G,[G,G]] \arrow[r, two heads] &
  G^\mathrm{ab} \,.
\end{tikzcd}
\end{subeqn}
The left vertical arrow is surjective because commutators of lifts in
$E$ of elements of $G^\mathrm{ab}$ are mapped to the commutators of
lifts in $G/[G,[G,G]]$, and so give generators of $[G,G]/[G,[G,G]]$.

From the usual presentation of $G$ with generators
$\alpha_1,\beta_1,\ldots,\alpha_g,\beta_g$, with the only relation
$[\alpha_1,\beta_1]\cdots[\alpha_g,\beta_g]=1$, we see that the
obstruction in lifting $G\to G^\mathrm{ab}$ to $G\to E$
in the top row of~(\ref{eqn:push-out-pi_1}) is the image of
$[\alpha_1,\beta_1]\cdots[\alpha_g,\beta_g]$ in
$\rmH_2(J(\C),\Z)$. This image is a generator of the image of
$\rmH_2(C(\C),\Z)$ under~$j_b$. So the pushout
in~(\ref{eqn:push-out-pi_1}) factors through the pushout by the
quotient  of $\rmH_2(J(\C),\Z)$ by $\rmH_2(C(\C),\Z)$:
\begin{subeqn}\label{eqn:push-out-pi_1-2}
\begin{tikzcd}
  \rmH_2(J(\C),\Z)/\rmH_2(C(\C),\Z) \arrow[r, hook] \arrow[d, two heads] &
  E' \arrow[r, two heads] \arrow[d, two heads] &
  G^\mathrm{ab} \arrow[d, equals] \\
  {}[G,G]/[G,[G,G]] \arrow[r, hook] &
  G/[G,[G,G]] \arrow[r, two heads] &
  G^\mathrm{ab} \,.
\end{tikzcd}
\end{subeqn}
Using again the presentation of $G$ we can split this morphism of
extensions, and, using that $\rmH_2(J(\C),\Z)/\rmH_2(C(\C),\Z)$ is
generated by commutators of lifts of elements of $G^\mathrm{ab}$,
conclude that all vertical arrows in~(\ref{eqn:push-out-pi_1-2}) are
isomorphisms.

In particular, we have that $[G,G]/[G,[G,G]]$ is the same as
$\rmH_2(J(\C),\Z)/\rmH_2(C(\C),\Z)$. From (\ref{eq:c1_Ti}) we see that
the sub-$\Z$-module of $\rmH^2(J(\C),\Z(1))$ (note the Tate twist, now
we take the Hodge structures into account) spanned by the $mf_i$ is
obtained in 4 steps: take the kernel of
$\rmH^2(J(\C),\Z(1))\to\rmH^2(C(\C),\Z(1))$, take the $(0,0)$-part
(then we are dealing with $\Hom(J(\C),J^\vee(\C))^+$), then
$\Gal(\ol{\Q}/\Q)$ acts (because $\Hom(J(\C),J^\vee(\C))^+$ and
$\Hom(J_\Qbar,J^\vee_\Qbar)^+$ are equal), through the Galois group of
a finite extension of~$\Q$, take the invariants (because we only use
morphisms $f_i$ defined over~$\Q$), then take the image of the
multiplication by $m$ on that.

Dually, this means that $\pi_1(T(\C),1)$ arises as the pushout
\begin{subeqn}
\begin{tikzcd}
  \rmH_2(J(\C),\Z(-1))/\rmH_2(C(\C),\Z(-1)) \arrow[r, hook] \arrow[d] &
  G/[G,[G,G]] \arrow[r, two heads] \arrow[d] &
  G^\mathrm{ab} \arrow[d, equals] \\
  \left((\rmH_2(J(\C),\Z(-1))/\rmH_2(C(\C),\Z(-1)))_{(0,0)}\right)_{\Gal(\ol{\Q}/\Q)}
  \arrow[r, hook] &
  \pi_1(T(\C),1) \arrow[r, two heads] &
  G^\mathrm{ab}\,,
\end{tikzcd}
\end{subeqn}
where the subscript $(0,0)$ means the largest quotient of type
$(0,0)$, where the subscript $\Gal(\ol{\Q}/\Q)$ means co-invariants
modulo torsion, and where the left vertical map is $m$ times the
quotient map. We repeat that the morphism from $\pi_1(C(\C))=G$ to
$\pi_1(T(\C),1)$ given by the middle vertical map is induced by
$\wt{j_b}\colon C(\C)\to T(\C)$.

\subsection{Finiteness of rational points}\label{sec:finiteness}

In this section we reprove Faltings's finiteness result~\cite{Fal} in the special
case where $r<g+\rho-1$. This was already done in~\cite{Ba-Do-1},
Lemma~3.2 (where the base field is either $\Q$ or imaginary
quadratic).  We begin by collecting some ingredients on
good formal coordinates of the $\Gm$-biextension $P^{\times,\rho-1}\to J\times
J^{\vee,\rho-1}$ over~$\Q$, and on what $C$ looks like in such coordinates.

\subsubsection{Formal trivialisations}\label{sec:formal_triv}
Let $A$, $B$ and $G$ be connected smooth commutative group schemes
over a field $k\supset\Q$, and let $E\to A\times B$ be a commutative
$G$-biextension. Let $a$ be in $A(k)$, $b\in B(k)$ and $e\in
E(k)$. For $n\in\N$, let $A^{a,n}$ be the $n$th infinitesimal
neighborhood of $a$ in~$A$, hence its coordinate ring
is~$\calO_{A,a}/m_a^{n+1}$. We use similar notation for $B$ with~$b$,
and $E$ with~$e$, and also for the points $0$ of $A$, $B$ and~$E$,
and, similarly, the formal completion of $A$ at $a$ is denoted
by~$A^{a,\infty}$, etc. We also use such notation in a relative
context, for example, for the group schemes $E\to B$ and $E\to A$. We
view completions as $A^{a,\infty}$ as set-valued functors on the
category of local $k$-algebras with residue field $k$ such that every
element of the maximal ideal is nilpotent. For such a $k$-algebra~$R$,
$A^{a,\infty}(R)$ is the inverse image of $a$ under $A(R)\to
A(k)$. Then $A^{0,\infty}$ is the formal group of~$A$.

We now want to show that the formal $G^{0,\infty}$-biextension
$E^{0,\infty}\to A^{0,\infty}\times B^{0,\infty}$ is isomorphic to the
trivial biextension (the object
$G^{0,\infty}\times A^{0,\infty}\times B^{0,\infty}$ with $+_1$ given
by addition on the 1st and 2nd coordinate, and $+_2$ by addition on
the 1st and 3rd coordinate). As exp for $A^{0,\infty}$ gives a
functorial isomorphism
$T_{A/k}(0)\otimes_k{\Ga}_k^{0,\infty}\to A^{0,\infty}$, and similarly
for $B$ and~$G$, it suffices to prove this triviality for
$\Ga^{0,\infty}$-biextensions of $\Ga^{0,\infty}\times \Ga^{0,\infty}$
over~$k$. One easily checks that the group of automorphisms of the
trivial $\Ga^{0,\infty}$-biextension of
$\Ga^{0,\infty}\times \Ga^{0,\infty}$ over~$k$ that induce the
identity on all three $\Ga^{0,\infty}$'s is~$(k,+)$, with $c\in k$ acting
as $(g,a,b)\mapsto (g+cab,a,b)$. As this group is commutative, it then
follows that the group of automorphisms of the $G^{0,\infty}$-biextension
$E^{0,\infty}\to A^{0,\infty}\times B^{0,\infty}$ that induce identity
on $G^{0,\infty}$, $A^{0,\infty}, $and~$B^{0,\infty}$, is equal to the
$k$-vector space of $k$-bilinear maps
$T_{A/k}(0)\times T_{B/k}(0) \to T_{G/k}(0)$. This indicates how to
trivialise~$E^{0,\infty}$. We choose a section~$\tilde{e}$ of the
$G$-torsor $E\to A\times B$ over the closed subscheme
$A^{0,1}\times B^{0,1}$ of $A\times B$:
\[
\begin{tikzcd}
  & E \arrow[d] \\
A^{0,1}\times B^{0,1} \arrow[r] \arrow{ru}{\tilde{e}} & A\times B\,,
\end{tikzcd}\quad
\text{with $\tilde{e}(0,0)=e$ in~$E(k)$.}
\]
Note that
\[
  \calO(A^{0,1}\times B^{0,1})=
  (k\oplus m_{A^{0,1}})\otimes(k\oplus m_{B^{0,1}})
  =
  k\oplus m_{A^{0,1}}\oplus m_{B^{0,1}}\oplus(m_{A^{0,1}}\otimes m_{B^{0,1}})\,.
\]
Hence two such~$\tilde{e}$ differ by a $k$-algebra morphism from
$k\oplus m_{G^{0,2}}=k\oplus m_{G^{0,1}}\oplus \Sym^2 m_{G^{0,1}}$
(use the exponential map) to
$k\oplus m_{A^{0,1}}\oplus m_{B^{0,1}}\oplus(m_{A^{0,1}}\otimes
m_{B^{0,1}})$, hence by a triple of $k$-linear maps from $m_{G^{0,1}}$ to
$m_{A^{0,1}}$, $m_{B^{0,1}}$, and $m_{A^{0,1}}\otimes m_{B^{0,1}}$.
The linear maps $m_{G^{0,1}} \to m_{A^{0,1}}$ and
$m_{G^{0,1}} \to m_{B^{0,1}}$ correspond to the differences on
$A^{0,1}\times B^{0,0}$ and on $A^{0,0}\times B^{0,1}$,
respectively. There are unique such linear maps such that the adjusted
$\tilde{e}$ is compatible with the given trivialisations of
$E\to A\times B$ over $A^{0,1}\times B^{0,0}$ and over
$A^{0,0}\times B^{0,1}$. In geometric terms, $\tilde{e}$, assumed to
be adjusted, is then a splitting of
$T_G(0)_B\into T_{E/B}(0)\onto T_A(0)_B$ over $B^{0,1}$ that is
compatible with the already given splitting over $0\in B(k)$, and it
is also a splitting of $T_G(0)_A\into T_{E/A}(0)\onto T_B(0)_A$ over
$A^{0,1}$ that is compatible with the already given splitting over
$0\in A(k)$. The splitting over $B^{0,1}$ gives an isomorphism from
$(T_G(0)\oplus T_A(0))_{B^{0,1}}$ to $(T_{E/B})_{B^{0,1}}$. So the
exponential map, for~$+_1$, for the pullback to $B^{0,1}$ of $E\to B$,
gives an isomorphism of formal groups over~$B^{0,1}$:
\[
  \begin{tikzcd}
  \left((T_G(0)\oplus T_A(0))\otimes_k \Ga^{0,\infty}\right)_{B^{0,1}}
\arrow[r, hook, two heads] 
  & E_{B^{0,1}}^{0,\infty} \,.
  \end{tikzcd}
\]
Viewing $E_{B^{0,1}}^{0,\infty}$ as the tangent space at the zero
section of the pullback to~$A^{0,\infty}$ of~$E\to A$, this
isomorphism gives a splitting of $T_G(0)_A\into T_{E/A}(0)\onto
T_B(0)_A$ over~$A^{0,\infty}$. The exponential map for~$+_2$ for the
pulback to $A^{0,\infty}$ of $E\to A$ then gives an isomorphism of
formal groups over~$A^{0,\infty}$:
\[
  \begin{tikzcd}
    G^{0,\infty}\times B^{0,\infty}\times A^{0,\infty}
    \arrow[r, equals]
    & (G^{0,\infty}\times B^{0,\infty})_{A^{0,\infty}}
    \arrow[r, hook, two heads]
    & E_{A^{0,\infty}/A^{0,\infty}}^{0,\infty}
        \arrow[r, equals]
    & E^{0,\infty} \,,
  \end{tikzcd}
\]
where $E_{A^{0,\infty}/A^{0,\infty}}^{0,\infty}$ denotes the
completion along the zero section of the pullback via
$A^{0,\infty}\to A$ of $E\to A$.  The compatibility between $+_1$ and
$+_2$ on $E$ ensures that this isomorphism is an isomorphism of
biextensions, with the trivial biextension structure on the left.

Now that we know what good formal coordinates at $0$ in $E(k)$ are, we
look at the point $e$ in $E(k)$, over $(a,b)$ in $(A\times B)(k)$. We
produce an isomorphism $E^{0,\infty}\to E^{e,\infty}$, using the
partial group laws. Let $E_b$ be the fibre over~$b$ of $E\to B$.
We choose a section
\[
\begin{tikzcd}
  & E_b \arrow[d] \\
A^{a,1}\times \{b\} \arrow[r] \arrow{ru}{\tilde{e}_1} & A\times \{b\}
\end{tikzcd}\quad
\text{with $\tilde{e}_1(a,b)=e$ in~$E(k)$.}
\]
The exponentials for the group laws of $E_b$ and $A$ then give a
section 
\[
\begin{tikzcd}
  & E_b \arrow[d] \\
A^{a,\infty}\times \{b\} \arrow[r] \arrow{ru}{\tilde{e}^\infty_1} & A\times \{b\}\,,
\end{tikzcd}
\]
that we view as an $A^{a,\infty}$-valued point of~$E_b$, and as a
section of the group scheme $E_{A^{a,\infty}}\to A^{a,\infty}$, with
group law~$+_2$. The translation by $\tilde{e}^\infty_1$ on this group
scheme induces translation by~$b$ on $B_{A^{a,\infty}}$, and maps
$(a,0)$, the $0$ element of~$E_a$, to~$e$. Hence it induces an
isomorphism of formal schemes $E^{(a,0),\infty}\to E^{e,\infty}$. In
order to get an isomorphism $E^{0,\infty}\to E^{(a,0),\infty}$, we
repeat the process above, but with the roles of $A$ and $B$
exchanged. We choose a section
$\tilde{0}_2\colon \{a\}\times B^{0,1}\to E_a$ of
$E_a\to \{a\}\times B$. Then the exponential for~$+_2$ gives us a
section $\tilde{0}^\infty_2\colon \{a\}\times B^{0,\infty}\to E_a$ of
$E_a\to \{a\}\times B$. This $\tilde{0}^\infty_2$ is a section of the
group scheme $E_{B^{0,\infty}}\to B^{0,\infty}$, and the translation
on it by $\tilde{0}^\infty_2$ sends $0$ in $E(k)$ to $(a,0)$, hence
gives an isomorphism of formal schemes
$E^{0,\infty}\to E^{(a,0),\infty}$. Composition then gives us an
isomorphism $E^{0,\infty}\to E^{e,\infty}$, and the good formal
coordinates on $E$ at~$0\in E(k)$ give what we call good formal
coordinates at~$e$. Similarly, we get a section $\wt{0}_1^\infty$ of
$E_{A^{0,\infty}}\to A^{0,\infty}$ and a section $\wt{e}_2^\infty$
of $E_{B^{b,\infty}}\to B^{b,\infty}$ giving isomorphisms  
$E^{0,\infty}\to E^{(0,b),\infty}$ and $E^{(0,b),\infty}\to
E^{e,\infty}$, hence by composition a 2nd isomorphism
$E^{0,\infty}\to E^{e,\infty}$. These isomorphisms are equal for a
unique choice of $\tilde{0}_1$ and~$\tilde{e}_2$ (given the choices of
$\tilde{0}_2$ and~$\tilde{e}_1$).

In Section~\ref{sec:closure_and_formal_coords} we
will use that these isomorphisms transport all additions that occur
in~(\ref{eq:D_ttilde_n}) to additions in $E^{0,\infty}$ and therefore
to additions in the trivial formal biextension. 

\subsubsection{Zariski density of the curve in formally trivial
  coordinates}\label{sec:C_formally_Zar_dense}
Let $C$ be as in the beginning of Section~\ref{sec:alg_geometry}. Let
$\wt{C(\C)}$ be the inverse image of $C(\C)$ under the universal cover 
$\wt{T(\C)}\to T(\C)$. Then $\wt{C(\C)}$ is connected since
$\wt{j_b}\colon C\to T$ gives a surjection on complex fundamental
groups. Now we consider the complex analytic variety $\wt{T(\C)}$ as a
complex algebraic variety via the bijection $\wt{T(\C)}=\C^{g+\rho-1}$
as given in~(\ref{eq:univ_cover_T}). The analytic subset
$\wt{C(\C)}$ contains the orbit of $0$ under
$\pi_1(T(\C),1)$. This orbit surjects to the lattice of $J(\C)$ in
$\rmM_{g,1}(\C)$, and over each lattice point, its fibre in
$\rmM_{\rho-1,1}(\C)$ contains a translate of
$2\pi i\rmM_{\rho-1,1}(\Z)$. Hence this orbit is Zariski dense
in~$\C^{g+\rho-1}$. It follows that the formal completion of
$\wt{C(\C)}$ at any of its points is Zariski dense
in~$\C^{g+\rho-1}$: if a polynomial function on $\C^{g+\rho-1}$ is
zero on such a completion, then it vanishes on the connected component
of $\wt{C(\C)}$ of that point, hence on~$\wt{C(\C)}$ and consequently on~$\wt{T(\C)}$.

We express our conclusion in more algebraic terms: for $c\in C(\C)$,
with images $t\in T(\C)$ and in~$P^{\times,\rho-1}(\C)$, each
polynomial in good formal coordinates at $t$ of the biextension
$P^{\times,\rho-1}\to J\times J^\vee$ over~$\C$ that vanishes on
$\wt{j_b}(C_\C^{c,\infty})$, vanishes on~$T_\C^{t,\infty}$. This
statement then also holds with $\C$ replaced by any subfield, or even
any subring of the form $\Z_{(p)}$ with $p$ a prime number, or the
localisation of $\ol{\Z}$ (the integral closure of $\Z$ in~$\C$) at a
maximal ideal.

\subsubsection{The $p$-adic closure in good formal coordinates}
\label{sec:closure_and_formal_coords}
We stay in the situation of Section~\ref{sec:alg_geometry}, but we
denote $G:=\Gm^{\rho-1}$, $A:=J$ and $B:=J^{\vee,0\rho-1}$, and
$E:=P^{\times,\rho-1}$. Let $d_G$, $d_A$, and~$d_B$ be their
dimensions: $d_G=\rho-1$, $d_A=g$ and $d_B=(\rho-1)g$.

Let $p>2$ be a prime number.  From Section~\ref{sec:formal_triv} and
Lemma~\ref{lem:log_and_exp} we conclude that we can choose
\emph{formal} parameters for $E$ at~$0$, over~$\Z_{(p)}$, such that
they converge on the residue polydisk $E(\Z_p)_{\ol{0}}$, and such
that they induce the trivial biextension structure on
$\Z_p^{d_G}\times\Z_p^{d_A}\times\Z_p^{d_B}$. We keep the notation
of Section~\ref{sec:formal_triv}, for $e$ in $E(\Z_p)$, lying over
$(a,b)$ in $(A\times B)(\Z_p)$. This $e$ plays the role that $\wt{t}$ has
at the beginning of Section~\ref{sec:closure-finiteness}. As explained
at the end of Section~\ref{sec:formal_triv}, we may and do assume that
$e$ is in $E(\Z_p)_{\ol{0}}$, and hence $a\in A(\Z_p)_{\ol{0}}$ and
$b\in B(\Z_p)_{\ol{0}}$. 

Assume now that, as in Section~\ref{sec:closure-finiteness}, for
$i,j\in\{1,\ldots,r\}$, we have~$x_i$ in $A(\Z_p)_{\ol{0}}$ and $y_j$
in $B(\Z_p)_{\ol{0}}$, and $e_{i,j}$ in $E(\Z_p)_{\ol{0}}$ over
$(x_i,y_j)$, and $r_i$ in $E(\Z_p)_{\ol{0}}$ over $(x_i,b)$
and $s_j$ in $E(\Z_p)_{\ol{0}}$ over $(a,y_j)$. We denote the images
of all these elements under the bijection
\[
E(\Z_p)_{\ol{0}} \lto \Z_p^{d_G}\times\Z_p^{d_A}\times\Z_p^{d_B}
\]
as follows:
\begin{align*}
  x_i &\mapsto (0,x_i,0)\,,&  y_j &\mapsto (0,0,y_j)\,,&
  e_{i,j} &\mapsto (g_{i,j},x_i,y_j) \\
  r_i &\mapsto (r_i',x_i,b)\,,& s_j &\mapsto (s_j',a,y_j)\,,&
  e&\mapsto (e',a,b)\,.
\end{align*}
Then a straightforward computation shows that the image of $D(n)$ as
defined in (\ref{eq:D_ttilde_n}) is
\[
  \left(
    e'+\sum_in_ir'_i + \sum_jn_js'_j + \sum_{i,j}n_in_jg_{i,j},\,\,
    a+\sum_in_ix_i,\,\,
    b+\sum_jn_jy_j
  \right)\quad \text{in $\Z_p^{d_G}\times\Z_p^{d_A}\times\Z_p^{d_B}$.}
\]
The conclusion is that in these coordinates, the map
\[
  \kappa\colon\Z_p^r\lto \Z_p^{d_G}\times\Z_p^{d_A}\times\Z_p^{d_B}
\]
is a polynomial map, hence the Zariski closure of its image is an
algebraic variety of dimension at most~$r$. 

\subsubsection{Proof of finiteness}
The proof is by contradiction.  So assume that $r<g+\rho-1$, and that
$C(\Q)$ is infinite. Let $p>2$ be a prime number. Then there is a
$u\in C(\F_p)$ such that the residue disk $C(\Z_p)_u$ contains
infinitely many elements of~$C(\Q)$, hence infinitely many elements in
the image of~$\kappa$ of Section~\ref{thm:p-adic_closure}. By
construction, $\kappa(\Z_p^r)$ is contained in~$T(\Z_p)_t$. The image
of $T(\Z_p)_t$ in $\Z_p^{d_G}\times\Z_p^{d_A}\times\Z_p^{d_B}$ is
$\Z_p^{\rho-1}\times\Z_p^g$, with $\Z_p^g$ embedded in
$\Z_p^{d_A}\times\Z_p^{d_B}$ as a sub-$\Z_p$-module.  By the previous
section, the Zariski closure of $\kappa(\Z_p^r)$ in
$\Z_p^{d_G}\times\Z_p^{d_A}\times\Z_p^{d_B}$ is of dimension at
most~$r$. Hence there are non-zero polynomial functions on
$\Z_p^{\rho-1}\times\Z_p^g$ that are zero on infinitely many points
of~$C(\Z_p)_u$, and hence are zero on a non-empty open smaller
disk. This contradicts, via a ring morphism $\Z_p\to\C$, the
conclusion of Section~\ref{sec:C_formally_Zar_dense}.

\subsection{The relation with $p$-adic heights}\label{sec:p-adic_heights}
We want to compare the approach to quadratic Chabauty in this article
to the one in~\cite{Ba-Do-1}, by answering the question: which local
analytic coordinates on $T(\Z_p)$ and $C(\Q_p)$ lead to the equations,
in terms of $p$-adic heights, for the quadratic Chabauty set
$C(\Q_p)_2$ in~\cite{Ba-Do-1}? Before we do this, we note that the
Poincar\'e biextension has played a role in Arakelov theory, and in the
theory of $p$-adic heights, since a long time: see~\cite{Za}, \cite{M-T}
and~\cite{M-B_MP}. Moreover, \cite{Betts} gives a detailed description
how Kim's cohomological approach relates to $p$-adic heights in the
context of $\Gm$-torsors on abelian varieties. 

Let $p>2$ be a prime number of good reduction for~$C$.  We consider
the Poincar\'e torsor as $\calM^\times$ on $(J\times J)_{\Q_p}$
via~(\ref{eq:P_and_M_and_norm}), and we use the description of
$\calM^\times$ given in~(\ref{eq:M_and_N_no_base}).

Let $\mathcal D$ be the subset
$\mathrm{Div}^0(C_{\Q_p}) \times \mathrm{Div}^0(C_{\Q_p})$ made of
pairs of divisors $(D_1,D_2)$ having disjoint support. Let $W$ be an
isotropic complement of $\Omega^1_{C_{\Q_p}/\Q_p}(C_{\Q_p})$ in
$\rmH^1_{\mathrm{dR}}(C_{\Q_p}/\Q_p)$ and let
$\log\colon\Q_p^\times \to \Q_p$ be a group morphism extending the
formal logarithm on $1+p\Z_p$. With these choices made, Coleman and
Gross (\cite[(5.1)]{CoGr}) define the function (there
denoted~$\langle{\cdot},{\cdot}\rangle$)
\[
  h_p\colon \mathcal D \to \Q_p\,,
\]
the $p$-part of the $p$-adic height pairing. We define the function
\[
  \psi\colon  \calM^\times(\Q_p) \lto \Q_p
\]
by demanding that for every effective $D_1$ and $D_2$ in
$\mathrm{Div}(C_{\Q_p})$ of the same degree and every $E$ in
$\mathrm{Div}^0(C_{\Q_p})$, and every $\lambda$ in~$\Q_p^\times$, the element
\[
\lambda{\cdot}\Norm_{D_1/\Q_p}(1) \otimes\Norm_{D_2/\Q_p}(1)^{-1}
\]
in
\[
  \calM^\times(\calO_{C_{\Q_p}}(E),\Sigma(D_1)-\Sigma(D_2)) =
  \left(\Norm_{D_1/\Q_p}\calO_{C_{\Q_p}}(E) \otimes
  \Norm_{D_2/\Q_p}\calO_{C_{\Q_p}}(-E)\right)^{\times}
\]
is sent to
\[
  \psi(\lambda{\cdot}\Norm_{D_1/\Q_p}(1) \otimes\Norm_{D_2/\Q_p}(1)^{-1})
  := h_p(D_1-D_2, E)  + \log \lambda\,.
\]
That this depends only on the linear equivalence classes of
$D_1{-}D_2$ and $E$ follows from~(\ref{eq:phi_dirty}), plus
(see~\cite[Proposition 5.2]{CoGr}) the fact that $h_p$ is biadditive,
symmetric and, for any non-zero rational function $f$ on $C_{\Q_p}$
and any $D$ in $\mathrm{Div}^0(C_{\Q_p})$ with support disjoint from
that of~$\divisor(f)$, we have $h_p(D,\divisor(f)) =
\log(f(D))$. Moreover, expressing $h_p$ in terms of a Green function
$G$ as in \cite[Theorem 7.3]{Bes}, we deduce that, in each residue
disk of $\calM^\times(\Z_p)$, $\psi$ is given by a power series.  Let
$\omega_1, \ldots, \omega_g$ be a basis of
$\Omega^1_{C_{\Q_p}/\Q_p}(C_{\Q_p})$.  This basis gives a unique
morphism of groups $\log_J\colon J(\Q_p)\to\Q_p^g$ that extends the
logarithm of Lemma~\ref{lem:log_and_exp}. We define
\[
  \Psi := (\log_J\circ\pr_{J,1}, \log_J\circ\pr_{J,2}, \psi) \colon
  \calM^\times(\Q_p) \lto \Q_p^g \times \Q_p^g  \times \Q_p\,.
\]
By the biadditivity of~$h_p$, $\Psi$ is a morphism of biextensions,
with the trivial biextension structure on
$\Q_p^g \times \Q_p^g \times \Q_p$ as in~(\ref{eq:trivial_biext}). As
$p>2$, $\Psi$ induces, from each residue polydisk to its image, a
homeomorphism given by power series. Pulling back the coordinate
functions on $\Q_p^{2g+1}$ gives, for every
$x \in \calM^\times(\F_p)$, coordinates on~$\calM^\times(\Z_p)_x$.

We describe $\wt{j_b}$ and $\kappa$ in these coordinates. It suffices
to describe, for each $i$ in~$\{1,\ldots,\rho{-}1\}$,
$\wt{j_{b,i}} \colon C \to T_i$, and from now on we omit the
dependence on~$i$. For each $c \in C(\F_p)$, on
$T(\Z_p)_{\wt{j_b}(x)}$ we use the coordinates
$x_1 := f^*t_1, \ldots ,x_g:= f^*t_g$, $z:=f^*t_{2g+1}$ where $f$ is
the map $T \to \calM^\times$ and $t_1,\ldots, t_{2g+1}$ are the
coordinates on $\calM^\times(\Z_p)_{\wt{j_b}(c)}$ we just defined.
Since the map $\Psi$ is a morphism of biextensions, for $j$ in
$\{1,\ldots,g\}$, $x_j\circ\kappa$ is a polynomial of degree at
most~$1$, and $z\circ \kappa$ is a polynomial of degree at most~$2$.
As explained in Section~\ref{sec:explicit_j_b_tilde}, over~$\Z_p$,
$\wt{j_b}$ is given by a line bundle $\calL$ over $(C\times C)_{\Z_p}$
rigidified along $(C \times \{b\})_{\Z_p}$ and along the diagonal with
two sections $l_b$ and~$l$. Choosing a section that trivializes
$\calL$ on an open subset of $(C\times C)_{\Z_p}$ containing $(b,b)$,
$(c,b)$, and $(c,c)$ in $(C\times C)(\F_p)$ we get a divisor $D$ on
$(C\times C)_{\Z_p}$ whose support is disjoint from $(c,b)$
and~$(c,c)$, and an isomorphism between $\calL$ and $\calO(D)$ on
$(C \times C)_{\Z_p}$.  After modifying $D$ with a principal
horizontal divisor and a principal vertical divisor
$D|_{C {\times} \{b\}}$ and $\diag^*D$ are both equal to the the zero
divisor on~$C_{\Z_p}$, hence $l_b$ and $l$ are the extensions of
elements of~$\Q_p$, interpreted as rational sections of $\calO(D)$ on
$(C\times C)_{\Z_p}$. By Propositions~\ref{prop:M_L_D_E}
and~\ref{prop:jbtilde_precise}, there exists a unique
$\lambda \in \Q_p^\times$ such that, for each $d \in C(\Z_p)_c$,
\[
  \wt{j_b}(d) =
  \lambda \cdot \Norm_{d/\Z_p}(1) \otimes
  \Norm_{b/\Z_p}(1)^{-1}  \in \calM^\times(j_b(d), D|_{\{d\}{\times} C})\,.
\]
Since $x_j$ is the $j$-th coordinate of $\log_J$ and since $z$ is the
pullback of $\psi$, we deduce that  
\[
  x_1(\wt{j_b}(d)) = \int_b^d \omega_1,\, \ldots\,,\quad
  x_g(\wt{j_b}(d)) = \int_b^d \omega_g, \quad
  z(\wt{j_b}(d)) = h_p(d-b, D|_ {\{d\}{\times}C}) + \log \lambda \,,
\]
and, by \cite[Proof of Theorem 1.2]{Ba-Do-1} and
\cite[Lemma~5.5]{BDMTV}, the function
$d \mapsto h_p(d{-}b, D|_ {\{d\}{\times}C})$ is a sum of double
Coleman integrals.

It should now be easy to exactly interpret geometrically the
cohomological approach, showing that in the coordinates used here, the
equations for $C(\Q_p)_2$ are precisely equations for the intersection
of $C(\Q_p)$ and the $p$-adic closure of~$T(\Z)$. For doing
computations, one can do them in the geometric context of this
article, or, as in~\cite{BDMTV}, in terms of the \'etale fundamental
group of~$C$.  The connection between these is then given by $p$-adic
local systems on~$T$.

\paragraph{Author contributions}
\addcontentsline{toc}{section}{Author contributions}
This project started with an idea of Edixhoven in December 2017. From
then on Edixhoven and Lido worked together on the
project. Section~\ref{sec:example} is due entirely to
Lido. Section~\ref{sec:rems} was written in July and August 2020. 

\paragraph{Acknowledgements}
\addcontentsline{toc}{section}{Acknowledgements}
We thank Steffen M\"uller, Netan Dogra, Jennifer Balakrishnan, Kamal
Khuri-Makdisi, Jan Vonk, Barry Mazur and Gerd Faltings for discussions
and correspondence we had with them, Michael Stoll for suggesting the
title of this article, Pim Spelier for his MSc thesis~\cite{Spelier},
and Sachi Hashimoto for the cartoon guide~\cite{Ha}. Finally, we thank
the referee for their very thorough work and for making us increase
the number of propositions and lemmas so that, hopefully, this article
has become easier to parse and to digest.

\end{document}